\documentclass[11pt]{article}
\usepackage{amsmath,amsthm,amssymb}
\usepackage{mysty}

\renewcommand{\labelenumi}{(\theenumi)}
\theoremstyle{mythm}
\newtheorem{thm}{Theorem}[section]
\newtheorem{defn}[thm]{Definition}
\newtheorem{prop}[thm]{Proposition}
\newtheorem{lem}[thm]{Lemma}
\newtheorem{cor}[thm]{Corollary}
\theoremstyle{myrem}
\newtheorem{rem}[thm]{Remark}

\DeclareMathOperator{\Ind}{Ind}
\DeclareMathOperator{\Hom}{Hom}
\DeclareMathOperator{\re}{Re}
\DeclareMathOperator{\im}{Im}
\DeclareMathOperator{\supp}{supp}
\DeclareMathOperator{\Ad}{Ad}
\DeclareMathOperator{\ad}{ad}
\DeclareMathOperator{\wt}{wt}
\DeclareMathOperator{\Tr}{Tr}

\DeclareMathOperator{\Ker}{Ker}

\DeclareMathOperator{\Wh}{Wh}
\DeclareMathOperator{\ch}{ch}
\DeclareMathOperator{\Lie}{Lie}
\DeclareMathOperator{\Res}{Res}

\newcommand{\C}{\mathbb{C}}
\newcommand{\Z}{\mathbb{Z}}
\newcommand{\R}{\mathbb{R}}

\newcommand{\hyphen}{\mathchar`-}

\title{Generalized Jacquet modules of parabolic induction}
\author{Noriyuki Abe}
\address{Graduate School of Mathematical Sciences, the University of Tokyo, 3--8--1 Komaba, Meguro-ku, Tokyo 153--8914, Japan.}
\email{abenori@ms.u-tokyo.ac.jp}
\date{}
\begin{document}
\maketitle
\begin{abstract}
In this paper we study a generalization of the Jacquet module of a parabolic induction and construct a filtration on it.
The successive quotient of the filtration is written by using the twisting functor.
\end{abstract}

\section{Introduction}
The Jacquet module of a representation of a semisimple (or reductive) Lie group is introduced by Casselman~\cite{MR562655}.
One of the motivation of considering the Jacquet module is to investigate homomorphisms to principal series representations, which is an important invariant of a representation.

One of the powerful tools to study the Jacquet module of a parabolic induction is the Bruhat filtration~\cite{MR1767896}.
This is a filtration on the Jacquet module defined from the Bruhat decomposition.
Casselman-Hecht-Milicic~\cite{MR1767896} use the Bruhat filtration to determine the dimension of the (moderate-growth) Whittaker model of a principal series representation (another proof for Kostant's result).
In this paper, we study the Bruhat filtration and show that the successive quotient is described using the twisting functor defined by Arkhipov~\cite{MR2074588}.
If a principal series representation has the unique Langlands quotient, then the successive quotient is the induction from the Jacquet module of a smaller group (a Levi part of a parabolic subgroup).
However, in a general case, it becomes ``twisted'' induction, which has the same character as that of an induced representation but has a different module structure.

Moreover, we investigate its generalization, this is related to the Whittaker model.
In \cite{MR562655}, Casselman suggested to generalize the notion of the Jacquet module.
For this generalized Jacquet module, we can also define a Bruhat filtration and the successive quotient of the resulting filtration is described in terms of the generalized twisting functor.

This result gives a strategy to determine all Whittaker models of a parabolic induction.
To determine it, it is sufficient to study the successive quotients and extensions of the filtration.
In a special case, we can carry out these steps.

Now we state our results precisely.
Let $G$ be a connected semisimple Lie group, $G = KA_0N_0$ an Iwasawa decomposition and $P_0 = M_0A_0N_0$ a minimal parabolic subgroup and its Langlands decomposition.
As usual, the complexification of the Lie algebra is denoted by the corresponding German letter (for example, $\mathfrak{g} = \Lie(G)\otimes_\R\C$).
Fix a character $\eta$ of $N_0$.
Then for a representation $V$ of $G$, the generalized Jacquet modules $J'_\eta(V)$ and $J^*_\eta(V)$ are defined as follows.

\begin{defn}\label{defn:Jacquet modules in Introduction}
Let $V$ be a finite-length moderate growth Fr\'echet representation of $G$ (See Casselman~\cite[pp.~391]{MR1013462}).
We define $\mathfrak{g}$-modules $J'_\eta(V)$ and $J^*_\eta(V)$ by
\begin{align*}
J'_\eta(V) & = \left\{v\in V'\Bigm|
\begin{array}{l}
\text{For some $k$ and for all $X\in\mathfrak{n}_0$,}\\
\text{$(X - \eta(X))^kv = 0$}
\end{array}
\right\},\\
J^*_\eta(V) & = \left\{v\in (V_{\textnormal{$K$-finite}})^*\Bigm|
\begin{array}{l}
\text{For some $k$ and for all $X\in\mathfrak{n}_0$,}\\
\text{$(X - \eta(X))^kv = 0$}
\end{array}
\right\},
\end{align*}
where $V'$ is the continuous dual of $V$.
\end{defn}

Let $W$ be the little Weyl group of $G$ and take $w\in W$.
Then the generalized twisting functor $T_{w,\eta}$ is defined as follows.
Let $\overline{\mathfrak{n}_0}$ be the nilradical of the opposite parabolic subalgebra to $\mathfrak{p}_0$ and $e_1,\dots,e_l$ be a basis of $\Ad(w)\overline{\mathfrak{n}}_0\cap \mathfrak{n}_0$ such that each $e_i$ is a root vector with respect to $\mathfrak{h}$ where $\mathfrak{h}$ is a Cartan subalgebra of $\mathfrak{g}$ which contains $\mathfrak{a}_0$.
Moreover, we choose $e_i$ such that $\bigoplus_{i\le j - 1}\C e_i$ is an ideal of $\bigoplus_{i\le j}\C e_i$ for all $j$.
Let $U(\mathfrak{g})$ be the universal enveloping algebra of $\mathfrak{g}$ and $U(\mathfrak{g})_{e_i - \eta(e_i)}$ the localization of $U(\mathfrak{g})$ by a multiplicative set $\{(e_i - \eta(e_i))^n\mid n\in\Z_{>0}\}$.
Put $S_{w,\eta} = (U(\mathfrak{g})_{e_i - \eta(e_i)}/U(\mathfrak{g}))\otimes_{U(\mathfrak{g})}\dotsb\otimes_{U(\mathfrak{g})}(U(\mathfrak{g})_{e_l - \eta(e_l)}/U(\mathfrak{g}))$.
Then $S_{w,\eta}$ is a $\mathfrak{g}$-bimodule.
The twisting functor $T_{w,\eta}$ is defined by $T_{w,\eta}V = S_{w,\eta}\otimes_{U(\mathfrak{g})}(wV)$ where $wV$ is a representation twisted by $w$ (i.e., $Xv = \Ad(w)^{-1}(X)\cdot v$ for $X\in \mathfrak{g}$ and $v\in wV$ where dot means the original action).

Let $P$ be a parabolic subgroup containing $A_0N_0$ and take a Langlands decomposition $P = MAN$ such that $A_0\supset A$.
Define $\rho_0\in\mathfrak{a}_0^*$ by $\rho_0(H) = (1/2)\Tr \ad(H)|_{\mathfrak{n}_0}$.
Let $\rho$ be a restriction of $\rho_0$ on $\mathfrak{a}$.
An element of $\mathfrak{a}^*$ corresponds to a character of $A$.
We denote the corresponding character to $\lambda + \rho$ by $e^{\lambda + \rho}$ for $\lambda\in\mathfrak{a}^*$.
Then for an irreducible representation $\sigma$ of $M$ and $\lambda\in\mathfrak{a}^*$, the parabolic induction $\Ind_P^G(\sigma\otimes e^{\lambda + \rho})$ is defined.
Let $W_M$ be the little Weyl group of $M$.
Define a subset $W(M)$ of $W$ by $W(M) = \{w\in W\mid \text{for all positive restricted root $\alpha$ of $M$, $w(\alpha)$ is positive}\}$.
Then $W(M)$ is a complete representatives of $W/W_M$ and parameterizes $N_0$-orbits in $G/P$.
For $w\in W$, fix a lift in $G$ and denote it by the same letter $w$.
Enumerate $W(M) = \{w_1,\dots,w_r\}$ so that $\bigcup_{j\le i}N_0w_jP/P$ is a closed subset of $G/P$.
Using the $C^\infty$-realization of a parabolic induction, we can regard an element of $J'_\eta(\Ind_P^G(\sigma\otimes e^{\lambda + \rho}))$ as a distribution on $G/P$.
Then the Bruhat filtration $I_i\subset J'_\eta(\Ind_P^G(\sigma\otimes e^{\lambda + \rho}))$ is defined by 
\[
	I_i = \left\{x\in J'_\eta(\Ind_P^G(\sigma\otimes e^{\lambda + \rho}))\biggm| \supp x \subset \bigcup_{j\le i}N_0w_jP\right\}.
\]
Since $w_i\in W(M)$, we have $\Ad(w_i)(\mathfrak{m}\cap \mathfrak{n}_0)\subset \mathfrak{n}_0$.
Hence we can define a character $w_i^{-1}\eta$ of $\mathfrak{m}\cap \mathfrak{n}_0$ by $(w_i^{-1}\eta)(X) = \eta(\Ad(w_i)X)$.
Using this character, we can define an $\mathfrak{m}\oplus\mathfrak{a}$-module $J'_{w_i^{-1}\eta}(\sigma\otimes e^{\lambda + \rho})$.
Then we have the following theorem.
\begin{thm}[Theorem~\ref{thm:succ quot is I'_i}, Theorem~\ref{thm:structure of I_i/I_{i - 1}}]\label{thm:main theorem}
The filtration $\{I_i\}$ has the following properties.
\begin{enumerate}
\item If the character $\eta$ is not unitary, then $J'_\eta(\Ind_P^G(\sigma\otimes e^{\lambda+\rho})) = 0$.
\item Assume that $\eta$ is unitary. The module $I_i/I_{i - 1}$ is nonzero if and only if $\eta$ is trivial on $w_iNw_i^{-1}\cap N_0$ and $J'_{w_i^{-1}\eta}(\sigma\otimes e^{\lambda+\rho})\ne 0$.
\item If $I_i/I_{i - 1} \ne 0$ then $I_i/I_{i - 1} \simeq T_{w_i,\eta}(U(\mathfrak{g})\otimes_{U(\mathfrak{p})}J'_{w_i^{-1}\eta}(\sigma\otimes e^{\lambda+\rho}))$ where $\mathfrak{n}$ acts $J'_{w_i^{-1}\eta}(\sigma\otimes e^{\lambda+\rho})$ trivially.
\end{enumerate}
\end{thm}

Under the assumptions that $P$ is a minimal parabolic subgroup, $\sigma$ is the trivial representation, $\Ind_P^G(\sigma\otimes e^{\lambda + \rho})$ has the unique Langlands quotient and $\eta$ is the trivial representation, this theorem is proved in \cite{abe-2006}.
The proof we give in \cite{abe-2006} is algebraic, while we give an analytic and geometric proof in this paper.

For a module $J^*_\eta(\Ind_P^G(\sigma\otimes e^{\lambda + \rho}))$, we have the following theorem.
We define two functors.
For a $U(\mathfrak{g})$-module $V$, put $C(V) = ((V^*)_{\text{$\mathfrak{h}$-finite}})^*$ and $\Gamma_\eta(V) = \{v\in V\mid \text{for some $k$ and for all $X\in\mathfrak{n}_0$, $(X - \eta(X))^kv = 0$}\}$.
\begin{thm}[Theorem~\ref{thm:stucture of J^*(I(sigma,lambda))}]\label{thm:main theorem2}
There exists a filtration $0 = \widetilde{I_0}\subset \widetilde{I_1}\subset\cdots\subset\widetilde{I_r} = J^*_\eta(\Ind_P^G(\sigma\otimes\lambda))$ such that $\widetilde{I_i}/\widetilde{I_{i - 1}}\simeq \Gamma_\eta(C(T_{w_i}(U(\mathfrak{g})\otimes_{U(\mathfrak{p})}J^*(\sigma\otimes e^{\lambda+\rho}))))$ where $\mathfrak{n}$ acts $J^*(\sigma\otimes e^{\lambda+\rho})$ trivially.
\end{thm}

We state an application.
The space of Whittaker vectors $\Wh_\eta(D)$ is defined by $\Wh_\eta(D) = \{x\in D\mid \text{$(X - \eta(X))x = 0$ for all $X\in \mathfrak{n}_0$}\}$ for a $U(\mathfrak{g})$-module $D$.
If $V$ is a moderate-growth Fr\'echet representation of $G$, an element of $\Wh_\eta(V')$ corresponds to a moderate-growth homomorphism $V\to \Ind_{N_0}^G\eta$ and an element of $\Wh_\eta((V_{\text{$K$-finite}})^*)$ corresponds to an algebraic homomorphism $V_{\text{$K$-finite}}\to \Ind_{N_0}^G\eta$.
In particular, when $\eta$ is the trivial representation, these correspond to homomorphisms to principal series representations.

Let $\Sigma$ (resp.\ $\Sigma_M$) be the restricted root system for $(G,A_0)$ (resp.\ $(M,M\cap A_0)$), $\Sigma^+$ a positive system of $\Sigma$ corresponding to $N_0$ and $\Pi\subset\Sigma$ the set of simple roots determined by $\Sigma^+$.
Put $\Sigma_M^+ = \Sigma_M\cap \Sigma^+$.
Let $\widetilde{W}$ (resp.\ $\widetilde{W_M}$) be the (complex) Weyl group of $\mathfrak{g}$ (resp.\ $\mathfrak{m}$).
Let $\widetilde{\mu}\in\mathfrak{h}^*$ be the infinitesimal character of $\sigma$.
Let $\Delta$ be the root system of $(\mathfrak{g},\mathfrak{h})$.
Put $\Sigma_\eta^+ = (\sum_{\eta|_{\mathfrak{g}_\beta} \ne 0,\ \beta\in\Pi}\Z\beta)\cap \Sigma^+$.
Fix a $W$-invariant bilinear form $\langle\cdot,\cdot\rangle$ of $\mathfrak{a}_0$.
Using the direct decompositions $(\mathfrak{m}\cap\mathfrak{a}_0)^*\oplus \mathfrak{a}^* = \mathfrak{a}_0^*$ and $\mathfrak{a}_0^*\oplus(\mathfrak{h}\cap \mathfrak{m}_0)^* = \mathfrak{h}^*$, we regard $\mathfrak{a}^*\subset \mathfrak{a}_0^*\subset \mathfrak{h}^*$.
Recall that $\nu\in(\mathfrak{m}\cap\mathfrak{a}_0)^*$ is called an exponent of $\sigma$ if $\nu + \rho_0|_{\mathfrak{m}\cap\mathfrak{a}_0}$ is an $(\mathfrak{m}\cap \mathfrak{a}_0)$-weight of $\sigma/(\mathfrak{m}\cap \mathfrak{n}_0)\sigma$.
We prove the following theorem.
\begin{thm}[Theorem~\ref{thm:dimension Whittaker vectors}, Theorem~\ref{thm:dimension Whittaker vectors, algebraic}]\label{thm:main theorem3, dimension of the Whittaker vectors}
For $\lambda\in\mathfrak{a}^*$ and an irreducible representation $\sigma$ of $M$, the following formulae hold.
\begin{enumerate}
\item Assume that for all $w\in W$ such that $\eta|_{wNw^{-1}\cap N_0} = 1$, the following two conditions hold:
\begin{enumerate}
\item For each exponent $\nu$ of $\sigma$ and $\alpha\in \Sigma^+\setminus w^{-1}(\Sigma^+_M\cup\Sigma_\eta^+)$, we have $2\langle\alpha,\lambda+\nu\rangle/\lvert\alpha\rvert^2\not\in\Z_{\le 0}$.
\item For all $\widetilde{w}\in\widetilde{W}$, we have $\lambda - \widetilde{w}(\lambda + \widetilde{\mu})|_\mathfrak{a}\notin \Z_{\le 0}((\Sigma^+\setminus \Sigma_M^+)\cap w^{-1}\Sigma^+)|_\mathfrak{a}\setminus\{0\}$.
\end{enumerate}
Then we have
\begin{multline*}
	\dim\Wh_\eta((\Ind_P^G(\sigma\otimes e^{\lambda+\rho}))')\\ = \sum_{w\in W(M),\ \eta|_{wNw^{-1}\cap N_0} = 1}\dim \Wh_{w^{-1}\eta}(\sigma').
\end{multline*}
\item 
Assume that for all $\widetilde{w}\in\widetilde{W}\setminus\widetilde{W_M}$ we have $(\lambda + \widetilde{\mu}) - \widetilde{w}(\lambda + \widetilde{\mu}) \not\in\Z\Delta$.
Then we have
\begin{multline*}
	\dim\Wh_\eta((\Ind_P^G(\sigma\otimes e^{\lambda+\rho})_{\text{\normalfont $K$-finite}})^*)\\ = \sum_{w\in W(M)}\dim \Wh_{w^{-1}\eta}((\sigma_{\text{\normalfont $M\cap K$-finite}})^*).
\end{multline*}
\end{enumerate}
\end{thm}
In the case that $\sigma$ is finite-dimensional, we have the following theorem, which have been announced by T. Oshima (cf.\ his talk at National University of Singapore, January 11, 2006).
Let $\Delta_M$ be the root system for $(\mathfrak{m}\oplus\mathfrak{a},\mathfrak{h})$ and take a positive system $\Delta_M^+$ compatible with $\Sigma^+_M$.
Put $\widetilde{\rho_M} = (1/2)\sum_{\alpha\in\Delta_M^+}\alpha$.
For subsets $\Theta_1,\Theta_2$ of $\Pi$, put $\Sigma_{\Theta_i} = \Z\Theta_i\cap \Sigma$, $W(\Theta_i) = \{w\in W\mid w(\Theta_i)\subset \Sigma^+\}$, $W_{\Theta_i}$ the Weyl group of $\Sigma_{\Theta_i}$ and $W(\Theta_1,\Theta_2) = \{w\in W(\Theta_1)\cap W(\Theta_2)^{-1}\mid w(\Sigma_{\Theta_1})\cap \Sigma_{\Theta_2} = \emptyset\}$.
The parabolic subgroup $P$ defines a subset of $\Pi$.
We denote this set by $\Theta$.
\begin{thm}\label{thm:main theorem4, dimension of the Whittaker vectors, finite-dimensional case}
Assume that $\sigma$ is an irreducible finite-dimensional representation with highest weight $\widetilde{\nu}$.
Let $\dim_M(\lambda+\widetilde{\nu})$ be the dimension of a finite-dimensional irreducible representation of $M_0A_0$ with highest weight $\lambda+\widetilde{\nu}$.
\begin{enumerate}
\item Let $\widetilde{\nu}$ be the highest weight of $\sigma$.
Assume that for all $w\in W$ such that $\eta|_{wN_0w^{-1}\cap N_0} = 1$ the following two conditions hold:
\begin{enumerate}
\item For all $\alpha\in \Sigma^+\setminus w^{-1}(\Sigma^+_M\cup\Sigma_\eta^+)$ we have $2\langle\alpha,\lambda+w_0\widetilde{\nu}\rangle/\lvert\alpha\rvert^2\not\in\Z_{\le 0}$.
\item For all $\widetilde{w}\in\widetilde{W}$ we have $\lambda - \widetilde{w}(\lambda + \widetilde{\nu} + \widetilde{\rho_M})|_\mathfrak{a}\notin \Z_{\le 0}((\Sigma^+\setminus \Sigma_M^+)\cap w^{-1}\Sigma^+)|_\mathfrak{a}\setminus\{0\}$.
\end{enumerate}
Then we have
\[
	\dim \Wh_\eta(I(\sigma,\lambda)') = \# W(\supp\eta,\Theta)\times(\dim_M(\lambda+\widetilde{\nu}))
\]
\item Assume that for all $\widetilde{w}\in\widetilde{W}\setminus\widetilde{W_M}$, $(\lambda+\widetilde{\nu}) - \widetilde{w}(\lambda+\widetilde{\nu}) \not\in\Delta$. 
Then we have
\begin{multline*}
	\dim\Wh_\eta((I(\sigma,\lambda)_{\text{\normalfont $K$-finite}})^*) \\= \# W(\supp\eta,\Theta)\times \#W_{\supp\eta}\times(\dim_M(\lambda+\widetilde{\nu}))
\end{multline*}
\end{enumerate}
\end{thm}

We summarize the content of this paper.
In \S\ref{sec:Parabolic induction and Bruhat filtration}, we introduce the Bruhat filtration.
From \S\ref{sec:Parabolic induction and Bruhat filtration} to \S\ref{sec:The module I_i/I_i-1} we study the module $J'_\eta(\Ind_P^G(\sigma\otimes\lambda))$.
In \S\ref{sec:vanishing theorem} we prove the successive quotient is zero under some conditions.
The structure of the successive quotients is investigated in \S\ref{sec:Analytic continuation}.
We give the definition and properties of the generalized twisting functor in \S\ref{sec:Twisting functors} and, in \S\ref{sec:The module I_i/I_i-1} we reveal the relation between the twisting functor and the successive quotient.
We complete the proof of Theorem~\ref{thm:main theorem} in this section.
Theorem~\ref{thm:main theorem2} is proved in \S\ref{sec:the module J^*_eta(I(sigma,lambda))}.
In \S\ref{sec:Whittaker vectors}, the dimension of the space of Whittaker vectors is determined and Theorem~\ref{thm:main theorem3, dimension of the Whittaker vectors} and Theorem~\ref{thm:main theorem4, dimension of the Whittaker vectors, finite-dimensional case} are proved.

\subsection*{Acknowledgments}
The author is grateful to his advisor Hisayosi Matumoto for his advice and support.
He is supported by the Japan Society for the Promotion of Science Research Fellowships for Young Scientists.

\subsection*{List of Symbols}
\listofsymbols

\subsection*{Notation}
Throughout this paper we use the following notation.
As usual we denote the ring of integers, the set of non-negative integers, the set of positive integers, the real number field and the complex number field by $\Z,\Z_{\ge 0},\Z_{> 0},\R$ and $\C$, respectively.
Let $G$ be a connected semisimple Lie group and $\mathfrak{g}$ the complexification of its Lie algebra.
Fix a Cartan involution $\theta$ of $G$ and denote its derivation by the same letter $\theta$.
Let $\mathfrak{g} = \mathfrak{k}\oplus \mathfrak{s}$ be the decomposition of $\mathfrak{g}$ into the $+1$ and $-1$ eigenspaces for $\theta$.
Set $K = \{g\in G\mid \theta(g) = g\}$.
Let $P_0 = M_0A_0N_0$ be a minimal parabolic subgroup and its Langlands decomposition such that $M_0\subset K$ and $\Lie(A_0)\subset \mathfrak{s}$.
Denote the complexification of the Lie algebra of $P_0,M_0,A_0,N_0$ by $\mathfrak{p}_0,\mathfrak{m}_0,\mathfrak{a}_0,\mathfrak{n}_0$, respectively.
Take a parabolic subgroup $P$ which contains $P_0$ and denote its Langlands decomposition by $P = MAN$.
Here we assume $A\subset A_0$.
Let $\mathfrak{p},\mathfrak{m},\mathfrak{a},\mathfrak{n}$ be the complexification of the Lie algebra of $P,M,A,N$.
Put $\overline{P_0} = \theta(P_0)$, $\overline{N_0} = \theta(N_0)$, $\overline{P} = \theta(P)$, $\overline{N} = \theta(N)$, $\overline{\mathfrak{p}_0} = \theta(\mathfrak{p}_0)$, $\overline{\mathfrak{n}_0} = \theta(\mathfrak{n}_0)$, $\overline{\mathfrak{p}} = \theta(\mathfrak{p})$ and $\overline{\mathfrak{n}} = \theta(\mathfrak{n})$.

In general, we denote the dual space $\Hom_\C(V,\C)$ of a $\C$-vector space $V$ by $V^*$.
Let $\Sigma\subset\mathfrak{a}_0^*$ be the restricted root system for $(\mathfrak{g},\mathfrak{a}_0)$ and $\mathfrak{g}_\alpha$ the root space for $\alpha\in\Sigma$.
Then $\sum_{\alpha\in\Sigma}\R\alpha$ is a real form of $\mathfrak{a}_0^*$.
We denote the real part of $\lambda\in\mathfrak{a}_0^*$ with respect to this real form by $\re\lambda$ and the imaginary part by $\im\lambda$.
Let $\Sigma^+$ be the positive system determined by $\mathfrak{n}_0$.
Put $\rho_0 = \sum_{\alpha\in\Sigma^+}(\dim \mathfrak{g}_\alpha/2)\alpha$ and $\rho = \rho_0|_{\mathfrak{a}}$.
The positive system $\Sigma^+$ determines the set of simple roots $\Pi$.
Fix a totally order of $\sum_{\alpha\in\Sigma} \R\alpha$ such that the following conditions hold: (1) If $\alpha > \beta$ and $\gamma \in\sum_{\alpha\in\Sigma} \R\alpha$ then $\alpha + \gamma > \beta + \gamma$. (2) If $\alpha > 0$ and $c$ is a positive real number then $c\alpha > 0$. (3) For all $\alpha \in\Sigma^+$ we have $\alpha > 0$.
Write $W$ for the little Weyl group for $(\mathfrak{g},\mathfrak{a}_0)$, $e$ for the unit element of $W$ and $w_0$ for the longest element of $W$.
For $w\in W$, we fix a representative in $N_K(\mathfrak{a})$ and denote it by the same letter $w$.

Let $\mathfrak{t}_0$ be a Cartan subalgebra of $\mathfrak{m}_0$ and $T_0$ the corresponding Cartan subgroup of $M_0$.
Then $\mathfrak{h} = \mathfrak{t}_0\oplus\mathfrak{a}_0$ is a Cartan subalgebra of $\mathfrak{g}$.
Let $\Delta$ be the root system for $(\mathfrak{g},\mathfrak{h})$ and take a positive system $\Delta^+$ compatible with $\Sigma^+$, i.e., if $\alpha\in\Delta^+$ satisfies that $\alpha|_{\mathfrak{a}_0} \ne 0$ then $\alpha|_{\mathfrak{a}_0}\in \Sigma^+$.
Let $\mathfrak{g}^\mathfrak{h}_\alpha$ be the root space of $\alpha\in\Delta$ and $\widetilde{W}$ the Weyl group of $\Delta$.
Put $\widetilde{\rho} = (1/2)\sum_{\alpha\in\Delta^+}\alpha$.
By the decompositions $(\mathfrak{m}\cap \mathfrak{a}_0)^*\oplus\mathfrak{a}^* = \mathfrak{a}_0^*$ and $\mathfrak{t}_0^*\oplus\mathfrak{a}_0^* = \mathfrak{h}^*$, we always regard $\mathfrak{a}^*\subset\mathfrak{a}_0^*\subset\mathfrak{h}^*$.

We use the same notation for $M$, i.e., $\Sigma_M$ be the restricted root system of $M$, $\Sigma_M^+ = \Sigma_M\cap \Sigma^+$, $W_M$ the little Weyl group of $M$, $\Delta_M$ the root system of $M$, $\Delta_M^+ = \Delta_M\cap \Delta^+$, $\widetilde{W_M}$ the Weyl group of $M$ and $w_{M,0}$ the longest element of $W_M$.

We can define an anti-isomorphism of $U(\mathfrak{g})$ by $X\mapsto -X$ for $X\in \mathfrak{g}$.
We write this anti-isomorphism by $u\mapsto \check{u}$.

For a $\mathfrak{g}$-module $V$ and $g\in G$, we define a $\mathfrak{g}$-module $gV$ as follows: The representation space is $V$ and the action of $X\in\mathfrak{g}$ is $X\cdot v = (\Ad(g)^{-1}X)v$ for $v\in gV$.

For $\xi = (\xi_1,\dots,\xi_l)\in\Z^l$, put $\lvert\xi\rvert = \xi_1 + \dots + \xi_l$.

\section{Parabolic induction and the Bruhat filtration}\label{sec:Parabolic induction and Bruhat filtration}
Fix a character $\eta$ of $\mathfrak{n}_0$ and put $\supp_G \eta = \supp\eta = \{\alpha\in\Pi\mid \eta|_{\mathfrak{g}_\alpha} \ne 0\}$\newsym{$\supp_G\eta = \supp\eta$}.
The character $\eta$ is called \emph{non-degenerate} if $\supp\eta = \Pi$.
We denote the character of $N_0$ whose differential is $\eta$ by the same letter $\eta$.
\begin{defn}\label{defn:Jacquet modules}
Let $V$ be a finite-length moderate growth Fr\'echet representation of $G$ (See Casselman~\cite[pp.~391]{MR1013462}).
We define $\mathfrak{g}$-modules $J'_\eta(V)$ and $J^*_\eta(V)$ by
\begin{align*}
J'_\eta(V) & = \left\{v\in V'\Bigm|
\begin{array}{l}
\text{For some $k$ and for all $X\in\mathfrak{n}_0$,}\\
\text{$(X - \eta(X))^kv = 0$}
\end{array}
\right\},\\
J^*_\eta(V) & = \left\{v\in (V_{\textnormal{$K$-finite}})^*\Bigm|
\begin{array}{l}
\text{For some $k$ and for all $X\in\mathfrak{n}_0$,}\\
\text{$(X - \eta(X))^kv = 0$}
\end{array}
\right\},
\end{align*}
where $V'$ is the continuous dual of $V$.
\end{defn}

Put $J'(V) = J'_0(V)$ and $J^*(V) = J^*_0(V)$ where $0$ is the trivial representation of $\mathfrak{n}_0$.
The module $J^*(V)$ is the (dual of) \emph{Jacquet module} defined by Casselman~\cite{MR562655}.
By the automatic continuation theorem~\cite[Theorem~4.8]{MR727854}, we have $J'(V) = J^*(V)$.
The correspondence $V\mapsto J'_\eta(V)$ and $V\mapsto J^*_\eta(V)$ are functors from the category of $G$-modules to the category of $\mathfrak{g}$-modules.

In this section, we study the module $J'_\eta(V)$ for a parabolic induction $V$.
An element of $\mathfrak{a}^*$ is identified with a character of $A$.
We denote the character of $A$ corresponding to $\lambda + \rho$ by $e^{\lambda + \rho}$ where $\lambda\in\mathfrak{a}^*$.
For an irreducible moderate growth Fr\'echet representation $\sigma$ of $M$ and $\lambda\in\mathfrak{a}^*$, put 
\[
	I(\sigma,\lambda) = C^\infty\hyphen\Ind_P^G(\sigma\otimes e^{\lambda + \rho}).
\]
(For a moderate growth Fr\'echet representation, see Casselman~\cite{MR1013462}.)
The representation $I(\sigma,\lambda)$ has a natural structure of a moderate growth Fr\'echet representation.
\newsym{$I(\sigma,\lambda)$}
Denote its continuous dual by $I(\sigma,\lambda)'$.

Let $\mathcal{L}$ be a vector bundle on $G/P$ attached to the representation $\sigma\otimes e^{\lambda+\rho}$ and $\mathcal{L}'$ be the continuous dual vector bundle of $\mathcal{L}$. \newsym{$\mathcal{L}$}
\begin{rem}\label{rem:identify func on flag and G}
A $C^\infty$-section of $\mathcal{L}$ corresponds to a  $\sigma$-valued $C^\infty$-function $f$ on $G$ such that $f(gman) = \sigma(m)^{-1}e^{-(\lambda+\rho)(\log a)}f(g)$ for $g\in G$, $m\in M$, $a\in A$, $n \in N$.
In particular a $C^\infty$-function on $G/P$ corresponds to a right $P$-invariant $C^\infty$-function on $G$.
We use this identification throughout this paper.
\end{rem}

We use the notation in Appendix~\ref{sec:C^infty-function with values in Frechet space}.
We can regard $J'_\eta(I(\sigma,\lambda))$ as a subspace of $\mathcal{D}'(G/P,\mathcal{L})$.
Set $W(M) = \{w\in W\mid w(\Sigma_M^+)\subset \Sigma^+\}$.\newsym{$W(M)$}\newsym{$r$}
Then it is known that the multiplication map $W(M)\times W_M\to W$ is bijective~\cite[Proposition~5.13]{MR0142696}.
By the Bruhat decomposition, we have
\[
	G/P = \bigsqcup_{w\in W(M)}N_0wP/P.
\]
(Recall that we fix a representative of $w\in W$, see Notation.)
Enumerate $W(M) = \{w_1,\dots,w_r\}$ so that $\bigcup_{j \le i}N_0w_jP/P$ is a closed subset of $G/P$ for each $i$.
Then we can define a submodule $I_i$ of $J'_\eta(I(\sigma,\lambda))$ by
\[
	I_i = \left\{x\in J'_\eta(I(\sigma,\lambda))\Biggm| \supp x\subset \bigcup_{j \le i}N_0w_jP/P\right\}.\newsym{$I_i$}
\]
The filtration $\{I_i\}$ is called the Bruhat filtration~\cite{MR1767896}.
In the rest of this section, we study the module $I_i/I_{i - 1}$.
Put $U_i = w_i\overline{N}P/P$ and $O_i = N_0w_iP/P$.
The subset $U_i$ is an open subset of $G/P$ containing $O_i$ and $U_i\cap O_j = \emptyset$ if $j < i$.\newsym{$U_i$}\newsym{$O_i$}
Hence, the restriction map $\Res_i\colon I_i \to \mathcal{D}(U_i,\mathcal{L})$ induces an injective map $\Res_i\colon I_i/I_{i - 1}\to \mathcal{D}(U_i,\mathcal{L})$.\newsym{$\Res_i$}
Moreover, $\im\Res_i \subset \mathcal{T}_{O_i}(U_i,\mathcal{L})$.
We have $\mathcal{T}_{O_i}(U_i,\mathcal{L}) = U(\Ad(w_i)\overline{\mathfrak{n}}\cap\overline{\mathfrak{n}})\otimes_\C\mathcal{T}(O_i,\mathcal{L}|_{O_i})$ by Proposition~\ref{prop:structure theorem of tempered distributions whose support is contained in submanifold}.

Notice that by a map $n\mapsto nw_iP/P$ we have isomorphisms $w_i\overline{N}w_i^{-1}\simeq U_i$ and $w_i\overline{N}w_i^{-1}\cap N_0\simeq O_i$.
Fix a Haar measure on $w_i\overline{N}w_i^{-1}\cap N_0$.
Then we can define $\delta_i\in \mathcal{D}'(O_i,\mathcal{L}|_{O_i})$ by
\[
	\langle\delta_i,\varphi\rangle = \int_{w_i\overline{N}w_i^{-1}\cap N_0}\varphi(nw_i)dn.
\]
for $\varphi\in C_c^\infty(O_i,\mathcal{L}|_{O_i})$.\newsym{$\delta_i$}
By the exponential map $\Ad(w_i)\Lie(\overline{N})\to w_i\overline{N}w_i^{-1}$ and diffeomorphism $w_i\overline{N}w_i^{-1}\simeq U_i$, $U_i$ has a vector space structure and $O_i$ is a subspace of $U_i$.
Let $\mathcal{P}(O_i)$ be the ring of polynomials on $O_i$ (cf.\ Appendix~\ref{subsec:Distributions on a nilpotent Lie group} or \cite{MR1070979}).\newsym{$\mathcal{P}(O_i)$}
Define a $C^\infty$-function $\eta_i$ on $O_i$ by $\eta_i(nw_iP/P) = \eta(n)$ for $n\in w_i\overline{N}w_i^{-1}\cap N_0$.\newsym{$\eta_i$}
For a $C^\infty$-function $f$ on $O_i$ and $u'\in \sigma'$, we define $f\otimes u'\in C^\infty(O_i,\sigma')$ by $(f\otimes u')(x) = f(x)u'$.
Since $w_i\in W(M)$, $\Ad(w_i)(\mathfrak{m}\cap \mathfrak{n}_0) \subset \mathfrak{n}_0$.
Hence we can define a character $w_i^{-1}\eta$ of $\mathfrak{m}\cap\mathfrak{n}_0$ by $(w_i^{-1}\eta)(X) = \eta(\Ad(w_i)X)$.
Using this character, we can define the Jacquet module $J'_{w_i^{-1}\eta}(\sigma\otimes e^{\lambda + \rho})$ of $MA$-representation $\sigma\otimes e^{\lambda + \rho}$.
This is an $\mathfrak{m}\oplus\mathfrak{a}$-module.
Put
\[
	I'_i = \left\{\sum_{k = 1}^l T_k(((f_k\eta_i^{-1})\otimes u_k')\delta_i)\biggm|
\begin{array}{ll}
T_k\in U(\Ad(w_i)\overline{\mathfrak{n}}\cap \overline{\mathfrak{n}}),& f_k\in \mathcal{P}(O_i),\\
u_k'\in J'_{w_i^{-1}\eta}(\sigma\otimes e^{\lambda+\rho})
\end{array}
\right\}.
\]
The space $I'_i$ is a $U(\mathfrak{g})$-submodule of $\mathcal{D}'(U_i,\mathcal{L})$.
Our aim is to prove that if $i$ satisfies some conditions then $I_i/I_{i - 1} \simeq I'_i$.

\begin{lem}\label{lem:bracket of n and bar_n}
Let $E_1,\dots,E_n$ be a basis of $\Ad(w_i)\overline{\mathfrak{n}}\cap \overline{\mathfrak{n}_0}$ such that each $E_s$ is a restricted root vector for some root (say $\alpha_s$) and $F\in(\Ad(w_i)\overline{\mathfrak{n}}\cap \mathfrak{n}_0)\oplus\Ad(w_i)(\mathfrak{m}\cap\mathfrak{n}_0)$.
(Notice that $\Ad(w_i)(\mathfrak{m}\cap \mathfrak{n}_0) = \Ad(w_i)\mathfrak{m}\cap \mathfrak{n}_0$ since $w_i\in W(M)$, so $(\Ad(w_i)\overline{\mathfrak{n}}\cap \mathfrak{n}_0)\oplus\Ad(w_i)(\mathfrak{m}\cap\mathfrak{n}_0) = \Ad(w_i)(\overline{\mathfrak{n}}\oplus\mathfrak{m})\cap \mathfrak{n}_0$ is a subalgebra of $\mathfrak{g}$.)
For $\xi = (\xi_1,\xi_2,\dots,\xi_n)\in\Z_{\ge 0}^n$, set $E^\xi = E_1^{\xi_1}E_2^{\xi_2}\dotsm E_n^{\xi_n}$.
Then for all $c\in\C$ we have
\begin{multline*}
	[(F - c)^k,E^\xi] \in \left(\sum_{\xi'\in A(\xi)}\C E^{\xi'}\right) U((\Ad(w_i)\overline{\mathfrak{n}}\cap \mathfrak{n}_0)\oplus\Ad(w_i)(\mathfrak{m}\cap\mathfrak{n}_0))
	\\\subset U(\Ad(w_i)(\overline{\mathfrak{n}}\oplus(\mathfrak{m}\cap \mathfrak{n}_0)))
\end{multline*}
where $A(\xi) = \{\xi'\in\Z_{\ge 0}^n\mid \text{$\lvert\xi'\rvert < \lvert\xi\rvert$, or $\lvert\xi'\rvert = \lvert\xi\rvert$ and $\sum \xi'_i\alpha_i < \sum\xi_i\alpha_i$}\}$.\end{lem}
\begin{proof}
We may assume $k = 1$.
We prove the lemma by induction on $\lvert\xi\rvert$.
We have
\[
	[F - c,E^\xi] = [F,E^\xi] = \sum_{s = 1}^n\sum_{l = 0}^{\xi_s - 1}E_1^{\xi_1}\dotsm E_{s - 1}^{\xi_{s - 1}} E_s^l [F,E_s] E_s^{\xi_s - l - 1} E_{s + 1}^{\xi_{s + 1}}\dotsm E_n^{\xi_n}.
\]
Hence, it is sufficient to prove
\begin{multline*}
	E_1^{\xi_1}\dotsm E_{s - 1}^{\xi_{s - 1}} E_s^l [F,E_s] E_s^{\xi_s - l - 1} E_{s + 1}^{\xi_{s + 1}}\dotsm E_n^{\xi_n}\\
	 \in \left(\sum_{\xi'\in A(\xi)}\C E^{\xi'}\right) U((\Ad(w_i)\overline{\mathfrak{n}}\cap \mathfrak{n}_0)\oplus\Ad(w_i)(\mathfrak{m}\cap\mathfrak{n}_0)).
\end{multline*}
We may assume that $F$ is a restricted root vector.
If $[F,E_s]\in \Ad(w_i)\overline{\mathfrak{n}}\cap \overline{\mathfrak{n}_0}$ then the claim hold.
Assume that $[F,E_s]\in (\Ad(w_i)\overline{\mathfrak{n}}\cap \mathfrak{n}_0)\oplus\Ad(w_i)(\mathfrak{m}\cap\mathfrak{n}_0)$.
Put $\xi^{(1)} = (\xi_1,\dots,\xi_{s - 1},l,0,\dots,0)\in\Z^n$ and $\xi^{(2)} = (0,\dots,0,\xi_s - l - 1,\xi_{s + 1},\dots,\xi_n)\in\Z^n$.
Using inductive hypothesis, we have
\begin{align*}
& E^{\xi^{(1)}}\bigl[ [F,E_s],E^{\xi^{(2)}}\bigr] \\
& \in E^{\xi^{(1)}}\left(\sum_{\xi'\in A(\xi^{(2)})}\C E^{\xi'}\right)U((\Ad(w_i)\overline{\mathfrak{n}}\cap \mathfrak{n}_0)\oplus\Ad(w_i)(\mathfrak{m}\cap\mathfrak{n}_0))\\
& \subset \left(\sum_{\xi'\in A(\xi^{(1)} + \xi^{(2)})}\C E^{\xi'}\right)U((\Ad(w_i)\overline{\mathfrak{n}}\cap \mathfrak{n}_0)\oplus\Ad(w_i)(\mathfrak{m}\cap\mathfrak{n}_0))\\
& \subset \left(\sum_{\xi'\in A(\xi)}\C E^{\xi'}\right)U((\Ad(w_i)\overline{\mathfrak{n}}\cap \mathfrak{n}_0)\oplus\Ad(w_i)(\mathfrak{m}\cap\mathfrak{n}_0))
\end{align*}
On the other hand, we have
\begin{multline*}
E^{\xi^{(1)}}E^{\xi^{(2)}}[F,E_s]\in \left(\sum_{\lvert\xi'\rvert \le \lvert\xi^{(1)} +\xi^{(2)}\rvert}\C E^{\xi'}\right)[F,E_s]\\
\subset \left(\sum_{\lvert\xi'\rvert \le \lvert\xi^{(1)} +\xi^{(2)}\rvert}\C E^{\xi'}\right)(\Ad(w_i)\overline{\mathfrak{n}}\cap \mathfrak{n}_0)\oplus\Ad(w_i)(\mathfrak{m}\cap\mathfrak{n}_0).
\end{multline*}
Since $\lvert\xi^{(1)} + \xi^{(2)}\rvert = \lvert\xi\rvert - 1 < \lvert\xi\rvert$, we get the lemma.
\end{proof}

Let $X$ be an element of the normalizer of $\Ad(w_i)\overline{\mathfrak{n}}\cap\mathfrak{n}_0$ in $\mathfrak{g}$.
For $f\in C^\infty(O_i)$ we define $D_i(X)f\in C^\infty(O_i)$ by
\[
	(D_i(X)f)(nw_i) = \left.\frac{d}{dt}f(\exp(-tX)n\exp(tX)w_i)\right|_{t = 0}
\]
where $n \in w_i\overline{N}w_i^{-1}\cap N_0$.\newsym{$D_i(X)$}\label{symbol:D_i}

\begin{lem}\label{lem:no delta part}
Fix $f\in C^\infty(O_i)$, $u'\in (\sigma\otimes e^{\lambda+\rho})'$ and $X\in\mathfrak{g}$.
\begin{enumerate}
\item If $X\in \mathfrak{a}_0$, then $X$ normalizes $\Ad(w_i)\overline{\mathfrak{n}}\cap \mathfrak{n}_0$ and we have
\begin{multline*}
X((f\otimes u')\delta_i) = ((D_i(X)f)\otimes u')\delta_i + (f\otimes((\Ad(w_i)^{-1}X)u'))\delta_i\\ + (w_i\rho_0 - \rho_0)(X)(f\otimes u')\delta_i.
\end{multline*}
\item If $X\in \Ad(w_i)(\mathfrak{m}\cap \mathfrak{n}_0)$ or $X\in\mathfrak{m}_0$, then $X$ normalizes $\Ad(w_i)\overline{\mathfrak{n}}\cap \mathfrak{n}_0$ and we have 
\[
X((f\otimes u')\delta_i) = ((D_i(X)f)\otimes u')\delta_i + (((\Ad(w_i)^{-1}X)u')\otimes f) \delta_i.
\]
\end{enumerate}
\end{lem}

\begin{proof}
Let $X$ be as in the lemma.
First we prove that $X$ normalizes $\Ad(w_i)\overline{\mathfrak{n}}\cap \mathfrak{n}_0$.
If $X\in \mathfrak{m}_0 + \mathfrak{a}_0$, then $X$ normalizes each restricted root space.
Hence, $X$ normalizes $\Ad(w_i)\overline{\mathfrak{n}}\cap \mathfrak{n}_0$.
If $X\in \Ad(w_i)(\mathfrak{m}\cap \mathfrak{n}_0)$, then $X\in \mathfrak{n}_0$ since $w_i\in W(M)$.
Hence, $X$ normalizes $\mathfrak{n}_0$.
Since $\mathfrak{m}$ normalizes $\overline{\mathfrak{n}}$, $X$ normalizes $\Ad(w_i)\overline{\mathfrak{n}}$.

Put $g_t = \exp(tX)$.
Take $\varphi\in C_c^\infty(U_i,\mathcal{L})$ and we regard $\varphi$ as a $\sigma$-valued $C^\infty$-function on $w_i\overline{N}P$ (Remark~\ref{rem:identify func on flag and G}).
Since $w_ig_tw_i^{-1}\in P$, we have $\varphi(xw_ig_tw_i^{-1}) = \sigma(w_ig_tw_i^{-1})^{-1}\varphi(x)$.
Put $D(t) = \lvert\det(\Ad(g_t)^{-1}|_{\Ad(w_i)\overline{\mathfrak{n}}\cap \mathfrak{n}_0})\rvert$.
Then
\begin{align*}
& \langle X((f\otimes u')\delta_i),\varphi\rangle = \langle (f\otimes u')\delta_i ,-X\varphi\rangle\\
& = \frac{d}{dt}\left.\int_{w_i\overline{N}w_i^{-1}\cap N_0}u'(\varphi(g_tnw_i))f(nw_i)dn\right|_{t = 0}\\
& = \frac{d}{dt}\left.\int_{w_i\overline{N}w_i^{-1}\cap N_0}u'(\varphi((g_tng_t^{-1})w_i(w_i^{-1}g_tw_i)))f(nw_i)dn\right|_{t = 0}\\
& = \frac{d}{dt}\left.\int_{w_i\overline{N}w_i^{-1}\cap N_0}u'(\sigma(w_i^{-1}g_tw_i)^{-1}\varphi((g_tng_t^{-1})w_i))f(nw_i)dn\right|_{t = 0}\\
& = \frac{d}{dt}\left.\int_{w_i\overline{N}w_i^{-1}\cap N_0}u'(\sigma(w_i^{-1}g_tw_i)^{-1}\varphi(nw_i))f(g_t^{-1}ng_tw_i)D(t) dn\right|_{t = 0}\\
& = \frac{d}{dt}\left.\int_{w_i\overline{N}w_i^{-1}\cap N_0}((w_i^{-1}g_tw_i)u')(\varphi(nw_i))f(g_t^{-1}ng_tw_i)D(t) dn\right|_{t = 0}
\end{align*}
This implies
\begin{multline*}
X((f\otimes u')\delta_i)
= ((D_i(X)f)\otimes u')\delta_i + (f\otimes((\Ad(w_i)^{-1}X)u'))\delta_i\\
+ \left.\frac{d}{dt}\lvert\det(\Ad(g_t)^{-1}|_{\Ad(w_i)\overline{\mathfrak{n}}\cap \mathfrak{n}})\rvert\right|_{t = 0}((f\otimes u')\delta_i)
\end{multline*}

(1)
Assume that $X\in \mathfrak{a}_0$.
Since $w_i\in W(M)$, we have $w_i\overline{N}w_i^{-1}\cap N_0 = w_i\overline{N}_0w_i^{-1}\cap N_0$.
This implies that $\det(\Ad(g_t)^{-1}|_{\Ad(w_i)\overline{\mathfrak{n}}\cap \mathfrak{n}_0}) = e^{t(w_i\rho_0 - \rho_0)(X)}$.

(2)
First assume that $X\in\mathfrak{m}_0$.
Since $g\mapsto \det(\Ad(g)^{-1}|_{\Ad(w_i)\overline{\mathfrak{n}}\cap \mathfrak{n}_0})$ is $1$-dimensional representation, it is unitary since $M_0$ is compact.
Hence we have $\lvert\det(\Ad(g_t)^{-1}|_{\Ad(w_i)\overline{\mathfrak{n}}\cap \mathfrak{n}_0})\rvert = 1$.
Next assume that $X\in (\Ad(w_i)\mathfrak{m}\cap \mathfrak{n}_0)$.
Then $\ad(X)$ is nilpotent.
Hence, $\Ad(g_t) - 1$ is nilpotent.
This implies $\det(\Ad(g_t)^{-1}|_{\Ad(w_i)\overline{\mathfrak{n}}\cap \mathfrak{n}_0}) = 1$.
\end{proof}

\begin{lem}\label{lem:succ quot is sub of I'}
Let $x\in\mathcal{T}_{O_i}(U_i,\mathcal{L})$.
Assume that for all $X\in \Ad(w_i)\overline{\mathfrak{p}}\cap \mathfrak{n}_0$ there exists a positive integer $k$ such that $(X - \eta(X))^kx = 0$.
Then $x\in I_i'$.
In particular we have $\im\Res_i\subset I_i'$.
\end{lem}
\begin{proof}
Let $E_s$ and $\alpha_s$ be as in Lemma~\ref{lem:bracket of n and bar_n}.
For $\xi = (\xi_1,\xi_2,\dots,\xi_n)\in\Z_{\ge 0}^n$, set $E^\xi = E_1^{\xi_1}E_2^{\xi_2}\dots E_n^{\xi_n}$.
Since $x\in\mathcal{T}_{O_i}(U_i,\mathcal{L}) = U(\Ad(w_i)\overline{\mathfrak{n}}\cap \mathfrak{n}_0)\otimes \mathcal{T}(O_i,\mathcal{L})$, there exist $x_\xi\in \mathcal{T}(O_i,\mathcal{L})$ such that $x = \sum_\xi E^\xi x_\xi$ (finite sum).

First we prove $x_\xi\in (\mathcal{P}(O_i)\eta_i^{-1}\otimes(\sigma\otimes e^{\lambda+\rho})')\delta_i$ by backward induction on the lexicological order of $(\lvert\xi\rvert,\sum_s \xi_s\alpha_s)$.
Fix a nonzero element $F\in\Ad(w_i)\overline{\mathfrak{n}}\cap\mathfrak{n}_0$.
Then $(F - \eta(F))^kx = \sum_\xi [(F - \eta(F))^k,E^\xi] (x_\xi) + \sum_\xi E^\xi((F - \eta(F))^k x_\xi)$.
Assume that $(F - \eta(F))^kx = 0$.
Define the set $A(\xi)$ as in Lemma~\ref{lem:bracket of n and bar_n}.
By Lemma~\ref{lem:bracket of n and bar_n}, we have
\begin{multline*}
\sum_\xi E^\xi ((F - \eta(F))^kx_\xi) = -\sum_\xi [(F - \eta(F))^k,E^\xi](x_\xi)\\
\in \sum_\xi \left(\sum_{\xi'\in A(\xi)}\C E^{\xi'}\right) U((\Ad(w_i)\overline{\mathfrak{n}}\cap \mathfrak{n})\oplus\Ad(w_i)(\mathfrak{m}\cap\mathfrak{n}_0))(x_{\xi}).
\end{multline*}
Put $B(\xi) = \{\xi'\mid \text{$\lvert\xi'\rvert > \lvert\xi\rvert$ or $\lvert\xi'\rvert = \lvert\xi\rvert$ and $\sum \xi'_s\alpha_s > \sum \xi_s\alpha_s$} \}$.
Notice that $U((\Ad(w_i)\overline{\mathfrak{n}}\cap \mathfrak{n})\oplus\Ad(w_i)(\mathfrak{m}\cap\mathfrak{n}_0))(x_{\xi})\in \mathcal{T}(O_i,\mathcal{L})$.
Since $\mathcal{T}_{O_i}(U_i,\mathcal{L}) = U(\Ad(w_i)\overline{\mathfrak{n}}\cap\overline{\mathfrak{n}})\otimes_\C\mathcal{T}(O_i,\mathcal{L}|_{O_i})$, we have
\[
	(F - \eta(F))^kx_\xi \in \sum_{\xi'\in B(\xi)}U((\Ad(w_i)\overline{\mathfrak{n}}\cap \mathfrak{n})\oplus\Ad(w_i)(\mathfrak{m}\cap\mathfrak{n}_0))(x_{\xi'}).
\]
By inductive hypothesis, $x_{\xi'}\in (\mathcal{P}(O_i)\eta_i^{-1}\otimes (\sigma\otimes e^{\lambda+\rho})')\delta_i$ for all $\xi'\in B(\xi)$.
Hence we have $(F - \eta(F))^kx_{\xi}\in (\mathcal{P}(O_i)\eta_i^{-1}\otimes (\sigma\otimes e^{\lambda+\rho})')\delta_i$.
Therefore $x_{\xi}\in(\mathcal{P}(O_i)\eta_i^{-1}\otimes(\sigma\otimes e^{\lambda+\rho})')\delta_i$ by Corollary~\ref{cor:polynomial by some power of n}.
%Hence, there exists a positive integer $k'$ such that $(F - \eta(F))^{k'}x_\xi = 0$.
%Hence, by Lemma~\ref{lem:n-nilpotent function}, we have $x_\xi\in \mathcal{P}(O_i)\otimes(\sigma\otimes e^{\lambda+\rho})'$.

Hence, we can write $x = \sum_{\xi}E^\xi \sum_l (f_{\xi,l}\eta_i^{-1}\otimes u_{\xi,l}')\delta_i$ (finite sum), where $f_{\xi,l}\in \mathcal{P}(O_i)$ and $u_{\xi,l}'\in (\sigma\otimes e^{\lambda+\rho})'$.
Moreover, we can assume that $f_{\xi,l}$ is an $\mathfrak{a}_0$-weight vector with respect to $D_i$ and $\{f_{\xi,l}\}_l$ is lineally independent for each $\xi$.
We prove $u_{\xi,l}'\in J'_{w_i^{-1}\eta}(\sigma\otimes e^{\lambda+\rho})$.
Take $F\in \mathfrak{n}_0\cap \mathfrak{m}$.
By Lemma~\ref{lem:no delta part}, we have
\begin{multline*}
(\Ad(w_i)F - \eta(\Ad(w_i)F))^kx \\
= \sum_{\xi,l}[(\Ad(w_i)F - \eta(\Ad(w_i)F))^k,E^\xi]((f_{\xi,l}\eta_i^{-1}\otimes u_{xi,l}')\delta_i)\\
+ \sum_{\xi}E^\xi\sum_{p = 0}^k\binom{k}{p}(((D_i(\Ad(w_i)F))^{k - p}(f_{\xi,l})\eta_i^{-1})\otimes\\(F - \eta(\Ad(w_i)F))^p(u_{\xi,l}'))\delta_i.
\end{multline*}
Now we prove $u_{\xi,l}'\in J'_{w_i^{-1}\eta}(\sigma\otimes e^{\lambda+\rho})$ by backward induction on the lexicological order of $(\lvert\xi\rvert,\sum \xi_s\alpha_s,-\wt f_{\xi,l})$ where $\wt f_{\xi,l}$ is an $\mathfrak{a}_0$-weight of $f_{\xi,l}$ with respect to $D_i$.
Take $k$ such that $(\Ad(w_i)F - \eta(\Ad(w_i)F))^k x = 0$.
Then we have
\begin{multline*}
f_{\xi,l}\otimes (F - \eta(\Ad(w_i)F))^k(u'_{\xi,l})\delta_i\\ \in
\sum_{\eta\in B(\xi),l}U((\Ad(w_i)\overline{\mathfrak{n}}\cap \mathfrak{n}_0)\oplus\Ad(w_i)(\mathfrak{m}\cap\mathfrak{n}_0))((f_{\eta,l}\eta_i^{-1}\otimes u_{\eta,l}')\delta_i)\\
+ \sum_{\wt f_{\eta,l'} < \wt f_{\xi,l}}\sum_p(((D_i(\Ad(w_i)F))^pf_{\eta,l'}\eta_i^{-1})\otimes(U(\C F)u_{\eta,l'}'))\delta_i.
\end{multline*}
By inductive hypothesis, we have $(F - \eta(F))^ku_{\xi,l}'\in J'_{w_i^{-1}\eta}(\sigma\otimes e^{\lambda+\rho})$.
This implies that $u_{\xi,l}'\in J'_{w_i^{-1}\eta}(\sigma\otimes e^{\lambda+\rho})$.
\end{proof}

In fact, we have $\im\Res_i = I'_i$ under some conditions.
This is proved in Section~\ref{sec:Analytic continuation}.

\section{Vanishing theorem}\label{sec:vanishing theorem}
In this section, we fix $i \in \{1,2,\dots,r\}$ and a basis $\{e_1,e_2,\dots,e_l\}$ of $\Ad(w_i)\overline{\mathfrak{n}}\cap\mathfrak{n}_0$.
Here we assume that each $e_i$ is a restricted root vector and denote its root by $\alpha_i$.

By the decomposition 
\begin{multline*}
N_0/[N_0,N_0] \simeq ((w_i\overline{P}w_i^{-1}\cap N_0)/(w_i\overline{P}w_i^{-1}\cap [N_0,N_0])) \\\times ((w_iNw_i^{-1}\cap N_0)/(w_iNw_i^{-1}\cap [N_0,N_0]))
\end{multline*}
where $[\cdot,\cdot]$ is the commutator group, we can define a character $\eta'$ of $N_0$ by $\eta'(n) = \eta(n)$ for $n\in w_i\overline{P}w_i^{-1}\cap N_0$ and $\eta'(n) = 1$ for $n\in w_iNw_i^{-1}\cap N_0$.

\begin{lem}\label{lem:acts nilp}
Let $X\in \mathfrak{n}_0$.
Then for all $x\in I_i'$ there exists a positive integer $k$ such that $(X - \eta'(X))^kx = 0$.
\end{lem}

To prove this lemma, we prepare some notation.
For $X\in \Ad(w_i)\overline{\mathfrak{n}}\cap \mathfrak{n}_0$, we define a differential operator $R_i'(X)$ on $O_i$ by
\[
	(R'_i(X)\varphi)(nw_iP/P) = \left.\frac{d}{dt}\varphi(n\exp(tX)w_iP/P)\right|_{t = 0}
\]
where $n\in w_i\overline{N}w_i^{-1}\cap N_0$.\newsym{$R'_i(X)$}
(Recall that $w_i\overline{N}w_i^{-1}\cap N_0\simeq O_i$ by the map $n\mapsto nw_iP/P$.)

For $X\in \mathfrak{g}$, we define a differential operator $\widetilde{R}_i(X)$ on $w_i\overline{N}P$ by the same way, i.e., for a $C^\infty$-function $\varphi$ on $w_i\overline{N}P$, put
\[
	(\widetilde{R}_i(X)\varphi)(pw_i) = \left.\frac{d}{dt}\varphi(p\exp(tX)w_i)\right|_{t = 0}
\]
for $p \in w_i\overline{N}Pw_i^{-1}$.\newsym{$\widetilde{R}_i(X)$}
Notice that even if $\varphi$ is right $P$-invariant, $\widetilde{R}_i(X)\varphi$ is not right $P$-invariant in general.

Since $R'_i$ (resp.\ $\widetilde{R}_i$) is a Lie algebra homomorphism, we can define a differential operator $R'_i(T)$ (resp.\ $\widetilde{R}_i(T)$) for $T\in U(\Ad(w_i)\overline{\mathfrak{n}}\cap \mathfrak{n}_0)$ (resp.\ $T\in U(\mathfrak{g})$) as usual.
For $T\in U(\mathfrak{g})$, $f\in C^\infty(O_i)$ and $u'\in (\sigma\otimes(\lambda + \rho))'$, we define $\delta_i(T,f,u')\in \mathcal{D}'_{O_i}(U_i,\mathcal{L})$\newsym{$\delta_i(T,f,u')$} by 
\[
	\langle \delta_i(T,f,u'),\varphi\rangle = 
	\int_{w_i\overline{N}w_i^{-1}\cap N_0} f(nw_i)u'((\widetilde{R}_i(T)\varphi)(nw_i))dn
\]
where $\varphi\in C^\infty_c(U_i,\mathcal{L})$ and we regard $\varphi$ as a function on $w_i\overline{N}P$ (Remark~\ref{rem:identify func on flag and G}).
The following lemma is easy to prove.
\begin{lem}\label{lem:fundamental properties of delta_i}
We have the following properties.
\begin{enumerate}
\item For $X\in \Ad(w_i)\overline{\mathfrak{n}}\cap\mathfrak{n}_0$, $\delta_i(XT,f,u') = \delta_i(T,R'_i(-X)(f),u')$.
\item For $X\in \Ad(w_i)\mathfrak{p}$, $\delta_i(TX,f,u') = \delta_i(T,f,\Ad(w_i)^{-1}Xu')$.
\item The map $C^\infty(O_i)\otimes_{U(\Ad(w_i)\overline{\mathfrak{n}}\cap \mathfrak{n})} U(\mathfrak{g})\otimes_{U(\Ad(w_i)\mathfrak{p})} w_i(\sigma\otimes e^{\lambda+\rho})'\to \mathcal{D}'(U_i,O_i,\mathcal{L})$ defined by $f\otimes T\otimes u'\mapsto \delta_i(T,f,u')$ is injective.
\end{enumerate}
\end{lem}

\begin{lem}\label{lem:left2right}
Let $\{e_i\}$ be a basis of $\Ad(w_i)\overline{\mathfrak{n}}\cap \mathfrak{n}_0$ such that $e_i$ is a restricted root vector, $\alpha_i$ the restricted root for $e_i$, $T,T'\in U(\mathfrak{g})$, $f\in C^\infty(O_i)$ and $u'\in (\sigma\otimes e^{\lambda + \rho})'$.
Then we have
\begin{multline*}
T\delta_i(T',f,u') \\= 
\sum_{(k_1,\dots,k_l)\in \Z_{\ge 0}^l}\delta_i\left((\ad(e_l)^{k_l}\dotsm \ad(e_1)^{k_1}T)T',f\prod_{s = 1}^l\frac{(-x_s)^{k_s}}{k_s!},u'\right),
\end{multline*}
where $x_i$ is a polynomial on $O_i$ given by $\exp(a_1e_1)\dotsm \exp(a_le_l)w_iP/P\mapsto a_i$.
(Notice that the right hand side is a finite sum since $\ad(e_i)$ is nilpotent.)
\end{lem}
\begin{proof}
We remark that by a map $(a_1,\dots,a_l)\mapsto \exp(a_1e_1)\dotsm \exp(a_le_l)$, we have a diffeomorphism $\R^l\simeq w_i\overline{N}w_i^{-1}\cap N_0$ and a Haar measure of $w_i\overline{N}w_i^{-1}\cap N_0$ corresponds to the Euclidean measure of $\R^l$.
Take $\varphi\in C_c^\infty(w_i\overline{N}P,\sigma\otimes e^{\lambda+\rho})$.
Put $n(a) = \exp(a_1e_1)\dotsm \exp(a_le_l)$ for $a = (a_1,\dots,a_l)$.
Recall the definition of $\check{T}$ from Notation.
For $T\in \mathfrak{g}$, we have
\begin{align*}
&\langle T\delta_i(T',f,u'),\varphi\rangle\\
& = \int_{\R^l}u'((\check{T}\widetilde{R}_i(T')\varphi)(n(a)w_i))f(n(a)w_i)da\\
& = \left.\frac{d}{dt}\int_{\R^l}u'(\widetilde{R}_i(T')\varphi)(\exp(tT)n(a)w_i))f(n(a)w_i)da\right|_{t = 0}\\
& = \left.\frac{d}{dt}\int_{\R^l}u'((\widetilde{R}_i(T')\varphi)(n(a)\exp(t\Ad(n(a))^{-1}T)w_i))f(n(a)w_i)da\right|_{t = 0}.
\end{align*}
The formula
\begin{align*}
\Ad(n(a))^{-1}T & = e^{-\ad(a_le_l)}\dotsm e^{-\ad(a_1e_1)}T\\
& = \sum_{(k_1,\dots,k_l)\in \Z_{\ge 0}^l}\frac{(-a_1)^{k_1}}{k_1!}\dotsm \frac{(-a_l)^{k_l}}{k_l!}\ad(e_l)^{k_l}\dotsm \ad(e_1)^{k_1}T
\end{align*}
gives the lemma.
\end{proof}

For $\mathbf{k} = (k_1,\dots,k_l)$, we denote an operator $\ad(e_l)^{k_l}\dotsm \ad(e_1)^{k_1}$ on $\mathfrak{g}$ by $\ad(e)^{\mathbf{k}}$ and a function $((-x_1)^{k_1}/k_1!)\dotsm ((-x_l)^{k_l}/k_l!)\in\mathcal{P}(O_i)$  by $f_\mathbf{k}$.

\begin{lem}\label{lem:diff vanish}
Let $\mathbf{k} = (k_1,\dots,k_l)\in\Z_{\ge 0}^l$ and $X\in \mathfrak{n}_0$.
Assume that $\ad(e)^\mathbf{k}X \in \Ad(w_i)\overline{\mathfrak{n}}\cap \mathfrak{n}_0$.
Then we have $R'_i(\ad(e)^\mathbf{k}X)f_\mathbf{k} = 0$.
\end{lem}
\begin{proof}
We may assume that $X$ is a restricted root vector and denote its restricted root by $\alpha$.
We consider an $\mathfrak{a}_0$-weight with respect to $D_i$.
An $\mathfrak{a}_0$-weight of $f_\mathbf{k}$ is $-\sum_s k_s\alpha_s$.
This implies that $R'_i(\ad(e)^\mathbf{k}X)f_\mathbf{k}$ has an $\mathfrak{a}_0$-weight $\alpha$.
However, $\mathcal{P}(O_i)$ has a decomposition into the direct sum of $\mathfrak{a}_0$-weight spaces and its weight belongs to $\{\sum_{\beta\in\Sigma^+}b_\beta\beta\mid b_\beta\in\Z_{\le 0}\}$.
Hence, we have $R'_i(\ad(e)^\mathbf{k}X)f_\mathbf{k} = 0$.
\end{proof}

%\begin{rem}
%Lemma~\ref{lem:acts nilp} holds if we consider the left regular action.
%\end{rem}

For $f\in\mathcal{P}(O_i)$ and $X\in\mathfrak{n}_0$ we define $L_X(f)$\newsym{$L_X$} by
\[
	L_X(f)(nw_i) = \left.\frac{d}{dt}f(\exp(-tX)nw_i)\right|_{t = 0}.
\]

\begin{lem}\label{lem:caluculation of Xdelta(1,f,u)}
Let $X\in \mathfrak{n}_0$ be a restricted root vector.
For $f\in\mathcal{P}(O_i)$ and $u'\in J'_{w_i^{-1}\eta}(\sigma\otimes e^{\lambda+\rho})$, we have
\begin{multline*}
	(X - \eta'(X))\delta_i(1,f\eta_i^{-1},u') = \delta_i(1,L_X(f)\eta_i^{-1},u') \\
	+ \sum_{\ad(e)^\mathbf{k}X\in \Ad(w_i)\mathfrak{n}_0\cap \mathfrak{n}_0}\delta_i(1,ff_{\mathbf{k}}\eta_i^{-1},(\Ad(w_i)^{-1}(\ad(e)^{\mathbf{k}}X) - \eta'(\ad(e)^{\mathbf{k}}X))u').
\end{multline*}
(Again the sum of the right hand side is finite.)
\end{lem}
\begin{proof}
We have
\[
	X\delta_i(1,f\eta_i^{-1},u')
	= \sum_{\mathbf{k}\in\Z_{\ge 0}^l}\delta_i(\ad(e)^\mathbf{k}X,ff_\mathbf{k}\eta_i^{-1},u').
\]
by Lemma~\ref{lem:left2right}.
Since $\ad(e)^\mathbf{k}X$ belongs to $\mathfrak{n}_0$ and is a restricted root vector, we have either $\ad(e)^\mathbf{k}X\in \Ad(w_i)\overline{\mathfrak{n}_0}\cap \mathfrak{n}_0$ or $\ad(e)^\mathbf{k}X\in \Ad(w_i)\mathfrak{n}_0\cap \mathfrak{n}_0$.
Recall that $\Ad(w_i)\overline{\mathfrak{n}_0}\cap \mathfrak{n}_0 = \Ad(w_i)\overline{\mathfrak{n}}\cap \mathfrak{n}_0$ since $w_i\in W(M)$.
Assume that $\ad(e)^\mathbf{k}X\in \Ad(w_i)\overline{\mathfrak{n}}\cap\mathfrak{n}_0$.
By the definition of $\eta_i$ and $\eta'$, we have $R'_i(-\ad(e)^\mathbf{k}X)(\eta_i^{-1}) = \eta(\ad(e)^\mathbf{k}X)\eta_i^{-1} = \eta'(\ad(e)^\mathbf{k}X)\eta_i^{-1}$.
Hence, using Lemma~\ref{lem:diff vanish},
\begin{align*}
&\delta_i(\ad(e)^\mathbf{k}X,ff_\mathbf{k}\eta_i^{-1},u') \\
&= \delta_i(1,R'_i(-\ad(e)^\mathbf{k}X)(ff_\mathbf{k}\eta_i^{-1}),u')\\
&= \delta_i(1,R'_i(-\ad(e)^\mathbf{k}X)(f)f_\mathbf{k}\eta_i^{-1},u') + \eta'(\ad(e)^\mathbf{k}X)\delta_i(1,ff_\mathbf{k}\eta_i^{-1},u').
\end{align*}
Next assume that $\ad(e)^{\mathbf{k}}X\in \Ad(w_i)\mathfrak{n}_0\cap \mathfrak{n}_0$.
For $h\in\mathcal{P}(O_i)$, define $\widetilde{h}\in\mathcal{P}(U_i)$ by  $\widetilde{h}(nn_0w_iP) = h(nw_iP)$ for $n\in w_i\overline{N}w_i^{-1}\cap N_0$ and $n_0\in w_i\overline{N}w_i^{-1}\cap \overline{N_0}$.
Then we have $\widetilde{R'_i(Y)h} = \widetilde{R}_i(Y)\widetilde{h}$ for all $Y\in \Ad(w_i)\overline{\mathfrak{n}_0}\cap \mathfrak{n}_0$.
Since $\widetilde{f}(pnw_i) = \widetilde{f}(pw_i)$ for $p\in w_i\overline{N}Pw_i^{-1}$ and $n\in w_iN_0w_i^{-1}\cap N_0$, we have $\widetilde{R}_i(-\ad(e)^{\mathbf{k}}X)(\widetilde{f}) = 0$.
Hence we have
\begin{multline*}
	\delta_i(\ad(e)^\mathbf{k}X,ff_\mathbf{k}\eta_i^{-1},u') = \delta_i(1,ff_\mathbf{k}\eta_i^{-1},\Ad(w_i)^{-1}(\ad(e)^{\mathbf{k}}X)u')\\
	= \delta_i(1,R'_i(-\ad(e)^\mathbf{k}X)(\widetilde{f})|_{O_i}f_\mathbf{k}\eta_i^{-1},u') + \\ \delta_i(1,ff_\mathbf{k}\eta_i^{-1},\Ad(w_i)^{-1}(\ad(e)^{\mathbf{k}}X)u').
\end{multline*}
By the same calculation as the proof of Lemma~\ref{lem:left2right}, we have
\[
	\widetilde{L_X(f)} = L_X(\widetilde{f}) = \sum_{\mathbf{k}\in\Z_{\ge 0}^l}\widetilde{R}_i(-\ad(e)^\mathbf{k}X)(\widetilde{f})\widetilde{f_\mathbf{k}}.
\]
Hence
\begin{align*}
\sum_{\mathbf{k}\in\Z_{\ge 0}^l}\delta_i(1,\widetilde{R}_i(-\ad(e)^\mathbf{k}X)(\widetilde{f})|_{O_i}f_\mathbf{k}\eta_i^{-1},u') &= \delta_i(1,\widetilde{L_X(f)}|_{O_i}\eta_i^{-1},u')\\ & = \delta_i(1,L_X(f)\eta_i^{-1},u').
\end{align*}
These imply that
\begin{multline*}
(X - \eta'(X))\delta_i(1,f\eta_i^{-1},u') = \delta_i(1,L_X(f)\eta_i^{-1},u')\\
+ \sum_{\ad(e)^\mathbf{k}X\in \Ad(w_i)\mathfrak{n}_0\cap \mathfrak{n}_0}\delta_i(1,ff_\mathbf{k}\eta_i^{-1},\Ad(w_i)^{-1}(\ad(e)^{\mathbf{k}}X)u')\\
+ \sum_{\ad(e)^\mathbf{k}X\in \Ad(w_i)\overline{\mathfrak{n}}\cap \mathfrak{n}_0}\eta'(\ad(e)^\mathbf{k}X)\delta_i(1,ff_\mathbf{k}\eta_i^{-1},u')\\
- \eta'(X)\delta_i(1,f\eta_i^{-1},u').
\end{multline*}
Since $\eta'$ is a character, if $\mathbf{k}\ne (0,\dots,0)$ then $\eta'(\ad(e)^\mathbf{k}X) = 0$.
Hence we have
\[
	\sum_{\mathbf{k}\in\Z_{\ge 0}^n}\eta'(\ad(e)^\mathbf{k}X)\delta_i(1,ff_\mathbf{k}\eta_i^{-1},u')
	=
	\eta'(X)\delta_i(1,f\eta_i^{-1},u').
\]
This implies
\begin{multline*}
\left(\sum_{\ad(e)^\mathbf{k}X\in \Ad(w_i)\overline{\mathfrak{n}}\cap \mathfrak{n}_0}\eta'(\ad(e)^\mathbf{k}X)\delta_i(1,ff_\mathbf{k}\eta_i^{-1},u')\right)
- \eta'(X)\delta_i(1,f\eta_i^{-1},u')\\
= \sum_{\ad(e)^\mathbf{k}X\in\Ad(w_i)\mathfrak{n}_0\cap \mathfrak{n}_0}\eta'(\ad(e)^\mathbf{k})\delta_i(1,ff_\mathbf{k}\eta_i^{-1},u').
\end{multline*}
We get the lemma.
\end{proof}

\begin{proof}[Proof of Lemma~\ref{lem:acts nilp}]
Since $\ad(\mathfrak{n}_0)$ acts $\mathfrak{g}$ nilpotently, the subspace 
\[
	\{x\in I'_i\mid \text{for some $k$ and for all $X\in \mathfrak{n}_0$, $(X - \eta(X))^kx = 0$}\}
\]
is $\mathfrak{g}$-stable.
Hence we may assume that $x = ((f\eta_i^{-1})\otimes u')\delta_i = \delta_i(1,f\eta_i^{-1},u')$ for some $f\in\mathcal{P}(O_i)$ and $u'\in J'_{w_i^{-1}\eta}(\sigma\otimes e^{\lambda+\rho})$.

Set $V = U(\Ad(w_i)^{-1}\mathfrak{n}_0\cap \mathfrak{n}_0)u'$ where $\mathfrak{n}$ acts $J'_{w_i^{-1}\eta}(\sigma\otimes e^{\lambda+\rho})$ trivially.
Then $V$ is finite-dimensional.
By applying Engel's theorem for $V\otimes (-w_i^{-1}\eta')$, there exists a filtration $0 = V_0 \subset V_1\subset\cdots\subset V_p = V$ such that $(V_s/V_{s - 1})\otimes (-w_i^{-1}\eta'|_{\Ad(w_i)^{-1}\mathfrak{n}_0\cap\mathfrak{n}_0})$ is the trivial representation of $\Ad(w_i)^{-1}\mathfrak{n}_0\cap\mathfrak{n}_0$.
Then we have $V_s/V_{s - 1}\simeq w_i^{-1}\eta'|_{\Ad(w_i)^{-1}\mathfrak{n}_0\cap\mathfrak{n}_0}$ for all $s = 1,2,\dots,p$.
We prove the lemma by induction on $p = \dim V$.

We may assume that $X$ is a restricted root vector.
By Lemma~\ref{lem:caluculation of Xdelta(1,f,u)}, we have
\begin{multline*}
	(X - \eta'(X))\delta_i(1,f\eta_i^{-1},u')\in \delta_i(1,L_X(f)\eta_i^{-1},u')\\ + \sum_{h\in\mathcal{P}(O_i),\ v'\in V_{p - 1}}\delta_i(1,h\eta_i^{-1},v').
\end{multline*}
Since $f$ is a polynomial, there exists a positive integer $c$ such that $(L_X)^c(f) = 0$.
Then $(X - \eta'(X))^c\delta_i(1,f\eta_i^{-1},u') \in \sum_{h\in\mathcal{P}(O_i), v'\in V_{p - 1}}\delta_i(1,h\eta_i^{-1},v')$.
By inductive hypothesis the lemma is proved.
\end{proof}

From the lemma, we get the following vanishing theorem.
Recall that we define the character $w_i^{-1}\eta$ of $\mathfrak{m}\cap \mathfrak{n}_0$ by $(w_i^{-1}\eta)(X) = \eta(\Ad(w_i)X)$.
\begin{lem}\label{lem:vanishing lemma}
Assume that $I_i/I_{i - 1} \ne 0$.
Then the following conditions hold.
\begin{enumerate}
\item The character $\eta$ is unitary.
\item The character $\eta$ is zero on $\Ad(w_i)\mathfrak{n}\cap\mathfrak{n}_0$.
\item The module $J'_{w_i^{-1}\eta}(\sigma\otimes e^{\lambda+\rho})$ is not zero.
\end{enumerate}
\end{lem}
\begin{proof}
(2)
By Lemma~\ref{lem:acts nilp} and the definition of $J'_\eta$, if $I_i/I_{i - 1} \ne 0$ then $\eta = \eta'$.
By the definition of $\eta'$, $\eta = \eta'$ is equivalent to $\eta|_{\Ad(w_i)\mathfrak{n}\cap\mathfrak{n}_0} = 0$.

(3)
This is clear from Lemma~\ref{lem:succ quot is sub of I'}.

(1)
It is sufficient to prove that if $\eta$ is not unitary then $J'_\eta(V) = 0$ for all irreducible representation $V$ of $G$.
By Casselman's subrepresentation theorem, $V$ is a subrepresentation of some principal series representation.
Since $J'_\eta$ is an exact functor, we may assume $V$ is a principal series representation $\Ind_{P_0}^G(\sigma_0\otimes e^{\lambda_0 + \rho_0})$.

Take the Bruhat filtration $\{I_i\}$ of $J'_\eta(V)$.
We prove $I_i/I_{i - 1} = 0$ for all $i$.
By (2), if $\eta$ is non-trivial on $w_iN_0w_i^{-1}\cap N_0$ then $I_i/I_{i - 1} = 0$.
Hence we may assume that $\eta$ is not unitary on $w_i\overline{N_0}w_i^{-1}\cap N_0$.
In this case, an nonzero element of $I_i'$ is not tempered.
Hence $I_i/I_{i - 1} = 0$.
\end{proof}

\begin{rem}
In the next section it is proved that the conditions of Lemma~\ref{lem:vanishing lemma} is also sufficient (Theorem~\ref{thm:succ quot is I'_i}).
\end{rem}

\begin{defn}[Whittaker vectors]\label{defn:Whittaker vectors}
Let $V$ be a $U(\mathfrak{n}_0)$-module.
We define a vector space $\Wh_\eta(V)$\newsym{$\Wh_\eta(V)$} by 
\[
	\Wh_\eta(V) = \{v\in V\mid \text{for all $X\in\mathfrak{n}_0$ we have $Xv = \eta(X)v$}\}.
\]
An element of $\Wh_\eta(V)$ is called a \emph{Whittaker vector}.
\end{defn}

\begin{lem}\label{lem:Whittaker vector, no differential part}
Assume that $\eta|_{\Ad(w_i)\mathfrak{n}\cap \mathfrak{n}_0} = 0$.
Then we have 
\begin{multline*}
\Wh_\eta\left(\left\{\sum_s(f_s\eta_i^{-1}\otimes u'_s)\delta_i\mid f_s\in\mathcal{P}(O_i),\ u'_s\in J'_{w_i^{-1}\eta}(\sigma\otimes e^{\lambda+\rho})\right\}\right) \\
= \{(\eta_i^{-1}\otimes u')\delta_i\mid u'_s\in \Wh_{w_i^{-1}\eta}(\sigma\otimes e^{\lambda+\rho})\}.
\end{multline*}
\end{lem}
\begin{proof}
By the assumption, we have $\eta = \eta'$.
Hence the right hand side is a subspace of the left hand side by Lemma~\ref{lem:caluculation of Xdelta(1,f,u)}.

Take $x = \sum_s (f_s\eta_i^{-1}\otimes u'_s) = \sum_s \delta(1,f_s\eta_i^{-1},u'_s)\in \Wh_\eta(I'_i)$.
We assume that $\{u'_s\}$ is linearly independent.
Take $X\in \Ad(w_i)\overline{\mathfrak{n}}\cap \mathfrak{n}_0$.
Since $\ad(e)^\mathbf{k}X\in \Ad(w_i)\overline{\mathfrak{n}}\cap \mathfrak{n}_0$ for all $\mathbf{k}\in\Z_{\ge 0}^l$, we have $\sum_s\delta_i(1,L_X(f_s)\eta_i^{-1},u'_s) = 0$ by Lemma~\ref{lem:caluculation of Xdelta(1,f,u)}.
Hence $L_X(f_s) = 0$.
This implies $f_s\in \C$.

From the above argument, $x = \delta(1,\eta_i^{-1},u')$ for some $u'\in J'_{w_i^{-1}\eta}(\sigma\otimes e^{\lambda+\rho})$.
Take $X\in \Ad(w_i)\mathfrak{m}\cap \mathfrak{n}_0$.
By Lemma~\ref{lem:caluculation of Xdelta(1,f,u)}, we have
\[
	\delta_i(1,\eta_i^{-1},(\Ad(w_i)^{-1}X - \eta(X))u')\in \sum_{\mathbf{k}\ne 0,\ u_\mathbf{k}\in J'_{w_i^{-1}\eta}(\sigma\otimes e^{\lambda+\rho})}\delta_i(1,f_\mathbf{k}\eta_i^{-1},u_{\mathbf{k}}).
\]
If $\mathbf{k}\ne 0$ then the degree of $f_\mathbf{k}$ is greater than $0$.
So the left hand side must be $0$.
Hence we have $(\Ad(w_i)^{-1}X - \eta(X))u' = 0$.
We have the lemma.
\end{proof}

The following lemma is well-known, but we give a proof for the readers (cf.\ Casselman-Hecht-Milicic~\cite{MR1767896}, Yamashita~\cite{MR849220}).
\begin{lem}\label{lem:Whittaker vectors in non-degenerate case}
Assume that $\supp\eta = \Pi$.
Let $x\in \Wh_\eta(I(\sigma,\lambda)')$.
Then there exists $u'\in \Wh_{w_r^{-1}\eta}((\sigma\otimes e^{\lambda+\rho})')$ such that $x = (\eta_r^{-1}\otimes u')\delta_r$.
\end{lem}
Recall that $r = \# W(M) = \#(W/W_M)$.

\begin{proof}
Assume that $i < r$.
Then $w_iw_{M,0}$ is not the longest element of $W$.
There exists a simple root $\alpha\in\Pi$ such that $s_\alpha w_iw_{M,0} > w_iw_{M,0}$.
This means that $w_iw_{M,0}\Sigma^+\cap \Sigma^+ = s_\alpha(s_\alpha w_iw_{M,0}\Sigma^+\cap \Sigma^+)\cup \{\alpha\}$.
The left hand side is $w_i(\Sigma^+\setminus\Sigma_M^+)\cap \Sigma^+$.
Hence, $\eta$ is not trivial on $\Ad(w_i)\mathfrak{n}\cap\mathfrak{n}_0$.
By Lemma~\ref{lem:vanishing lemma}, $I_i/I_{i - 1} = 0$.
This implies that $J'_\eta(I(\sigma,\lambda))\subset I_r'$.
There exists a polynomial $f_s\in\mathcal{P}(X_r)$ and $u_s'\in J'_{w_r^{-1}\eta}(\sigma\otimes e^{\lambda+\rho})$ such that $x = \sum_s((f_s\eta_r^{-1})\otimes u'_s)\delta_r$.
By Lemma~\ref{lem:Whittaker vector, no differential part}, we have the lemma.
\end{proof}

\section{Analytic continuation}\label{sec:Analytic continuation}
The aim of this section is to prove that $\im\Res_i = I'_i$ if $I_i/I_{i - 1} \ne 0$.

Let $P_\eta$ be the parabolic subgroup corresponding to $\supp\eta \subset \Pi$  containing $P_0$ and $P_\eta = M_\eta A_\eta N_\eta$ its Langlands decomposition such that $A_\eta\subset A_0$.\newsym{$P_\eta = M_\eta A_\eta N_\eta$}
Denote the complexification of the Lie algebra of $P_\eta$, $M_\eta$, $A_\eta$, $N_\eta$ by $\mathfrak{p}_\eta$, $\mathfrak{m}_\eta$, $\mathfrak{a}_\eta$, $\mathfrak{n}_\eta$, respectively.\newsym{$\mathfrak{p}_\eta = \mathfrak{m}_\eta\oplus\mathfrak{a}_\eta\oplus\mathfrak{n}_\eta$}
Put $\mathfrak{l}_\eta = \mathfrak{m}_\eta\oplus\mathfrak{a}_\eta$, $\overline{N_\eta} = \theta(N_\eta)$ and $\overline{\mathfrak{n}_\eta} = \theta(\mathfrak{n}_\eta)$.\newsym{$\mathfrak{l}_\eta$}\newsym{$\overline{N_\eta}$}\newsym{$\overline{\mathfrak{n}_\eta}$}
Set $\Sigma^+_\eta = \{\sum_{\alpha\in\supp\eta}n_\alpha\alpha\in\Sigma^+\mid n_\alpha\in\Z_{\ge 0}\}$ and $\Sigma^-_\eta = -\Sigma^+_\eta$.\newsym{$\Sigma^+_\eta,\Sigma^-_\eta$}
The same notation will be used for $M$ with suffix $M$, i.e., $P_{M,\eta} = M_{M,\eta}A_{M,\eta}N_{M,\eta}$ is the parabolic subgroup of $M$ containing $M\cap P_0$ corresponding to $\supp\eta\cap\Sigma_M^+$, $\mathfrak{p}_{\mathfrak{m},\eta} = \mathfrak{m}_{\mathfrak{m},\eta}\oplus\mathfrak{a}_{\mathfrak{m},\eta}\oplus\mathfrak{n}_{\mathfrak{m},\eta}$ is a complexification of the Lie algebra of $P_{M,\eta} = M_{M,\eta}A_{M,\eta}N_{M,\eta}$.%
\newsym{$P_{M,\eta} = M_{M,\eta}A_{M,\eta}N_{M,\eta}$}\newsym{$\mathfrak{p}_{\mathfrak{m},\eta} = \mathfrak{m}_{\mathfrak{m},\eta}\oplus\mathfrak{a}_{\mathfrak{m},\eta}\oplus\mathfrak{n}_{\mathfrak{m},\eta}$}

For $w\in W$, there is an open dense subset $w\overline{N}P/P$ of $G/P$ and it is diffeomorphic to $\overline{N}$.
Then for $w,w'\in W$, there exists a map $\Phi_{w,w'}$\newsym{$\Phi_{w,w'}$} from some open dense subset $U\subset\overline{N}$ to $\overline{N}$ such that $w\overline{n}P/P = w'\Phi_{w,w'}(\overline{n})P/P$ for $\overline{n}\in U$.
%%%%%%%%%%%%%%%%%%%%%%%%%%%%%%%%%%%%%%%%%%%%%%%%%%%%%%%%%%%%%%%%
The map $\Phi_{w,w'}$ is a rational function.
%%%%%%%%%%%%%%%%%%%%%%%%%%%%%%%%%%%%%%%%%%%%%%%%%%%%%%%%%%%%%%%%
%\begin{lem}\label{lem:Phi is rational}
%The diffeomorphism $\Phi_{w,w'}$ is a rational function.
%\end{lem}
%\begin{proof}
%
%\end{proof}

Since the exponential map $\exp\colon\Lie(\overline{N})\to \overline{N}$ is diffeomorphism, $\overline{N}$ has a structure of a vector space.
\begin{lem}\label{lem:property of e^rho}
\begin{enumerate}
\item The map $\overline{N}\to \R$ defined by $\overline{n}\mapsto e^{8\rho(H(\overline{n}))}$ is a polynomial.
\item For all $\overline{n}\in\overline{N}$ we have $e^{8\rho(H(\overline{n}))} \ge 1$.
\item Take $H_0\in \mathfrak{a}$ such that $\alpha(H_0) = -1$ for all $\alpha\in \Pi\setminus\Sigma_M$.
There exists a continuous function $Q(\overline{n})\ge 0$ on $\overline{N}$ such that the following conditions hold:
(a) The function $Q$ vanishes only at the unit element.
(b) $e^{8\rho(H(\overline{n}))} \ge Q(\overline{n})$.
(c) $Q(\exp(tH_0)\overline{n}\exp(-tH_0)) \ge e^{8t}Q(\overline{n})$ for $t\in \R_{>0}$ and $\overline{n}\in\overline{N}$.
\end{enumerate}
\end{lem}
\begin{proof}
By Knapp~\cite[Proposition~7.19]{MR1880691}, there exists an irreducible finite-dimensional $V_{4\rho}$ of $\mathfrak{g}$ with the highest weight $4\rho\in \mathfrak{a}_0^*\subset \mathfrak{h}^*$.
Let $v_{4\rho}\in V_{4\rho}$ be a highest weight vector and $v_{-4\rho}^*\in V_{4\rho}^*$ the lowest weight vector of $V_{4\rho}^*$.
Then $\C v_{4\rho}$ is a $1$-dimensional unitary representation of $M$.
Take $\overline{n}\in \overline{N}$ and decompose $\overline{n} = kan$ where $k\in K$, $a\in A_0$ and $n\in N_0$.

First we prove (1).
We have $\theta(\overline{n})^{-1}\overline{n} = \theta(n)^{-1}a^2n$.
Hence
\begin{align*}
\langle \theta(\overline{n})^{-1}\overline{n}v_{4\rho},v^*_{-4\rho}\rangle
& = \langle \theta(n)^{-1}a^2nv_{4\rho},v^*_{-4\rho}\rangle\\
& = \langle a^2nv_{4\rho},\theta(n)v^*_{-4\rho}\rangle\\
& = e^{8\rho(H(\overline{n}))}\langle v_{4\rho},v^*_{-4\rho}\rangle.
\end{align*}
The left hand side is a polynomial.

Next we prove (2) and (3).
Fix a compact real form of $\mathfrak{g}$ containing $\Lie(K)$ and take an inner product on $V_{4\rho}$ which is invariant under this compact real form.
We normalize an inner product $||\cdot||$ so that $||v_{4\rho}|| = 1$.
Then we have $||\overline{n}v_{4\rho}|| = ||kanv_{4\rho}|| = ||av_{4\rho}|| = e^{4\rho(H(\overline{n}))}||v_{4\rho}|| = e^{4\rho(H(\overline{n}))}$.
For $\nu\in\mathfrak{h}^*$ let $Q_\nu(\overline{n})\in V_{4\rho}$ be the $\nu$-weight vector such that $\overline{n}v_{4\rho} = \sum_\nu Q_{\nu}(\overline{n})$.
Then we have $e^{8\rho(H(\overline{n}))} = \sum_\nu ||Q_\nu(\overline{n})||^2$.
Since $Q_{4\rho}(\overline{n}) = v_{4\rho}$, we have $e^{8\rho(H(\overline{n}))} \ge 1$.

Put $Q(\overline{n}) = \sum_{w\in W(M)\setminus\{e\}}||Q_{4w\rho}(\overline{n})||^2$.
Assume that $\overline{n} \ne e$.
Then there exist $w\in W(M)\setminus\{e\}$, $m'\in M$, $a'\in A$, $n'\in N$ and $\overline{n}'\in \overline{N}$ such that $\overline{n} = w\overline{n}'m'a'n'$.
Let $v_{-4w\rho}^*\in V_{4\rho}^*$ be a weight vector with $\mathfrak{h}$-weight $-4w\rho$.
Then we have
\begin{multline*}
	||Q_{4w\rho}(\overline{n})||
	= \lvert\langle \overline{n}v_{4\rho},v_{-4w\rho}^*\rangle\rvert
	= \lvert\langle w\overline{n}'m'a'n'v_{4\rho},v_{-4w\rho}^*\rangle\rvert\\
	= \lvert\langle a'v_{4\rho},w^{-1}v^*_{-4w\rho}\rangle\rvert
	= e^{4\rho(\log a')}\lvert \langle v_{4\rho},w^{-1}v_{-4w\rho}^*\rangle\rvert
	\ne 0.
\end{multline*}
Hence, if $\overline{n}\in \overline{N}\setminus\{e\}$ then $Q(\overline{n})\ne 0$.

Let $t$ be a positive real number.
Using $Q_\nu(\exp(tH_0)\overline{n}\exp(-tH_0)) = e^{t(\nu - 4\rho)(H_0)}Q_\nu(\overline{n})$, we have
\[
Q(\exp(tH_0)\overline{n}\exp(-tH_0))
= \sum_{w\in W(M)\setminus\{e\}} e^{8t(w\rho - \rho)(H_0)}\lvert Q_{4w\rho_0}(\overline{n})\rvert^2.
\]
Since $(w\rho - \rho)(H_0) \ge 1$ for $w\in W(M)\setminus\{e\}$, we get the lemma.
\end{proof}

\begin{rem}\label{rem:asymptotic at infinity of e^rho}
The condition Lemma~\ref{lem:property of e^rho} (3) implies that $\lim_{\overline{n}\to \infty}Q(\overline{n}) = \infty$.
The proof is the following.
Take $H_0$ as in Lemma~\ref{lem:property of e^rho}.
Let $\{e_1,\dots,e_l\}$ be a basis of $\overline{\mathfrak{n}}$.
Here, we assume that each $e_i$ is a restricted root vector and denote its root by $\alpha_i$.
Any $\overline{n}\in \overline{N}$ can be written as $\overline{n} = \exp(\sum_{i = 1}^l a_ie_i)$ where $a_i\in \R$.
Put $r(\overline{n}) = \sum_{i = 1}^l\lvert a_i\rvert^{-1/\alpha_i(H_0)}$.
Set $C = \min_{r(\overline{n}) = 1}Q(\overline{n})$.
Since $Q(\overline{n}) > 0$ if $\overline{n}$ is not the unit element, $C > 0$.
Then we have $Q(\overline{n})\ge Cr(\overline{n})^8$ if $r(\overline{n}) > 1$.
If $\overline{n}\to \infty$ then $r(\overline{n})\to\infty$.
Hence, $Q(\overline{n})\to\infty$.
\end{rem}

\begin{lem}\label{lem:extension of polynomials}
Let $f$ be a polynomial on $\overline{N}$.
There exists a positive integer $k$ and a $C^\infty$-function $h$ on $G/P$ such that $h(w_i\overline{n}P/P) = e^{-k\rho(H(\overline{n}))}f(\overline{n})$ for all $\overline{n}\in\overline{N}$.
\end{lem}
\begin{proof}
By Lemma~\ref{lem:property of e^rho} and Remark~\ref{rem:asymptotic at infinity of e^rho}, we can choose a positive integer $C$ such that $e^{-8C\rho(H(\overline{n}))}f(\overline{n}) \to 0$ when $\overline{n}\to \infty$.
Let $\widetilde{f}$ be a function on $U_i$ defined by $\widetilde{f}(w_i\overline{n}P/P) = e^{-8C\rho(H(\overline{n}))}f(\overline{n})$ for $\overline{n}\in \overline{N}$.
We prove that $\widetilde{f}$ can be extended to $G/P$.
Take $w\in W(M)$.
Then $\widetilde{f}$ is defined in a subset of $w\overline{N}P/P$.
Using a diffeomorphism $\overline{N}\simeq w\overline{N}P/P$, $\widetilde{f}$ defines a rational function $\widetilde{f}\circ \Phi_{w_i,w}$ defined on an open dense subset of $\overline{N}$.
By the condition of $C$, the function $\widetilde{f}\circ \Phi_{w_i,w}$ has no pole.
Hence, $\widetilde{f}$ defines a $C^\infty$-function on $w\overline{N}P/P$.
Since $\bigcup_{w\in W(M)}w\overline{N}P/P = G/P$, the lemma follows.
\end{proof}

Define $\kappa\colon G\to K$ and $H\colon G\to \Lie(A_0)$ by $g\in \kappa(g)\exp H(g) N_0$.\newsym{$\kappa$}\newsym{$H$}
Recall that for a representation $V$ of $\mathfrak{g}$, $\nu\in \mathfrak{a}_0^*$ is called an exponent of $V$ if $\nu + \rho_0|_{\mathfrak{m}\cap\mathfrak{a}_0}$ is an $\mathfrak{a}_0$-weight of $V/\mathfrak{n}_0V$.

\begin{prop}\label{prop:convergence and continuation}
Let $\varphi$ be a $\sigma$-valued function on $K$ which satisfies $\varphi(km) = \sigma(m)^{-1}\varphi(k)$ for all $k\in K$ and $m\in M\cap K$.
We define $\varphi_\lambda\in I(\sigma,\lambda)$ by $\varphi_\lambda(kman) = e^{-(\lambda + \rho)(\log a)}\sigma(m)^{-1}\varphi(k)$ for $k\in K$, $m\in M$, $a\in A$ and $n\in N$.
For $u'\in J'_{w_i^{-1}\eta}(\sigma\otimes e^{\lambda+\rho})$ and $f\in\mathcal{P}(O_i)$, put $I_{f,u'}(\varphi_\lambda) = \int_{w_i\overline{N}w_i^{-1}\cap N_0}u'(\varphi_\lambda(nw_i))\eta(n)^{-1}f(nw_i)dn$. (If $\supp\varphi\subset K\cap w_i\overline{N}P$ then the integral converges.)
\begin{enumerate}
\item If $\langle\alpha,\re\lambda\rangle$ is sufficiently large for each $\alpha\in\Sigma^+\setminus\Sigma^+_M$ then the integral $I_{f,u'}(\varphi_\lambda)$ absolutely converges.
\item As a function of $\lambda$, the integral $I_{f,u'}(\varphi_\lambda)$ has a meromorphic continuation to $\mathfrak{a}^*$.
\item If $\supp\eta = \Pi$ and $i = r$ then $I_{f,u'}(\varphi_\lambda)$ is holomorphic for all $\lambda\in\mathfrak{a}^*$.
\item Let $\nu$ be an exponent of $\sigma$ and $u'\in \Wh_{w_i^{-1}\eta}((\sigma\otimes e^{\lambda+\rho})')$.
If $2\langle\alpha,\lambda+\nu\rangle/\lvert\alpha\rvert^2\not\in \Z_{\le 0}$ for all $\alpha\in \Sigma^+\setminus w_i^{-1}(\Sigma^+\cup\Sigma^-_{\eta})$ then $I_{1,u'}(\varphi_\mu)$ is holomorphic at $\mu = \lambda$.
\end{enumerate}
\end{prop}
\begin{proof}
First we prove (1).
If $f = 1$ then this is a well-known result.
For a general $f$, extends $f$ to a function on $w_i\overline{N}P/P$ by $f(w_inn') = f(w_in)$ for $n\in w_i\overline{N}w_i^{-1}\cap N_0$ and $n'\in w_i\overline{N}w_i^{-1}\cap \overline{N_0}$.
Then by Lemma~\ref{lem:extension of polynomials} there exists a positive number $C$ such that $\overline{n}\mapsto e^{-C\rho(H(\overline{n}))}f(w_i\overline{n})$ extends to a function $h$ on $G/P$.
Since 
\[
	I_{f,u'}(\varphi_\lambda) = \int_{w_i\overline{N}w_i^{-1}\cap N_0}u'(\varphi(\kappa(nw))e^{-(\lambda+\rho)(H(nw_r))}f(nw_r)\eta(n)^{-1}dn,
\]
we have $I_{f,u'}(\varphi_\lambda) = I_{1,u'}((\varphi h)_{\lambda - C\rho})$.

We prove (3).
By dualizing Casselman's subrepresentation theorem, there exist a representation $\sigma_0$ of $M_0$ and $\lambda_0\in\mathfrak{a}_0^*$ such that $\sigma$ is a quotient of $\Ind_{M\cap P_0}^{M}(\sigma_0\otimes e^{\lambda_0})$.
Then we may regard $u'\in J'_{w_r^{-1}\eta}(\Ind_{M\cap P_0}^M(\sigma_0\otimes e^{\lambda_0}))$.
By the proof of Lemma~\ref{lem:Whittaker vectors in non-degenerate case}, there exist a polynomial $f_0$ on $(M\cap N_0)w_{M,0}(M\cap P_0)/(M\cap P_0)$ and $u'_0\in (\sigma_0\otimes e^{\lambda_0})'$ such that $u'$ is given by
\[
	\varphi_0\mapsto \int_{M\cap N_0}u'_0(\varphi_0(n_0w_{M,0}))f_0(n_0w_{M,0})\eta(n_0)^{-1}dn_0
\]
Let $\pi\colon \Ind_{P_0}^G(\sigma_0\otimes e^{\lambda + \lambda_0 + \rho})\to I(\sigma,\lambda)$ be the map induced from the quotient map $\Ind_{M\cap P_0}^{M}(\sigma_0\otimes e^{\lambda_0})\to \sigma$.
Take $\widetilde{\varphi}\colon K\to \sigma_0$ which satisfies $\widetilde{\varphi}(km) = \sigma_0^{-1}(m)\widetilde{\varphi}(k)\ (k\in K,\ m\in M_0)$ and $\pi(\widetilde{\varphi}_{\lambda + \lambda_0}) = \varphi_\lambda$.
Define a polynomial $\widetilde{f}\in \mathcal{P}(w_iw_{M,0}\overline{N_0}P_0/P_0)$ by 
\[
	\widetilde{f}(w_iw_{M,0}nn_0P_0/P_0) = f(w_inP/P)f_0(w_{M,0}n_0(M\cap P_0)/(M\cap P_0))
\]
for $n\in \overline{N}$ and $n_0\in M\cap \overline{N_0}$. (Notice that $w_{M,0}(M\cap \overline{N_0}) = (M\cap N_0)w_{M,0}$.)
Then we have
\begin{multline*}
	I_{f,u'}(\varphi_\lambda)\\ = \int_{w_iw_{M,0}\overline{N_0}(w_iw_{M,0})^{-1}\cap N_0}u_0'(\widetilde{\varphi}(nw_iw_{M,0}))\widetilde{f}(w_iw_{M,0}nP_0/P_0)\eta(n)^{-1}dn.
\end{multline*}
Hence, we may assume that $P$ is minimal.
By the same argument in (1), we may assume $f = 1$.
If $f = 1$ then this integral is known as a Jacquet integral and the analytic continuation is well-known~\cite{MR0271275}.

We prove (2) and (4).
By the same argument in (1), we may assume that $f = 1$.
Take $w'\in W_{M_\eta}$ and $w''\in W(M_\eta)^{-1}$ such that $w_i = w'w''$.
Then we have $w_i\overline{N}w_i^{-1}\cap N_0 = (w'\overline{N_0}(w')^{-1}\cap N_0)w'(w''\overline{N_0}(w'')^{-1}\cap N_0)(w')^{-1}$.
The condition $w'\in W_{M_\eta}$ implies that $w'(\Sigma^+\setminus\Sigma^+_\eta) = \Sigma^+\setminus\Sigma^+_\eta$.
Hence, $\supp\eta \cap w'\Sigma^+ = \supp\eta\cap w'\Sigma_\eta^+$.
This implies 
\begin{multline*}
	\supp\eta\cap w'(w''\Sigma^-\cap \Sigma^+) = \supp\eta\cap w_i\Sigma^-\cap w'\Sigma^+\\ = \supp\eta\cap w_i\Sigma^-\cap w_i(w'')^{-1}\Sigma_\eta^+\subset \supp\eta\cap w_i\Sigma^-\cap w_i\Sigma^+ = \emptyset,
\end{multline*}
i.e., $\eta$ is trivial on $w'(w''\overline{N_0}(w'')^{-1}\cap N_0)(w')^{-1}$.
Hence, we have
\[
	I_{1,u'}(\varphi)
	= \int_{w'\overline{N_0}(w')^{-1}\cap N_0}\int_{w''\overline{N_0}(w'')^{-1}\cap N_0}u'(\varphi(n_1w'n_2w''))\eta(n_1)^{-1}dn_2dn_1.
\]
Put $P' = (w''P(w'')^{-1}\cap M_\eta)N_\eta$.
By the definition of $W(M_\eta)$, we have $w''N_0(w'')^{-1}\supset N_0\cap M_\eta$, this implies that $P'$ (resp.\ $w''P(w'')^{-1}\cap M_\eta$) is a parabolic subgroup of $G$ (resp.\ $M_\eta$).
Define a $G$-module homomorphism $A(\sigma,\lambda)\colon I(\sigma,\lambda)\to \Ind_{P'}^G(w''(\sigma)\otimes e^{w''\lambda+\rho})$ by
\[
	(A(\sigma,\lambda)\varphi)(x) = \int_{w''\overline{N_0}(w'')^{-1}\cap N_0}\varphi(xnw'')dn.
\]
By a result of Knapp and Stein~\cite{MR582703}, this homomorphism has a meromorphic continuation.
We have
\[
	I_{1,u'}(\varphi) = \int_{w'\overline{N_0}(w')^{-1}\cap N_0}u'((A(\sigma,\lambda)\varphi)(nw'))\eta(n)^{-1}dn.
\]
Notice that $w'\overline{N_0}(w')^{-1}\cap N_0\subset M_\eta$.
Hence we get (2) by (3).

To prove (4), we calculate $(w'')^{-1}\Sigma^-\cap\Sigma^+$.
Since $(w'')^{-1}\in W(M_\eta)$, we have $(w'')^{-1}\Sigma_\eta^-\subset \Sigma^-$.
Hence $(w'')^{-1}\Sigma_\eta^-\cap \Sigma^+ = \emptyset$.
Then
\begin{align*}
(w'')^{-1}\Sigma^-\cap \Sigma^+ & = (w'')^{-1}(\Sigma^-\setminus\Sigma_\eta^-)\cap \Sigma^+\\
& = (w'')^{-1}(w')^{-1}(\Sigma^-\setminus\Sigma_\eta^-)\cap\Sigma^+\\
& = w_i^{-1}(\Sigma^-\setminus\Sigma_\eta^-)\cap\Sigma^+\\
& = \Sigma^+\setminus w_i^{-1}(\Sigma^+\cup\Sigma_\eta^-).
\end{align*}
Hence we have $2\langle\alpha,\lambda+\nu\rangle/\lvert\alpha\rvert^2\not\in\Z_{\ge 0}$ for all $\alpha\in(w'')^{-1}\Sigma^-\cap\Sigma^+$.
By an argument of Knapp and Stein~\cite{MR582703}, $A(\sigma,\mu)$ is holomorphic at $\mu = \lambda$ if $\lambda$ satisfies the conditions of (4).
Hence we get (4).
\end{proof}

In the rest of this section, we denote the Bruhat filtration $I_i\subset J'(I(\sigma,\lambda))$ by $I_i(\lambda)$.
The following result is a corollary of Proposition~\ref{prop:convergence and continuation}.
\begin{lem}\label{lem:meromorphic extension}
Let $x\in I_i'$.
Then there exists a distribution $x_t\in I_i(\lambda + t\rho)$ with meromorphic parameter $t$ such that $x_t|_{U_i}$ is a distribution with holomorphic parameter $t$ and $(x|_{U_i})|_{t = 0} = x$.

Moreover, if $Tx = 0$ for $T\in U(\mathfrak{g})$, then $Tx_t = 0$.
\end{lem}

Let $C^\infty(K,\sigma)$ be the space of $\sigma$-valued $C^\infty$-functions.
For $X\in \mathfrak{g}$ and $\lambda\in\mathfrak{a}^*$, we define an operator $D(X,\lambda)$ on $C^\infty(K,\sigma)$ as follows.
For $\varphi\in C^\infty(K,\sigma)$, 
\begin{multline*}
	(D(X,\lambda)\varphi)(k)\\ = \frac{d}{dt}\left.(\sigma\otimes e^{\lambda+\rho})(\exp (-H(\exp(-tX)k)))\varphi(\kappa(\exp(-tX)k))\right|_{t = 0}.\newsym{$D(X,\lambda)$}
\end{multline*}
If we regard $I(\sigma,\lambda)$ as a subspace of $C^\infty(K,\sigma)$, $(X\varphi)(k) = (D(X,\lambda)\varphi)(k)$ for $\varphi\in I(\sigma,\lambda)$.
It is easy to see that for some $D_1$ and $D_2$ we have $D(X,\lambda + t\rho) = D_1 + tD_2$ for all $t\in\C$.

\begin{lem}\label{lem:holomorphic extension}
Assume that the conditions of Lemma~\ref{lem:vanishing lemma} (1)--(3) hold.
For $x\in I'_i$ there exists a distribution $x_t\in I_i(\lambda + t\rho)$ with holomorphic parameter $t$ defined near $t = 0$ such that $x_0 = x$ on $U_i$.
\end{lem}
\begin{proof}
First we remark that $\eta = \eta'$ in Lemma~\ref{lem:acts nilp} by the condition (2) of Lemma~\ref{lem:vanishing lemma}.

We prove by induction on $i$.
If $i = 1$, then $x\in I'_1$.
Take a distribution $x_t\in I_1(\lambda + t\rho)$ as in Lemma~\ref{lem:meromorphic extension}.
Then $x_t|_{U_1}$ is holomorphic with respect to the parameter $t$.
Since $\supp x_t\subset X_1$, $x_t|_{(G/P)\setminus X_1}$ is holomorphic with respect to the parameter $t$.
Hence $x_t$ is holomorphic with respect to the parameter $t$ on $U_1\cup ((G/P)\setminus X_1) = G/P$.
We have the lemma.

Assume that $i > 1$.
First we prove the following claim: for $y\in I_{i - 1}$, there exists a distribution $y_t\in I_{i - 1}(\lambda + t\rho)$ with holomorphic parameter $t$ defined near $t = 0$ such that $y_0 = y$.
Using inductive hypothesis to $y|_{U_{i - 1}}$, there exists a distribution $y_t^{(i - 1)}\in I_{i - 1}(\lambda + t\rho)$ with holomorphic parameter $t$ defined near $t = 0$ such that $y_0^{(i - 1)} = y$ on $U_{i - 1}$.
Since the supports of both sides are contained in $\bigcup_{j\le i - 1}N_0w_jP/P$, we have $y_0^{(i - 1)} = y$ on $\bigcup_{j\ge i - 1}N_0w_jP/P$.
Using inductive hypothesis to $(y - y_0^{(i - 1)})|_{U_{i - 2}}$, there exists a distribution $y_t^{(i - 2)}\in I_{i - 2}(\lambda + t\rho)$ with holomorphic parameter $t$ defined near $t = 0$ such that $y_0^{(i - 2)} = y - y_0^{(i - 1)}$ on $U_{i - 2}$.
Since the supports of both sides are contained in $\bigcup_{j\le i - 2}N_0w_jP/P$, we have $y_0^{(i - 1)} + y_0^{(i - 2)} = y$ on $\bigcup_{j\ge i - 2}N_0w_jP/P$.
Iterating this argument, for $j = 1,\dots,i - 1$ there exists a distribution $y_t^{(j)}\in I_j(\lambda + t\rho)$ with holomorphic parameter $t$ defined near $t = 0$ such that $y = y_0^{(1)} + \dots + y_0^{(i - 1)}$.
Hence we get the claim.

Now we prove the lemma.
By Lemma~\ref{lem:meromorphic extension}, there exists a distribution $x'_t\in I_i(\lambda + t\rho)$ with meromorphic parameter $t$ such that $x'_t|_{U_i}$ is holomorphic and $(x'_t|_{U_i})|_{t = 0} = x$.
Let $x'_t = \sum_{s = -p}^\infty x^{(s)}t^s$ be the Laurent series of $x'_t$.
Now we prove the following claim: if there exists a distribution $x'_t = \sum_{s = -p}^\infty x^{(s)}t^s\in I_i(\lambda + t\rho)$ with meromorphic parameter $t$ defined near $t = 0$ such that $x'_t|_{U_i}$ is holomorphic and $(x'_t|_{U_i})|_{t = 0} = x$, then there exists $x_t\in I(\lambda)$ with holomorphic parameter $t$ defined near $t = 0$ such that $x_0|_{U_i} = x$.
We prove the claim by induction on $p$.

If $p = 0$, we have nothing to prove.
Assume $p > 0$.
Take $E\in \mathfrak{n}_0$ and define differential operators $E_0$ and $E_1$ by $D(E,\lambda + t\rho) = E_0 + tE_1$.
By Lemma~\ref{lem:acts nilp}, there exists a positive integer $k$ such that $(E_0 + tE_1 - \eta(E))^k x'_t = 0$.
Hence, we have $(E_0 - \eta(E))^kx^{(-p)} = 0$.
Since $x_t|_{U_i}$ is holomorphic, we have $\supp x^{(-p)}\subset \bigcup_{j < i}N_0w_jP/P$.
Hence we have $x^{(-p)}\in I_{i - 1}$.
By the claim stated in the third paragraph of this proof, there exists $x''_t\in I_{i - 1}(\lambda + t\rho)$ with holomorphic parameter $t$ defined near $t = 0$ such that $x''_0 = x^{(-p)}$.
Using inductive hypothesis for $x_t' - t^{-p}x''_t$, we get the claim and the claim implies the lemma.
\end{proof}

\begin{thm}\label{thm:succ quot is I'_i}
\begin{enumerate}
\item The module $I_i/I_{i - 1}$ is non-zero if and only if the conditions of Lemma~\ref{lem:vanishing lemma} (1)--(3) hold.
\item If $I_i/I_{i - 1} \ne 0$ then we have $I_i/I_{i - 1} \simeq I'_i$.
\end{enumerate}
\end{thm}
\begin{proof}
Assume that the conditions of Lemma~\ref{lem:vanishing lemma} (1)--(3) hold.
We prove that the restriction map $\Res_i\colon I_i\to I'_i$ is surjective.

For $x\in I'_i$, take $x_t\in I_i(\lambda + t\rho)$ as in Lemma~\ref{lem:holomorphic extension}.
Then we have $\Res_i(x_0) = (x_0)|_{U_i} = x$.
Hence $\Res_i$ is surjective.
\end{proof}

\section{Twisting functors}\label{sec:Twisting functors}
Arkhipov defined the \emph{twisting functor} for $\widetilde{w}\in\widetilde{W}$~\cite{MR2074588}.
In this section, we define a modification of the twisting functor.

Let $\mathfrak{g}^\mathfrak{h}_\alpha$ be the root space of $\alpha\in\Delta$.\newsym{$\mathfrak{g}^\mathfrak{h}_\alpha$}
Set $\mathfrak{u}_0 = \bigoplus_{\alpha\in \Delta^+}\mathfrak{g}^\mathfrak{h}_\alpha$\newsym{$\mathfrak{u}_0$}, $\overline{\mathfrak{u}_0} = \bigoplus_{\alpha\in \Delta^+}\mathfrak{g}^\mathfrak{h}_{-\alpha}$\newsym{$\overline{\mathfrak{u}_0}$} and $\mathfrak{u}_{0,{\widetilde{w}}} = \Ad(\widetilde{w})\overline{\mathfrak{u}_0}\cap \mathfrak{u}_0$.\newsym{$\mathfrak{u}_{0,{\widetilde{w}}}$}
Let $\psi$ be a character of $\mathfrak{u}_{0,{\widetilde{w}}}$.
Put $S_{\widetilde{w},\psi} = U(\mathfrak{g})\otimes_{U(\mathfrak{u}_{0,{\widetilde{w}}})}((U(\mathfrak{u}_{0,{\widetilde{w}}})^*)_{\text{$\mathfrak{h}$-finite}}\otimes_\C\psi)$.\newsym{$S_{\widetilde{w},\psi}$}
This is a right $U(\mathfrak{u}_{0,{\widetilde{w}}})$-module and left $U(\mathfrak{g})$-module.
We define a $U(\mathfrak{g})$-bimodule structure on $S_{\widetilde{w},\psi}$ in the following way.
Let $\{e_1,\dots,e_l\}$ be a basis of $\mathfrak{u}_{0,{\widetilde{w}}}$ such that each $e_i$ is a root vector and $\bigoplus_{s\le t - 1}\C e_s$ is an ideal of $\bigoplus_{s\le t}\C e_s$ for each $t = 1,2,\dots,l$.
Notice that a multiplicative set $\{(e_k - \psi(e_k))^n\mid n\in\Z_{\ge 0}\}$ satisfies the Ore condition for $k = 1,2,\dots,l$.
Then we can consider the localization of $U(\mathfrak{g})$ by $\{(e_k - \psi(e_k))^n\mid n\in\Z_{\ge 0}\}$.
We denote the resulting algebra by $U(\mathfrak{g})_{e_k - \psi(e_k)}$.
Put $S_{e_k - \psi(e_k)} = U(\mathfrak{g})_{e_k - \psi(e_k)}/U(\mathfrak{g})$.
Then $S_{e_k - \psi(e_k)}$ is a $U(\mathfrak{g})$-bimodule.\newsym{$S_{e_k - \eta(e_k)}$}

\begin{prop}\label{prop:1-dim decomposition of S_w}
As a right $U(\mathfrak{u}_{0,w})$-module and left $U(\mathfrak{g})$-module, we have $S_{\widetilde{w},\psi} \simeq S_{e_1 - \psi(e_1)}\otimes_{U(\mathfrak{g})}S_{e_2 - \psi(e_2)}\otimes_{U(\mathfrak{g})}\dots\otimes_{U(\mathfrak{g})}S_{e_l - \psi(e_l)}$.
Moreover, the $U(\mathfrak{g})$-bimodule structure induced from this isomorphism is independent of a choice of $e_i$.
\end{prop}

The proof of this proposition is similar to that of Arkhipov~\cite[Thoerem~2.1.6]{MR2074588}.
We omit it.
An element of the right hand side is written as a sum of a form $(e_1 - \eta(e_1))^{-(k_1 + 1)}\otimes\dots\otimes (e_l - \eta(e_l))^{-(k_l + 1)}T$ for $T\in U(\mathfrak{g})$.
We denote this element by $(e_1 - \eta(e_1))^{-(k_1 + 1)}\dotsm (e_l - \eta(e_l))^{-(k_l + 1)}T$ for short.

Proposition~\ref{prop:1-dim decomposition of S_w} gives the $U(\mathfrak{g})$-bimodule structure of $S_{\widetilde{w},\psi}$.
For a $U(\mathfrak{g})$-module $V$, we define a $U(\mathfrak{g})$-module $T_{\widetilde{w},\psi}V$ by $T_{\widetilde{w},\psi}V = S_{\widetilde{w},\psi}\otimes_{U(\mathfrak{g})}(\widetilde{w}V)$.
(Recall that $\widetilde{w}V$ is a $\mathfrak{g}$-module twisted by $\widetilde{w}$. See Notation.)
This gives the twisting functor $T_{\widetilde{w},\psi}$.\newsym{$T_{\widetilde{w},\psi}$}
If $\psi$ is the trivial representation, $T_{\widetilde{w},\psi}$ is the twisting functor defined by Arkhipov.
We put $T_{\widetilde{w}} = T_{\widetilde{w},0}$ where $0$ is the trivial representation.

The restriction map gives a map $N_K(\mathfrak{h})/Z_K(\mathfrak{h})\to W$ and its kernel is isomorphic to $N_{M_0}(\mathfrak{t}_0)/Z_{M_0}(\mathfrak{t}_0)$ (Recall that $\mathfrak{t}_0$ is a Cartan subalgebra of $\mathfrak{m}_0$).
The last group is isomorphic to $\widetilde{W_{M_0}}$.

\begin{lem}\label{lem:good lift for W}
Let $w\in W$.
Then there exists $\iota(w)\in N_K(\mathfrak{h})$ such that $\Ad(\iota(w))|_{\mathfrak{a}_0} = w$ and $\Ad(\iota(w))(\Delta_{M_0}^+) = \Delta_{M_0}^+$.
\end{lem}
\begin{proof}
Since $W\simeq N_K(\mathfrak{a}_0)/Z_K(\mathfrak{a}_{0})$, there exists $k\in N_K(\mathfrak{a}_{0})$ such that $\Ad(k)|_{\mathfrak{a}_0} = w$.
Then $k$ normalizes $M_0$.
Hence, there exists $m\in M_0$ such that $km$ normalizes $T_0$.
This implies $km\in N_K(A_0T_0)$.
Take $w'\in N_{M_0}(\mathfrak{t}_0)$ such that $\Ad(kmw')(\Delta^+_{M_0}) = \Delta^+_{M_0}$ and put $\iota(w) = kmw'$.
Then $\iota(w)$ satisfies the conditions of the lemma.
\end{proof}

The map $\iota$ gives an injective map $W\to N_K(\mathfrak{h})/Z_K(\mathfrak{h})$.
Since the group $N_K(\mathfrak{h})/Z_K(\mathfrak{h})$ can be regard as a subgroup of $\widetilde{W}$, we can regard $W$ as a subgroup of $\widetilde{W}$.
Hence, we can define the twisting functor $T_{w,\psi}$ for $w\in W$ and the character $\psi$ of $\Ad(w)\overline{\mathfrak{n}_0}\cap \mathfrak{n}_0$.
For a simplicity, we write $w$ instead of $\iota(w)$. (We regard $W$ as a subgroup of $\widetilde{W}$ by $\iota$.)

\begin{prop}\label{prop:decomposition of T_ww'}
Let $w,w'\in W$ and $\psi$ a character of $\Ad(ww')\overline{\mathfrak{n}_0}\cap \mathfrak{n}_0$.
Assume that $\ell(w) + \ell(w') = \ell(ww')$ where $\ell(w)$ is the length of $w\in W$.
Then we have $T_{w,\psi}T_{w',w^{-1}\psi} = T_{ww',\psi}$.
\end{prop}
\begin{proof}
By the assumption, we have $\Sigma^+\cap ww'\Sigma^- = (\Sigma^+\cap w\Sigma^-)\cup w(\Sigma^+\cap w'\Sigma^-)$.
Put $\Delta_0^\pm = \Delta^\pm\setminus\Delta_{M_0}^\pm$.
Then we have $\Delta_0^+\cap ww'\Delta_0^- = (\Delta_0^+\cap w\Delta_0^-)\cup w(\Delta_0^+\cap w'\Delta_0^-)$.
Since $w\Delta_{M_0}^\pm = \Delta_{M_0}^\pm$, we have $\Delta_0^+\cap w\Delta_0^- = \Delta^+\cap w\Delta^-$.
Hence, $\Delta^+\cap ww'\Delta^- = (\Delta^+\cap w\Delta^-)\cup w(\Delta^+\cap w'\Delta^-)$.
This implies that $\widetilde{\ell}(w) + \widetilde{\ell}(w') = \widetilde{\ell}(ww')$ where $\widetilde{\ell}(w)$ is the length of $w$ as an element of $\widetilde{W}$.
Hence, the proposition follows from the construction of the twisting functor (See Andersen and Lauritzen~\cite[Remark~6.1 (ii)]{MR1985191}).
\end{proof}

\begin{lem}\label{lem:Xe^-k}
Let $e$ be a nilpotent element of $\mathfrak{g}$, $X\in \mathfrak{g}$ and $k\in\Z_{\ge 0}$.
For $c\in \C$ we have the following equation in $U(\mathfrak{g})_{e - c}$.
\[
	X(e - c)^{-(k + 1)} = \sum_{n = 0}^\infty \binom{n + k}{k}(e - c)^{-(n + k + 1)}\ad (e)^n(X).
\]
\end{lem}
\begin{proof}
We prove the lemma by induction on $k$.
If $k = 0$, then the lemma is well-known.
Assume that $k > 0$.
Then we have
\begin{align*}
X(e - c)^{-(k + 1)} & = \sum_{k_0 = 0}^\infty (e - c)^{-(k_0 + 1)}\ad(e)^{k_0}(X) (e - c)^{-k}\\
& = \sum_{k_0 = 0}^\infty \sum_{k_1 = 0}^\infty \binom{k_1 + k - 1}{k - 1}(e - c)^{-(k_0 + k_1 + k + 1)}\ad(e)^{k_0 + k_1}(X)\\
& = \sum_{n = 0}^\infty \sum_{l' = 0}^n \binom{l' + k - 1}{k - 1}(e - c)^{-(n + k + 1)}\ad(e)^n(X)\\
& = \sum_{n = 0}^\infty \binom{n + k}{k}(e - c)^{-(n + k + 1)}\ad(e)^n(X).
\end{align*}
This proves the lemma.
\end{proof}

\section{The module $I_i/I_{i - 1}$}\label{sec:The module I_i/I_i-1}
Put $J_i = U(\mathfrak{g})\otimes_{U(\mathfrak{p})}J'_{w_i^{-1}\eta}(\sigma\otimes e^{\lambda+\rho})$\newsym{$J_i$}, where $\mathfrak{n}$ acts $J'_{w_i^{-1}\eta}(\sigma\otimes e^{\lambda+\rho})$ trivially.
In this section, we prove the following theorem.
\begin{thm}\label{thm:structure of I_i/I_{i - 1}}
Assume that $I_i/I_{i - 1} \ne 0$.
Then we have $I_i/I_{i - 1} \simeq T_{w_i,\eta}J_i$.
\end{thm}

Notice that $\mathfrak{u}_{0,{w_i}} = \Ad(w_i)\overline{\mathfrak{n}}\cap \mathfrak{n}_0$ since $w_i(\Delta_{M}^+) \subset \Delta^+$.
In this section fix $i\in\{1,\dots,l\}$ and a basis $\{e_1,e_2,\dots,e_l\}$ of $\mathfrak{u}_{0,w_i}$ such that each vector $e_i$ is a root vector and $\bigoplus_{s\le t-1}\C e_s$ is an ideal of $\bigoplus_{s\le t}\C e_s$.
Let $\alpha_s$ be the restricted root with respect to $e_s$.
As in Section~\ref{sec:vanishing theorem}, for $\mathbf{k} = (k_1,\dots,k_l)\in\Z_{\ge 0}^l$ we denote $\ad(e_l)^{k_l}\dotsm \ad(e_1)^{k_1}$ by $\ad(e)^\mathbf{k}$ and $((-x_1)^{k_1}/k_1!)\dotsm((-x_l)^{k_l}/k_l!)$ by $f_\mathbf{k}$.

\begin{lem}\label{lem:left2right for I'_i}
We have
\[
	I'_i = \left\{\sum_{s = 1}^t \delta_i(T_s,f_s\eta_i^{-1},u'_s)\Bigm|
	\begin{array}{ll}
	T_s\in U(\Ad(w_i)\overline{\mathfrak{n}}\cap \mathfrak{n}_0),&
	f_s\in \mathcal{P}(O_i),\\
	u'_s\in J'_{w_i^{-1}\eta}(\sigma\otimes(\lambda + \rho))
	\end{array}
	\right\}.
\]
\end{lem}
\begin{proof}
By Lemma~\ref{lem:left2right}, we have
\[
	T((f\otimes u')\delta_i) = 
	\sum_{\mathbf{k}\in\Z_{\ge 0}^l}\delta_i(\ad(e)^\mathbf{k}T,ff_\mathbf{k},u')
\]
for $T\in U(\mathfrak{g})$, $f\in \mathcal{P}(O_i)\eta_i^{-1}$ and $u'\in \sigma'$.
Hence, the left hand side is a subset of the right hand side.
Define $f_\mathbf{k}'\in\mathcal{P}(O_i)$ by $f_\mathbf{k}' = (x_1^{k_1}/k_1!)\dotsm(x_l^{k_l}/k_l!)$.
By the similar calculation of Lemma~\ref{lem:left2right}, we have
\[
	\delta_i(T,f,u') = 
	\sum_{\mathbf{k}\in\Z_{\ge 0}^l}(\ad(e)^{\mathbf{k}}T)(((ff_\mathbf{k}')\otimes u')\delta_i).
\]
This implies that the right hand side is contained in the left hand side.
\end{proof}

By the definition of the twisting functor and Poincar\"e-Birkhoff-Witt theorem, we have the following lemma.
For $\mathbf{k} = (k_1,\dots,k_l)\in\Z^l$ put $(e - \eta(e))^\mathbf{k} = (e_1 - \eta(e_1))^{k_1}\dotsm (e_l - \eta(e_l))^{k_l}\in S_{w,\eta}$.
Set $\mathbf{1} = (1,\dots,1)\in\Z^l$.
\begin{lem}\label{lem:induction + twisting}
Let $V$ be a $\mathfrak{p}$-module.
Then we have a $\C$-vector space isomorphism
\begin{multline*}
\left(\sum_{\mathbf{k}\in\Z_{\ge 0}^l}\C (e - \eta(e))^{-(\mathbf{k} + \mathbf{1})}\right)\otimes_{U(\Ad(w_i)\overline{\mathfrak{n}}\cap\mathfrak{n}_0)}U(\mathfrak{g})\otimes_{U(\Ad(w_i)\mathfrak{p})}w_iV\\
\simeq
T_{w_i}(U(\mathfrak{g})\otimes_{U(\mathfrak{p})}V)
\end{multline*}
given by $E\otimes T\otimes v\mapsto ET\otimes (1\otimes v)$.
(Notice that $ET\in S_{w_i,0}$.)
\end{lem}

\begin{proof}[Proof of Theorem~\ref{thm:structure of I_i/I_{i - 1}}]
By Lemma~\ref{lem:left2right for I'_i}, we have an isomorphism as a vector space,
\[
	I'_i \simeq 
	\mathcal{P}(O_i)
	\otimes_{U(\Ad(w_i)\overline{\mathfrak{n}}\cap \mathfrak{n}_0)} U(\mathfrak{g})
	\otimes_{U(\Ad(w_i)\mathfrak{p})} w_iJ'_{w_i^{-1}\eta}(\sigma\otimes(\lambda + \rho))
\]
given by $\delta_i(T,f,u')\mapsto f\otimes T\otimes u'$.

Notice that $\mathfrak{u}_{0,w_i} = \Ad(w_i)\overline{\mathfrak{n}}\cap \mathfrak{n}_0$ since $w_i\in W(M)$.
By Lemma~\ref{lem:induction + twisting}, we have
\begin{multline*}
T_{w_i,\eta}(J_i) \simeq 
\left(\sum_{\mathbf{k}\in\Z_{\ge 0}^l} \C (e - \eta(e))^{-(\mathbf{k}+\mathbf{1})}\right)
\otimes_{U(\Ad(w_i)\overline{\mathfrak{n}}\cap \mathfrak{n}_0)} U(\mathfrak{g})\\
\otimes_{U(\Ad(w_i)\mathfrak{p})} w_iJ'_{w_i^{-1}\eta}(\sigma\otimes(\lambda + \rho)).
\end{multline*}
Here we remark $\sum_{\mathbf{k}\in\Z_{\ge 0}^l} \C (e - \eta(e))^{-(\mathbf{k}+\mathbf{1})}$ is an $\Ad(w_i)\overline{\mathfrak{n}}\cap\mathfrak{n}_0$-stable subspace of $S_{w_i,\eta}$.
Hence, we can define a $\C$-vector space isomorphism $\Phi\colon T_{w_i,\eta}(J_i)\to I'_i$ by
\[
	\Phi((e - \eta(e))^{-(\mathbf{k}+\mathbf{1})}\otimes T\otimes u')
	=
	\delta_i(T,f_\mathbf{k}\eta_i^{-1},u').
\]
We prove that $\Phi$ is a $\mathfrak{g}$-homomorphism.

Fix $X\in \mathfrak{g}$.
We prove that
\[
\Phi(X((e - \eta(e))^{-(\mathbf{k}+\mathbf{1})}\otimes T\otimes u'))
= X\Phi((e - \eta(e))^{-(\mathbf{k}+\mathbf{1})}\otimes T\otimes u').
\]
By Lemma~\ref{lem:Xe^-k}, we have
\begin{multline*}
	X((e - \eta(e))^{-(\mathbf{k}+\mathbf{1})}\otimes T\otimes u')\\
	= \sum_{p_s\ge 0}\binom{p_1 + k_1}{k_1}\dotsm \binom{p_l + k_l}{k_l}(e - \eta(e))^{-(\mathbf{k}+\mathbf{p}+\mathbf{1})}\otimes(\ad(e)^\mathbf{p}X)T\otimes u'.
\end{multline*}
where $\mathbf{p} = (p_1,\dots,p_l)$.
Hence, we have
\begin{multline*}
	\Phi(X((e - \eta(e))^{-(\mathbf{k}+\mathbf{1})}\otimes T\otimes u'))\\
	= \sum_{p_s\ge 0}\delta_i\left((\ad(e)^\mathbf{p}X)T,\left(\frac{(-x_1)^{k_1 + p_1}}{k_1!p_1!}\dotsm\frac{(-x_l)^{k_l + p_l}}{k_l!p_l!}\right)\eta_i^{-1}, u'\right).
\end{multline*}
By Lemma~\ref{lem:left2right}, we have
\begin{align*}
X\Phi((e - \eta(e))^{-(\mathbf{k} + 1)}\otimes T\otimes u')
& = X\delta_i(T,f_\mathbf{k}\eta_i^{-1},u')\\
& = \sum_{\mathbf{p}\in\Z_{\ge 0}^l}\delta_i((\ad(e)^{\mathbf{p}}X)T,f_\mathbf{k}f_\mathbf{p}\eta_i^{-1},u').
\end{align*}
Hence, we have the theorem.
\end{proof}

\section{The module $J^*_\eta(I(\sigma,\lambda))$}\label{sec:the module J^*_eta(I(sigma,lambda))}
Now we investigate a module $J^*_\eta(I(\sigma,\lambda))$.
For a finite-length Fr\'echet representation $V$ of $G$, put $J(V) = (\varprojlim_{k\to \infty}(V_{\text{$K$-finite}}/\mathfrak{n}_0^kV_{\text{$K$-finite}}))_{\text{$\mathfrak{a}$-finite}}$\newsym{$J(V)$}.
This is also called the Jacquet module of V~\cite{MR562655}.
Define a category $\mathcal{O}'_{P_0}$\newsym{$\mathcal{O}'_{P_0}$} by the full subcategory of finitely generated $\mathfrak{g}$-modules consisting an object $V$ satisfying the following conditions.
\begin{enumerate}
\item The action of $\mathfrak{p}_0$ is locally finite. (In particular, the action of $\mathfrak{n}_0$ is locally nilpotent.)
\item The module $V$ is $Z(\mathfrak{g})$-finite.
\item The group $M_0$ acts on $V$ and its differential coincides with the action of $\mathfrak{m}_0\subset \mathfrak{g}$.
\item For $\nu\in\mathfrak{a}_0^*$ let $V_\nu$ be the generalized $\mathfrak{a}_0$-weight space with weight $\nu$.
Then $V = \bigoplus_{\nu\in\mathfrak{a}_0^*}V_\nu$ and $\dim V_\nu < \infty$.
\end{enumerate}
We define the category $\mathcal{O}_{\overline{P_0}}'$ similarly.\newsym{$\mathcal{O}'_{\overline{P_0}}$}
Then for a finite-length Fr\'echet representation $V$ of $G$ we have $J(V)\in\mathcal{O}_{\overline{P_0}}'$ and $J^*(V)\in\mathcal{O}_{P_0}'$.
For a $U(\mathfrak{g})$-module $V$, put $D'(V) = (V^*)_{\text{$\mathfrak{h}$-finite}}$ and $C(V) = (D'(V))^*$.\newsym{$D'(V)$}\newsym{$C(V)$}
Denote a full-subcategory of $\mathfrak{g}$-modules consisting finitely-generated and locally $\mathfrak{h}\oplus\mathfrak{u}$-finite modules by $\mathcal{O}'$\newsym{$\mathcal{O}'$}.
If $V$ is an object of the category $\mathcal{O}'$ then $D'D'(V) \simeq V$.
The relation between $J^*$ and $J$ is as follows.
\begin{prop}\label{prop:relation J^* and J}
Let $V$ be a finite-length Fr\'echet representation of $G$.
Then we have $J^*(V) \simeq D'(J(V))$.
\end{prop}

The character $\eta\colon \mathfrak{n}_0\to \C$ defines an algebra homomorphism $U(\mathfrak{n}_0)\to \C$ by the universality of the universal enveloping algebra.
Let $\Ker\eta$ be the kernel of this algebra homomorphism and put $\Gamma_\eta(V) = \{v\in V\mid \text{for some $k$, $(\Ker\eta)^kv = 0$}\}$.\newsym{$\Gamma_\eta(V)$}
First we prove the following proposition.
\begin{prop}\label{prop:relation J^*_eta and J_eta}
Let $V$ be a finite-length Fr\'echet representation of $G$.
Then we have $J^*_\eta(V) \simeq \Gamma_\eta(J(V)^*)$.
\end{prop}
\begin{proof}
Recall that $\mathfrak{p}_\eta = \mathfrak{m}_\eta\oplus\mathfrak{a}_\eta\oplus\mathfrak{n}_\eta$ is the complexification of the Lie algebra of the parabolic subgroup corresponding to $\supp\eta$ (Section~\ref{sec:Analytic continuation}).
If $\supp\eta = \Pi$, this proposition is proved by Matumoto~\cite[Theorem~4.9.2]{MR1047117}.

Put $I = V_{\text{$K$-finite}}$.
Let $\eta_0\colon U(\mathfrak{m}\cap \mathfrak{n}_0)\to \C$ be the restriction of $\eta$ on $U(\mathfrak{m}\cap \mathfrak{n}_0)$.
Then we have 
\[
J^*_\eta(V) = \varinjlim_{k,l}(I/\mathfrak{n}_\eta^l(\Ker\eta_0)^kI)^* = \varinjlim_{k,l}((I/\mathfrak{n}_\eta^lI)/(\Ker\eta_0)^k(I/\mathfrak{n}_\eta^lI))^*.
\]
For a $U(\mathfrak{g})$-module $V_0$, put $G(V_0) = (\varprojlim_k V_0/\mathfrak{n}_0^kV_0)_{\text{$\mathfrak{a}$-finite}}$.
For a $U(\mathfrak{m}_\eta\oplus\mathfrak{a}_\eta)$-module $V_1$, put $G_{M_\eta}(V_1) = (\varprojlim_k V_1/(\mathfrak{m}_\eta\cap\mathfrak{n}_0)^kV_1)_{\textrm{$(\mathfrak{m}\cap\mathfrak{a}_0)$-finite}}$.
Since $I/\mathfrak{n}_\eta^lI$ is a Harish-Chandra module of $\mathfrak{m}_\eta\oplus\mathfrak{a}_\eta$, $J^*_{\eta_0}(I/\mathfrak{n}_\eta^lI) = \Gamma_{\eta_0}(G_{M_\eta}(I/\mathfrak{n}_\eta^lI)^*)$ by the result of Matumoto.
Taking a subspace annihilated by $(\Ker\eta_0)^k$, we have 
\[
((I/\mathfrak{n}_\eta^lI)/(\Ker\eta_0)^k(I/\mathfrak{n}_\eta^lI))^* = (G_{M_\eta}(I/\mathfrak{n}_\eta^lI)/(\Ker\eta_0)^kG_{M_\eta}(I/\mathfrak{n}_\eta^lI))^*.
\]
Since $I$ is a finitely-generated $U(\mathfrak{n}_0)$-module, the left hand side is finite-dimensional.
Hence, we have 
\[
(I/\mathfrak{n}_\eta^lI)/(\Ker\eta_0)^k(I/\mathfrak{n}_\eta^lI) = G_{M_\eta}(I/\mathfrak{n}_\eta^lI)/(\Ker\eta_0)^kG_{M_\eta}(I/\mathfrak{n}_\eta^lI).
\]
It is sufficient to prove that $G_{M_\eta}(I/\mathfrak{n}_\eta^lI) = G(I)/\mathfrak{n}_\eta^lG(I)$.
We have 
\[
(I/\mathfrak{n}_\eta^lI)/(\mathfrak{m}_\eta\cap\mathfrak{n}_0)^k(I/\mathfrak{n}_\eta^lI) = I/(\mathfrak{m}_\eta\cap\mathfrak{n}_0)^k\mathfrak{n}_\eta^lI = G(I)/(\mathfrak{m}_\eta\cap\mathfrak{n}_0)^k\mathfrak{n}_\eta^lG(I).
\]
Taking the projective limit we have $G_{M_\eta}(I/\mathfrak{n}_\eta^lI) = G_{M_\eta}(G(I)/\mathfrak{n}_\eta^lG(I))$.
Since $G(I)/\mathfrak{n}_\eta^lG(I)\in \mathcal{O}'_{M_\eta\cap\overline{P_0}}$ we have $G_{M_\eta}(G(I)/\mathfrak{n}_\eta^lG(I)) = G(I)/\mathfrak{n}_\eta^lG(I)$.
\end{proof}

Combining Theorem~\ref{thm:structure of I_i/I_{i - 1}}, Proposition~\ref{prop:relation J^*_eta and J_eta} and the automatic continuation theorem~\cite[Theorem~4.8]{MR727854}, we have the following theorem.

\begin{thm}\label{thm:stucture of J^*(I(sigma,lambda))}
There exists a filtration $0 = \widetilde{I_1}\subset\cdots \subset \widetilde{I_r} = J^*_\eta(I(\sigma,\lambda))$ such that $\widetilde{I_i}/\widetilde{I_{i - 1}} \simeq \Gamma_\eta(C(T_{w_i}(U(\mathfrak{g})\otimes_{U(\mathfrak{p})} J^*(\sigma\otimes e^{\lambda+\rho}))))$.\newsym{$\widetilde{I_i}$}
\end{thm}

\section{Whittaker vectors}\label{sec:Whittaker vectors}
In this section we study Whittaker vectors of $I(\sigma,\lambda)'$ and $(I(\sigma,\lambda)_{\text{$K$-finite}})^*$ (Definition~\ref{defn:Whittaker vectors}).

For $i$ such that $I_i/I_{i - 1}\ne 0$, we define some maps as follows.
Let $\gamma_1$ be the first projection with respect to the decomposition $U(\mathfrak{g}) = U(\mathfrak{l}_\eta)\oplus(\overline{\mathfrak{n}_\eta}U(\mathfrak{g}) + U(\mathfrak{g})\mathfrak{n}_\eta)$.
Notice that by Lemma~\ref{lem:vanishing lemma} if $I_i/I_{i - 1}\ne 0$ then we have $\mathfrak{l}_\eta\cap \Ad(w_i)\overline{\mathfrak{n}}\subset \mathfrak{n}_0$.
Define $\gamma_2$ by the first projection with respect to the decomposition $U(\mathfrak{l}_\eta) = U(\mathfrak{l}_\eta\cap \Ad(w_i)\mathfrak{p}) \oplus U(\mathfrak{l}_\eta)\Ker\eta|_{\mathfrak{l}_\eta\cap \Ad(w_i)\overline{\mathfrak{n}}}$.
Let $\gamma_3$ be the first projection with respect to the decomposition $U(\mathfrak{l}_\eta\cap \Ad(w_i)\mathfrak{p}) = U(\mathfrak{l}_\eta\cap \Ad(w_i)\mathfrak{l})\oplus (\mathfrak{l}_\eta\cap \Ad(w_i)\mathfrak{n})U(\mathfrak{l}_\eta\cap \Ad(w_i)\mathfrak{p})$.
Finally define $\gamma_4$ by the first projection with respect to the decomposition $U(\mathfrak{l}_\eta\cap \Ad(w_i)\mathfrak{l}) = U(\mathfrak{h})\oplus((\overline{\mathfrak{u}_0}\cap\mathfrak{l}_\eta\cap \Ad(w_i)\mathfrak{l})U(\mathfrak{l}_\eta\cap \Ad(w_i)\mathfrak{l}) + U(\mathfrak{l}_\eta\cap \Ad(w_i)\mathfrak{l})(\mathfrak{l}_\eta\cap \Ad(w_i)\mathfrak{l} \cap \mathfrak{u}_0))$.
Then the restriction of $\gamma_4\circ\gamma_3\circ\gamma_2\circ\gamma_1$ on $Z(\mathfrak{g})$ is the (non-shifted) Harish-Chandra homomorphism.
If $x\in \Wh_\eta(I_i/I_{i - 1})$ then $Tx = \gamma_2\gamma_1(T)x$ for $T\in Z(\mathfrak{g})$.
\begin{align*}
\gamma_1\colon & U(\mathfrak{g}) = U(\mathfrak{l}_\eta)\oplus(\overline{\mathfrak{n}_\eta}U(\mathfrak{g}) + U(\mathfrak{g})\mathfrak{n}_\eta)\to U(\mathfrak{l}_\eta),\\
\gamma_2\colon & U(\mathfrak{l}_\eta) = U(\mathfrak{l}_\eta\cap \Ad(w_i)\mathfrak{p}) \oplus U(\mathfrak{l}_\eta)\Ker\eta|_{\mathfrak{l}_\eta\cap \Ad(w_i)\overline{\mathfrak{n}}}\to U(\mathfrak{l}_\eta\cap \Ad(w_i)\mathfrak{p}),\\
\gamma_3\colon & U(\mathfrak{l}_\eta\cap \Ad(w_i)\mathfrak{p}) = U(\mathfrak{l}_\eta\cap \Ad(w_i)\mathfrak{l})\oplus (\mathfrak{l}_\eta\cap \Ad(w_i)\mathfrak{n})U(\mathfrak{l}_\eta\cap \Ad(w_i)\mathfrak{p})
\tag*{$\to U(\mathfrak{l}_\eta\cap \Ad(w_i)\mathfrak{l})$},\\
\gamma_4\colon & U(\mathfrak{l}_\eta\cap \Ad(w_i)\mathfrak{l}) 
= U(\mathfrak{h})\oplus((\overline{\mathfrak{u}_0}\cap\mathfrak{l}_\eta\cap \Ad(w_i)\mathfrak{l})U(\mathfrak{l}_\eta\cap \Ad(w_i)\mathfrak{l})
\tag*{$+ U(\mathfrak{l}_\eta\cap \Ad(w_i)\mathfrak{l})(\mathfrak{l}_\eta\cap \Ad(w_i)\mathfrak{l} \cap \mathfrak{u}_0))\to U(\mathfrak{h}).$}
\end{align*}
\newsym{$\gamma_1,\gamma_2,\gamma_3,\gamma_4$}
\begin{lem}\label{lem:lemma of infinitesimal character}
Let $V$ be a $U(\mathfrak{g})$-module with an infinitesimal character $\widetilde{\lambda}$, $\chi$ a character of $Z(\mathfrak{g})$ such that $z\in Z(\mathfrak{g})$ acts by $\chi(z)$ on $V$.
Take a nonzero element $v\in V$ such that $(\gamma_3\gamma_2\gamma_1(z) - \chi(z))v = 0$.
Moreover, assume that there exists $\mu\in\mathfrak{a}^*$ such that $Hv = (w_i\mu+\rho_0)(H)v$ for all $H\in \Ad(w_i)\mathfrak{a}$.
Then there exists $\widetilde{w}\in \widetilde{W}$ such that $\widetilde{w}\widetilde{\lambda}|_\mathfrak{a} = \mu$.
\end{lem}
\begin{proof}
Put $Z = (\gamma_3\gamma_2\gamma_1(Z(\mathfrak{g}))U(\Ad(w_i)\mathfrak{a}))$.
By the assumption, there exists a character $\chi_0$ of $Z$ such that $zv = \chi_0(z)v$ for all $z\in Z$.
By a theorem of Harsh-Chandra, $\gamma_4|_Z$ is injective and $\gamma_4(Z)\subset U(\mathfrak{h})$ is finite.
Hence there exists an element $\widetilde{\lambda_1}\in \mathfrak{h}^*$ such that $\widetilde{\lambda_1}\circ \gamma_4 = \chi_0$ where we denote the algebra homomorphism $U(\mathfrak{h})\to \C$ induced from $\widetilde{\lambda_1}$ by the same letter $\widetilde{\lambda_1}$.
Since $V$ has an infinitesimal character $\widetilde{\lambda}$, we have $\widetilde{\lambda_1}\in \widetilde{W}\widetilde{\lambda} + \widetilde{\rho}$.
Since $\gamma_4$ is trivial on $U(\Ad(w_i)\mathfrak{a})$, $\widetilde{\lambda_1}|_{\Ad(w_i)\mathfrak{a}} = (w_i\mu + \rho_0)|_{\Ad(w_i)\mathfrak{a}}$.
The restriction of $\widetilde{\rho}$ to $\mathfrak{a}_0$ is $\rho_0$.
Hence $\widetilde{\rho}|_{\Ad(w_i)\mathfrak{a}} = \rho_0|_{\Ad(w_i)\mathfrak{a}}$.
Then for some $\widetilde{w}\in\widetilde{W}$ we have $w_i\mu|_{\Ad(w_i)\mathfrak{a}} = \widetilde{w}\widetilde{\lambda}|_{\Ad(w_i)\mathfrak{a}}$.
We get the lemma.
\end{proof}

\begin{lem}\label{lem:to outside delta_i}
Let $X_1,\dots,X_n\in \mathfrak{g}$, $f_1\in C^\infty(O_i)$, $f_2\in C^\infty(U_i)$, $u'\in (\sigma\otimes e^{\lambda + \rho})'$.
Assume that $\widetilde{R'_i}(X_s)(f_2) = 0$ for all $s = 1,\dots,n$.
Then we have
\[
	\delta_i(X_1\dotsm X_n,f_1f_2,u') = \delta_i(X_1\dotsm X_n,f_1,u')f_2.
\]
\end{lem}
\begin{proof}
Put $T = X_1\dotsm X_n$.
By the assumption and Leibniz's rule, we have
\[
	f_2(nw_i)(\widetilde{R_i}(T)\varphi)(nw_i) = (\widetilde{R_i}(T)(\varphi f_2))(nw_i)
\]
Hence, by the definition, for $\varphi\in C^\infty_c(U_i,\mathcal{L})$, we have
\begin{align*}
	& \langle \delta_i(T,f_1f_2,u'),\varphi\rangle\\
	& =	\int_{w_i\overline{N}w_i^{-1}\cap N_0}f_1(nw_i)f_2(nw_i)(u'(\widetilde{R}_i(T)\varphi)(nw_i))dn\\
	& = \int_{w_i\overline{N}w_i^{-1}\cap N_0}f_1(nw_i)(u'(\widetilde{R}_i(T)(\varphi f_2))(nw_i))dn\\
	& = \langle\delta_i(T,f_1,u'),f_2\varphi\rangle\\
	& = \langle\delta_i(T,f_1,u')f_2,\varphi\rangle.
\end{align*}
We get the lemma.
\end{proof}

\begin{lem}\label{lem:property of V(nu)}
For $\nu\in \mathfrak{a}^*$ put
\[
	V(\nu) = \left\{\sum_s \delta_i(S_s,h_s\eta_i^{-1},v_s')\Biggm|
	\begin{array}{l}
	S_s\in U(\Ad(w_i)\overline{\mathfrak{n}}\cap \overline{\mathfrak{n}_0}),\ h_s\in\mathcal{P}(O_i),\\ v'_s\in J'_{w_i^{-1}\eta}(\sigma\otimes e^{\lambda+\rho}),\\ w_i^{-1}(\wt h_s + \wt S_s)|_\mathfrak{a} = \nu
	\end{array}
	\right\}.
\]
Here, $\wt h_s$ is an $\mathfrak{a}_0$-weight of $h_s$ with respect to $D_i$ (see page~\pageref{symbol:D_i}) and $\wt S_s$ is an $\mathfrak{a}_0$-weight of $S_s$ with respect to the adjoint action.
Define $\widetilde{\eta_i}\in C^\infty(U_i)$ by $\widetilde{\eta_i}(nn_0w_iP/P) = \eta_i(n)$ for $n\in w_i\overline{N}w_i^{-1}\cap N_0$ and $n_0\in w_i\overline{N}w_i^{-1}\cap \overline{N_0}$.
\begin{enumerate}
\item Let $X\in U(\mathfrak{l}_\eta\cap \Ad(w_i)\mathfrak{p})$.
Assume that $X$ is an $\mathfrak{a}_0$-weight vector.
For $\delta_i(T,f\eta_i^{-1},u')\in V(\nu)$, we have
\[
	X\delta_i(T,f\eta_i^{-1},u') - (X\delta_i(T,f\eta_i^{-1},u'))\widetilde{\eta_i}^{-1}\in\sum_{\nu' > \nu}V(\nu' + w_i^{-1}\wt T|_\mathfrak{a}).
\]
here, $\wt T$ is an $\mathfrak{a}_0$-weight of $T$ with respect to the adjoint action.
\item For $\delta_i(S_s,h_s\eta_i^{-1},v'_s)\in V(\nu)$, we have
\[
	\sum_s \delta_i(S_s,h_s,v'_s)\widetilde{\eta_i}^{-1}\not\in \sum_{\nu' > \nu}V(\nu').
\]
\end{enumerate}
\end{lem}
\begin{proof}
(1)
Fix a basis $\{e_1,e_2,\dots,e_l\}$ of $\mathfrak{u}_{0,w_i}$ such that each vector $e_i$ is a root vector and $\bigoplus_{s\le t-1}\C e_s$ is an ideal of $\bigoplus_{s\le t}\C e_s$.
Let $\alpha_s$ be the restricted root of $e_s$.
As in Section~\ref{sec:vanishing theorem}, for $\mathbf{k} = (k_1,\dots,k_l)\in\Z_{\ge 0}^l$ we denote $\ad(e_l)^{k_l}\dotsm \ad(e_1)^{k_1}$ by $\ad(e)^\mathbf{k}$ and $((-x_1)^{k_1}/k_1!)\dotsm((-x_l)^{k_l}/k_l!)$ by $f_\mathbf{k}$.
By Lemma~\ref{lem:left2right}, 
\[
	X\delta_i(T,f\eta_i^{-1},u') = \sum_{\mathbf{k}\in\Z_{\ge 0}^l}\delta_i((\ad(e)^\mathbf{k}X)T,ff_\mathbf{k}\eta_i^{-1},u').
\]
Take $a_\mathbf{k}^{(p)}\in U(\Ad(w_i)\overline{\mathfrak{n}}\cap\mathfrak{n}_0)$, $b_\mathbf{k}^{(p)}\in U(\Ad(w_i)\overline{\mathfrak{n}}\cap \overline{\mathfrak{n}_0})$ and $c_\mathbf{k}^{(p)}\in U(\Ad(w_i)\mathfrak{p})$ such that $(\ad(e)^\mathbf{k}X)T = \sum_p a_\mathbf{k}^{(p)}b_\mathbf{k}^{(p)}c_\mathbf{k}^{(p)}$ and $\wt((\ad(e)^\mathbf{k}X)T) = \wt a_\mathbf{k}^{(p)} + \wt b_\mathbf{k}^{(p)} + \wt c_\mathbf{k}^{(p)}$.
Then
\begin{align*}
&\delta_i((\ad(e)^\mathbf{k}X)T,ff_\mathbf{k}\eta_i^{-1},u')\\
& = \sum_p \delta_i(a_\mathbf{k}^{(p)}b_\mathbf{k}^{(p)}c_\mathbf{k}^{(p)},ff_\mathbf{k}\eta_i^{-1},u')\\
& = \sum_p \delta_i(b_\mathbf{k}^{(p)},R'_i(-a_\mathbf{k}^{(p)})(ff_\mathbf{k}\eta_i^{-1}),\Ad(w_i)^{-1}(c_\mathbf{k}^{(p)})u')
\end{align*}
By the Leibniz rule, there exists a subset $\mathcal{A}^{(p)}_\mathbf{k} \subset \{(a',a'')\in U(\Ad(w_i)\overline{\mathfrak{n}}\cap\mathfrak{n}_0)^2\mid \wt a' + \wt a'' = \wt a_\mathbf{k}^{(p)},\ a''\not\in\C\}$ such that 
\begin{align*}
	& \delta_i(b_\mathbf{k}^{(p)},R'_i(-a_\mathbf{k}^{(p)})(ff_\mathbf{k}\eta_i^{-1}) - R'_i(-a_\mathbf{k}^{(p)})(ff_\mathbf{k})\eta_i^{-1},\Ad(w_i)^{-1}(c_\mathbf{k}^{(p)})u')\\
	& = \sum_{(a',a'')\in\mathcal{A}^{(p)}_\mathbf{k}}\delta_i(b_\mathbf{k}^{(p)},R'_i(a')(ff_\mathbf{k})R'_i(a'')(\eta_i^{-1}),\Ad(w_i)^{-1}c_\mathbf{k}^{(p)}u')\\
	& = 
	\sum_{(a',a'')\in\mathcal{A}^{(p)}_\mathbf{k}}-\eta(a'')\delta_i(b_\mathbf{k}^{(p)},R'_i(a')(ff_\mathbf{k})\eta_i^{-1},\Ad(w_i)^{-1}c_\mathbf{k}^{(p)}u')
\end{align*}
By the Poincar\'e-Birkhoff-Witt theorem, we have a direct decomposition $U(\Ad(w_i)\mathfrak{p}) = U(\Ad(w_i)\mathfrak{p})(\Ad(w_i)\mathfrak{n})\oplus U(\Ad(w_i)\mathfrak{l})$.
Hence we may assume that $c_\mathbf{k}^{(p)} \in U(\Ad(w_i)\mathfrak{p})(\Ad(w_i)\mathfrak{n})$ or $c_\mathbf{k}^{(p)} \in U(\Ad(w_i)\mathfrak{l})$.
If $c_\mathbf{k}^{(p)} \in U(\Ad(w_i)\mathfrak{p})(\Ad(w_i)\mathfrak{n})$ then this sum is equal to $0$.
If $c_\mathbf{k}^{(p)} \in U(\Ad(w_i)\mathfrak{l})$ then $w_i^{-1}\wt c_\mathbf{k}^{(p)}|_{\mathfrak{a}} = 0$.
Hence,
\begin{align*}
&w_i^{-1}(\wt b_\mathbf{k}^{(p)} + \wt(R'_i(a')ff_\mathbf{k}))|_{\mathfrak{a}}\\
& = w_i^{-1}(\wt c_\mathbf{k}^{(p)} + \wt b_\mathbf{k}^{(p)} + \wt a' + \wt f + \wt f_\mathbf{k})|_{\mathfrak{a}}\\
& = w_i^{-1}(\wt ((\ad(e)^{\mathbf{k}}X)T) + \wt f + \wt f_\mathbf{k} - \wt a'')|_{\mathfrak{a}}\\
& = w_i^{-1}(\wt X + \wt T + \wt f - \wt a'')|_{\mathfrak{a}}\\
& = \nu + w_i^{-1}(\wt X - \wt a'')|_{\mathfrak{a}} > \nu + w_i^{-1}\wt X|_{\mathfrak{a}}.
\end{align*}
So we have
\begin{multline*}
\delta_i(b_\mathbf{k}^{(p)},R'_i(-a_\mathbf{k}^{(p)})(ff_\mathbf{k}\eta_i^{-1}) - R'_i(-a_\mathbf{k}^{(p)})(ff_\mathbf{k})\eta_i^{-1},\Ad(w_i)^{-1}(c_\mathbf{k}^{(p)})u')\\
\in \sum_{\nu' > \nu}V(\nu' + w_i^{-1}\wt X|_\mathfrak{a}).
\end{multline*}
By the definition of $\widetilde{\eta_i}$, we have $\widetilde{R_i}(X')\widetilde{\eta_i} = 0$ for $X'\in \Ad(w_i)\overline{\mathfrak{n}}\cap \overline{\mathfrak{n}_0}$.
Hence by Lemma~\ref{lem:to outside delta_i}, we have 
\begin{multline*}
	\delta_i(b_\mathbf{k}^{(p)},R'_i(-a_\mathbf{k}^{(p)})(ff_\mathbf{k})\eta_i^{-1},\Ad(w_i)^{-1}(c_\mathbf{k}^{(p)})u')\\ = \delta_i(b_\mathbf{k}^{(p)},R'_i(-a_\mathbf{k}^{(p)})(ff_\mathbf{k}),\Ad(w_i)^{-1}(c_\mathbf{k}^{(p)})u')\widetilde{\eta_i}^{-1}
\end{multline*}
Hence, we have
\begin{align*}
&\sum_{\mathbf{k},p}\delta_i(b_\mathbf{k}^{(p)},R'_i(-a_\mathbf{k}^{(p)})(ff_\mathbf{k})\eta_i^{-1},\Ad(w_i)^{-1}(c_\mathbf{k}^{(p)})u')\\
& = \sum_{\mathbf{k},p}\delta_i(b_\mathbf{k}^{(p)},R'_i(-a_\mathbf{k}^{(p)})(ff_\mathbf{k}),\Ad(w_i)^{-1}(c_\mathbf{k}^{(p)})u')\widetilde{\eta_i}^{-1}\\
& = \sum_{\mathbf{k},p}\delta_i(a_\mathbf{k}^{(p)}b_\mathbf{k}^{(p)}c_\mathbf{k}^{(p)},(ff_\mathbf{k}),u')\widetilde{\eta_i}^{-1}\\
& = (X\delta_i(T,f,u'))\widetilde{\eta_i}^{-1}.
\end{align*}
We get (1).

(2)
Take $T = 1$ in (1).
Then we have
\[
	\delta_i(T,f\eta_i^{-1},u') - (\delta_i(T,f\eta_i^{-1},u'))\widetilde{\eta_i}^{-1}\in\sum_{\nu' > \nu}V(\nu').
\]
Hence if
\[
	\sum_s \delta_i(S_s,h_s,v'_s)\widetilde{\eta_i}^{-1}\in \sum_{\nu' > \nu}V(\nu')
\]
then
\[
	\delta_i(T,f\eta_i^{-1},u')\in \sum_{\nu' > \nu}V(\nu')
\]
However, by Lemma~\ref{lem:fundamental properties of delta_i} (3), we have $V(\nu)\cap \sum_{\nu'\ne \nu}V(\nu') = 0$.
This is a contradiction.
\end{proof}

\begin{prop}\label{prop:Whittaker vectors in a Bruaht cell}
Let $\widetilde{\mu}\in (\mathfrak{h}\cap \mathfrak{m})^*$ be an infinitesimal character of $\sigma$.
Assume that $I_i/I_{i - 1}\ne 0$ and for all $\widetilde{w}\in \widetilde{W}$, 
\[
	\lambda - \widetilde{w}(\lambda + \widetilde{\mu})|_\mathfrak{a}\not\in \Z_{\le 0}((\Sigma^+\setminus\Sigma_M^+)\cap w_i^{-1}\Sigma^+)|_\mathfrak{a}\setminus\{0\}.
\]
Then
\[
	\Wh_\eta(I_i') = 
	\{(\eta_i^{-1}\otimes u')\delta_i\mid u'\in \Wh_{w_i^{-1}\eta}((\sigma\otimes e^{\lambda+\rho})')\}.
\]
\end{prop}
\begin{proof}
Let $x = \sum_s\delta_i(T_s,f_s\eta_i^{-1},u_s')$ be an element of $\Wh_\eta(I_i')$ where $T_s\in U(\Ad(w_i)\overline{\mathfrak{n}}\cap \overline{\mathfrak{n}_0})$, $f_s\in\mathcal{P}(O_i)$ and $u'_s\in J'_{w_i^{-1}\eta}(\sigma\otimes e^{\lambda+\rho})$.
For $X\in \Ad(w_i)\overline{\mathfrak{n}}\cap\mathfrak{n}_0$, we have $(X - \eta(X))x = \sum_s \delta_i(T_s,(L_X - \eta(X))(f_s\eta_i^{-1}),u'_s) = \sum_s \delta_i(T_s,L_X(f_s)\eta_i^{-1},u'_s)$ by Lemma~\ref{lem:caluculation of Xdelta(1,f,u)}.
Hence, we may assume $f_s = 1$.

Let $z\in Z(\mathfrak{g})$.
Since $J'_\eta(I(\sigma,\lambda))$ has an infinitesimal character $-(\lambda+\widetilde{\mu})$, $I_i'$ has the same character.
Let $\chi(z)$ be a complex number such that $z$ acts by $\chi(z)$ on $I'_i$.
Take $T_s$ and $u'_s$ such that $T_s$ are $\mathfrak{a}_0$-weight vectors and lineally independent.
Let $\nu = \min\{w_i^{-1}\wt T_s|_{\mathfrak{a}}\}_s$.
Then by Lemma~\ref{lem:property of V(nu)} (1), we have
\begin{multline*}
	\chi(z)x = zx = \gamma_2\gamma_1(z)x \\\in \left(\gamma_3\gamma_2\gamma_1(z)\sum_{w_o^{-1}\wt T_s|_{\mathfrak{a}} = \nu}\delta_i(T_s,1,u_s')\right)\widetilde{\eta_i}^{-1} + \sum_{\nu' > \nu}V(\nu').
\end{multline*}
By Lemma~\ref{lem:property of V(nu)} (1) ($T = 1$), we have
\[
	x \in \sum_{w_i^{-1}\wt T_s|_{\mathfrak{a}} = \nu}\delta_i(T_s,1,u'_s)\widetilde{\eta_i}^{-1} + \sum_{\nu' > \nu}V(\nu').
\]
Hence we have
\[
	\left((\chi(z) - \gamma_3\gamma_2\gamma_1(z))\left(\sum_{w_i^{-1}\wt T_s|_{\mathfrak{a}} = \nu}\delta_i(T_s,1,u'_s)\right)\right)\widetilde{\eta_i}^{-1}\in \sum_{\nu' > \nu}V(\nu').
\]
By Lemma~\ref{lem:property of V(nu)} (2), we have $(\chi(z) - \gamma_3\gamma_2\gamma_1(z))\delta_i(T_s,1,u_s') = 0$ for all $s$ such that $w_i^{-1}\wt T_s|_{\mathfrak{a}} = \nu$.
By the same calculation as that of the proof of Lemma~\ref{lem:no delta part}, $H\delta_i(T_s,1_u,u_s') = (-w_i\lambda + \wt T_s + \rho_0)(H)\delta_i(T_s,1_u,u_s')$ for $H\in \Ad(w_i)\mathfrak{a}$.
By Lemma~\ref{lem:lemma of infinitesimal character}, there exists a $\widetilde{w}\in\widetilde{W}$ such that $-\widetilde{w}(\lambda+\widetilde{\mu})|_{\Ad(w_i)\mathfrak{a}} = -w_i\lambda + \wt T_s$.
Then $\lambda - w_i^{-1}\widetilde{w}(\lambda+\widetilde{\mu})|_\mathfrak{a} = w_i^{-1}\wt T_s|_\mathfrak{a}\in\Z_{\le 0}((\Sigma^+\setminus\Sigma_M^+)\cap w_i^{-1}\Sigma^+)|_\mathfrak{a}$.
By the assumption, $\wt T_s = 0$, i.e., $T_s\in \C$.
Hence, we may assume that $x$ has a form $x = \delta_i(1,\eta_i^{-1},u') + \sum_{s\ge 2} \delta_i(T_s,\eta_i^{-1},u_s')$ where $\wt T_s \ne 0$ for all $s\ge 2$.

Take $X\in \mathfrak{n}_0\cap \Ad(w_i)\mathfrak{m}$.
Then by Lemma~\ref{lem:caluculation of Xdelta(1,f,u)} and the above claim,
\[
	0 = (X - \eta(X))x \in \delta_i(1,\eta_i^{-1},(\Ad(w_i)^{-1}X - \eta(X))u') + \sum_{\nu' > 0}V(\nu').
\]
By Lemma~\ref{lem:property of V(nu)}, we have $\delta_i(1,\eta_i^{-1},(\Ad(w_i)^{-1}X - \eta(X))u') = 0$.
Hence we have $u'\in \Wh_{w_i^{-1}\eta}((\sigma\otimes e^{\lambda + \rho})')$.
This implies that $x - \delta_i(1,\eta_i^{-1},u')\in \Wh_\eta(I_i')$.
If $x - \delta_i(1,\eta_i^{-1},u')\ne 0$, then by the above argument, we have $\min\{w_i^{-1}\wt T_s|_\mathfrak{a}\}_{s\ge 2} = 0$.
This is a contradiction.
\end{proof}

\begin{thm}\label{thm:dimension Whittaker vectors}
Assume that for all $w\in W(M)$ such that $\eta|_{wNw^{-1}\cap N_0} = 1$ the following two conditions hold:
\begin{enumerate}
\renewcommand*{\labelenumi}{(\alph{enumi})}
\item For each exponent $\nu$ of $\sigma$ and $\alpha\in \Sigma^+\setminus w^{-1}(\Sigma^+\cup\Sigma_\eta^-)$, we have $2\langle\alpha,\lambda+\nu\rangle/\lvert\alpha\rvert^2\not\in\Z_{\le 0}$.
\item For all $\widetilde{w}\in\widetilde{W}$ we have $\lambda - \widetilde{w}(\lambda + \widetilde{\mu})|_\mathfrak{a}\notin \Z_{\le 0}((\Sigma^+\setminus \Sigma_M^+)\cap w^{-1}\Sigma^+)|_\mathfrak{a}\setminus\{0\}$ where $\widetilde{\mu}$ is an infinitesimal character of $\sigma$.
\end{enumerate}
Moreover, assume that $\eta$ is unitary.
Then we have
\[
	\dim\Wh_\eta(I(\sigma,\lambda)') = \sum_{w\in W(M),\  w(\Sigma^+\setminus\Sigma^+_M)\cap \supp\eta = \emptyset}\dim \Wh_{w^{-1}\eta}((\sigma\otimes e^{\lambda+\rho})').
\]
\end{thm}
\begin{proof}
By the exact sequence $0\to I_{i - 1} \to I_i\to I_i/I_{i - 1}\to 0$, we have $0\to \Wh_\eta(I_{i - 1}) \to \Wh_\eta(I_i)\to \Wh_\eta(I_i/I_{i - 1})$.
By Lemma~\ref{prop:Whittaker vectors in a Bruaht cell}, it is sufficient to prove that the last map $\Wh_\eta(I_i)\to \Wh_\eta(I_i/I_{i - 1})$ is surjective.

Take $x\in \Wh_\eta(I'_i)\simeq \Wh_\eta(I_i/I_{i - 1})$.
Then $x$ is $(\eta_i\otimes u')\delta_i$ for some $u'\in \Wh_{w_i^{-1}\eta}(\sigma\otimes e^{\lambda+\rho})$.
By Lemma~\ref{lem:meromorphic extension}, there exists a distribution $x_t\in I_i(\lambda + t\rho)$ with meromorphic parameter $t$ such that $x_t|_{U_i}$ is holomorphic and $(x_t|_{U_i})|_{t = 0} = x$.
Moreover, $(X - \eta(X))x_t = 0$ for $X\in\mathfrak{n}_0$.
By Proposition~\ref{prop:convergence and continuation} and the condition (a), the distribution $x_t$ is holomorphic at $t = 0$.
Hence $x_0|_{U_i} = x$.
The map $\Wh_\eta(I_i)\to \Wh_\eta(I_i/I_{i - 1})$ is surjective.
\end{proof}

Next we consider the module $\Wh_\eta((I(\sigma,\lambda)_{\text{$K$-finite}})^*)$.
Take a filtration $\widetilde{I_i}\subset J^*_\eta(I(\sigma,\lambda))$ as in Theorem~\ref{thm:stucture of J^*(I(sigma,lambda))}.
\begin{lem}\label{lem:two step calc for Whittaker vector}
Let $V$ be an object of the category $\mathcal{O}'$.
Then we have $C(H^0(\mathfrak{n}_\eta,V)) = H^0(\mathfrak{n}_\eta,C(V))$ where $H^0(\mathfrak{n}_\eta,V) = \{v\in V\mid \mathfrak{n}_\eta v = 0\}$ is the $0$-th $\mathfrak{n}_\eta$-cohomology.
\end{lem}
\begin{proof}
We get the lemma by the following equation.
\begin{multline*}
H^0(\mathfrak{n}_\eta,C(V))  = H^0(\mathfrak{n}_\eta,D'(V)^*) = (D'(V)/\mathfrak{n}_\eta D'(V))^*\\
= CD'(D'(V)/\mathfrak{n}_\eta D'(V)) = C(H^0(\mathfrak{n}_\eta,D'(V)^*)_{\text{$\mathfrak{h}$-finite}})\\
= C(H^0(\mathfrak{n}_\eta,D'D'(V))) = C(H^0(\mathfrak{n}_\eta,V)).
\end{multline*}
\end{proof}

\begin{lem}\label{lem:relation in S_w}
Let $e_1,\dots, e_l$ be a basis of $\Ad(w_i)\overline{\mathfrak{n}}\cap\mathfrak{n}_0$ such that each $e_s$ is a root vector and $\bigoplus_{s\le t - 1}\C e_s$ is an ideal of $\bigoplus_{s\le t}\C e_s$.
In $S_{w_i,0}$ where $0$ is the trivial representation, we have the following formulae.
\begin{enumerate}
\item For all $t = 1,\dots,l$, 
\begin{multline*}
e_t(e_1^{-1}\dotsm e_{t - 1}^{-1}e_t^{-(k_t + 1)}\dotsm e_l^{-(k_l + 1)})\\ = e_1^{-1}\dotsm e_{t - 1}^{-1}e_t^{-k_t}e_{t + 1}^{-(k_{t + 1} + 1)}\dotsm e_l^{-(k_l + 1)}
\end{multline*}
\item Fix $t \in \{1,\dots,l\}$ such that $e_t\in \mathfrak{n}_\eta$.
Assume that $k_s = 0$ for all $s < t$ such that $e_s\in \mathfrak{n}_\eta$.
Then 
\begin{multline*}
e_t(e_1^{-(k_1 + 1)}\dotsm e_l^{-(k_l + 1)})\\ = e_1^{-(k_1 + 1)}\dots e_{t - 1}^{-(k_{t - 1} + 1)}e_t^{-k_t}e_{t + 1}^{-(k_t + 1)}\dots e_l^{-(k_l + 1)}.
\end{multline*}
\item $X(e_1^{-1}\dotsm e_l^{-1}) = (e_1^{-1}\dotsm e_l^{-1})X$  for $X\in \Ad(w_i)\mathfrak{m}\cap \mathfrak{n}_0$.
\end{enumerate}
\end{lem}
\begin{proof}
Let $\alpha_s$ be a restricted root corresponding to $e_s$.

(1)
It is sufficient to prove $e_t(e_1^{-1}\dotsm e_{t - 1}^{-1}) = (e_1^{-1}\dotsm e_{t - 1}^{-1})e_t$ in $S_{e_1}\otimes_{U(\mathfrak{g})}\dots \otimes_{U(\mathfrak{g})}S_{e_{t - 1}}$.
Since $\bigoplus_{s = 1}^{t - 1}\C e_s$ is an ideal of $\bigoplus_{s = 1}^t \C e_s$, we have 
\[
	e_t(e_1^{-1}\dotsm e_{t - 1}^{-1}) - (e_1^{-1}\dotsm e_{t - 1}^{-1})e_t\in \bigoplus_{k_s\ge 0}\C e_1^{-(k_1 + 1)}\dotsm e_{t - 1}^{-(k_{t - 1} + 1)}.
\]
An $\mathfrak{a}_0$-weight of the left hand side is $-\alpha_1 - \dots - \alpha_{t - 1} + \alpha_t$.
However, the set of $\mathfrak{a}_0$-weights of the right hand side is $\{-(k_1 + 1)\alpha_1 - \dots - (k_{t - 1} + 1)\alpha_{t - 1}\mid k_s\in\Z_{\ge 0}\}$.
Hence each $\mathfrak{a}_0$-weight appearing in the right hand side is less than that of the left hand side.
This implies $e_t(e_1^{-1}\dots e_{t - 1}^{-1}) - (e_1^{-1}\dots e_{t - 1}^{-1})e_t = 0$.

(2)
We prove $e_t(e_1^{-(k_1 + 1)}\dotsm e_{t - 1}^{-(k_{t - 1} + 1)}) = (e_1^{-(k_1 + 1)}\dots e_{t - 1}^{-(k_{t - 1} + 1)})e_t$ in $S_{e_1}\otimes_{U(\mathfrak{g})}\dots \otimes_{U(\mathfrak{g})}S_{e_{t - 1}}$.
As in the proof of (1), we have 
\begin{multline*}
	e_t(e_1^{-(k_1 + 1)}\dotsm e_{t - 1}^{-(k_{t - 1} + 1)}) - (e_1^{-(k_1 + 1)}\dotsm e_{t - 1}^{-(k_{t - 1} + 1)})e_t\\\in \bigoplus_{k_s\ge 0}\C e_1^{-(k_1 + 1)}\dotsm e_{t - 1}^{-(k_{t - 1} + 1)}.
\end{multline*}
An $\mathfrak{a}_\eta$-weight of the left hand side is $\sum_{e_s\in \mathfrak{n}_\eta,\ s < t}-\alpha_s + \alpha_t$.
However, the set of $\mathfrak{a}_\eta$-weights of the right hand side is $\{\sum_{e_s\in\mathfrak{n}_\eta,\ s < t}-(k_s + 1)\alpha_s\mid k_s\in\Z_{\ge 0}\}$.
Hence each $\mathfrak{a}_\eta$-weight appearing in the right hand side is less than that of the left hand side.
This implies the lemma.

(3)
We may assume $X$ is a restricted root vector.
Let $\alpha$ be a restricted root of $X$.
Since $X$ normalizes $\Ad(w_i)\overline{\mathfrak{n}}\cap \mathfrak{n}_0$, we have
\[
	X(e_1^{-1}\dotsm e_l^{-1}) - (e_1^{-1}\dotsm e_l^{-1})X\in \bigoplus_{k_s\ge 0}\C e_1^{-(k_1 + 1)}\dotsm e_l^{-(k_l + 1)}.
\]
Then $X(e_1^{-1}\dotsm e_l^{-1}) - (e_1^{-1}\dotsm e_l^{-1})X$ has an $\mathfrak{a}_0$-weight $-(\alpha_1 + \dots + \alpha_s) + \alpha$.
However, $e_1^{-(k_1 + 1)}\dotsm e_l^{-(k_l + 1)}$ has a $\mathfrak{a}_0$-weight $-((k_1 + 1)\alpha_1 + \dots + (k_l + 1)\alpha_l) < -(\alpha_1 + \dots + \alpha_s) + \alpha$.
Hence $X(e_1^{-1}\dotsm e_l^{-1}) - (e_1^{-1}\dotsm e_l^{-1})X = 0$.
\end{proof}

\begin{lem}\label{lem:e^{-1}-part}
Let $e_1,\dots, e_l$ be a basis of $\Ad(w_i)\overline{\mathfrak{n}}\cap\mathfrak{n}_0$ such that $e_s$ is a root vector and $\bigoplus_{s\le t - 1}\C e_s$ is an ideal of $\bigoplus_{s\le t}\C e_s$.
Let $V$ be a $U(\mathfrak{m}\oplus\mathfrak{a})$-representation.
Regard $V$ as a $\mathfrak{p}$-representation by $\mathfrak{n}V = 0$.
By Lemma~\ref{lem:induction + twisting}, we have $T_{w_i}(U(\mathfrak{g})\otimes_{U(\mathfrak{p})}V) \simeq (\bigoplus_{k_s\ge 0}\C e_1^{-(k_1 + 1)}\dotsm e_l^{-(k_l + 1)})\otimes U(\Ad(w_i)\overline{\mathfrak{n}}\cap \overline{\mathfrak{n}_0})\otimes w_iV$.
Then we have $\{v\in e_1^{-1}\dotsm e_l^{-1}\otimes 1\otimes w_iV\mid \mathfrak{n}_\eta v = 0\} = e_1^{-1}\dotsm e_l^{-1}\otimes 1\otimes H^0(\Ad(w_i)\mathfrak{m}\cap \mathfrak{n}_\eta,w_iV)$.
\end{lem}
\begin{proof}
Take $v = e_1^{-1}\dotsm e_l^{-1}\otimes 1\otimes v_0\in H^0(\mathfrak{n}_\eta,T_{w_i}(U(\mathfrak{g})\otimes_{U(\mathfrak{p})}V)$.
Then for $X\in \Ad(w_i)\mathfrak{m}\cap \mathfrak{n}_\eta$ we have $X(e_1^{-1}\dotsm e_l^{-1}\otimes 1\otimes v_0) = 0$.
By Lemma~\ref{lem:relation in S_w}, we have $e_1^{-1}\dotsm e_l^{-1}\otimes 1\otimes Xv_0 = 0$.
Hence $Xv_0 = 0$.
\end{proof}

By the definition of the Harish-Chandra homomorphism, we get the following lemma.

\begin{lem}\label{lem:infinitesimal character of a space of partial highest weight vectors}
Let $\mathfrak{q}$ be a parabolic subalgebra of $\mathfrak{g}$ containing $\mathfrak{h}\oplus\mathfrak{u}_0$.
Take a Levi decomposition $\mathfrak{l}\oplus\mathfrak{u}_\mathfrak{q}$ of $\mathfrak{q}$ such that $\mathfrak{h}\subset \mathfrak{l}$.
Let $\widetilde{W_\mathfrak{l}}\subset \widetilde{W}$ be the Weyl group of $\mathfrak{l}$, $V$ an $\mathfrak{l}$-module with an infinitesimal character $\widetilde{\mu}$.
Put $V' = H^0(\mathfrak{u}_\mathfrak{q},V)$ and $\widetilde{\rho_{\mathfrak{u}_\mathfrak{q}}}(H) = (1/2)\Tr \ad(H)|_{\mathfrak{u}_\mathfrak{q}}$ for $H\in \mathfrak{h}$.
Then $V'$ is $\mathfrak{l}$-stable and $V' = \bigoplus_{\widetilde{w}\in \widetilde{W_\mathfrak{l}}\backslash\widetilde{W}}(V')_{[\widetilde{w}\widetilde{\mu} - \widetilde{\rho_{\mathfrak{u}_\mathfrak{q}}}]}$ where $(V')_{[\widetilde{w}\widetilde{\mu} - \widetilde{\rho_{\mathfrak{u}_\mathfrak{q}}}]}$ is the maximal $\mathfrak{l}$-submodule which has an infinitesimal character $\widetilde{w}\widetilde{\mu} - \widetilde{\rho_{\mathfrak{u}_\mathfrak{q}}}$.
In particular, for an $\mathfrak{l}$-submodule $V''$ of $V'$, a highest weight of $V'/V''$ belongs to $\{\widetilde{w}\widetilde{\mu} - \widetilde{\rho}\mid \widetilde{w}\in\widetilde{W}\}$.
\end{lem}

The following lemma is well-known.
\begin{lem}\label{lem:distribution of weight}
Let $V\in \mathcal{O}'$.
Assume that $V$ has an infinitesimal character $\widetilde{\lambda}\in \mathfrak{h}^*$.
Then a $\mathfrak{h}$-weight appearing in $V$ is contained in $\{\widetilde{w}\widetilde{\lambda} - \widetilde{\rho} - \alpha\mid \widetilde{w}\in\widetilde{W},\ \alpha\in\Z_{\ge 0}\Delta^+\}$.
\end{lem}

Now we determine the dimension of the space of Whittaker vectors of $\widetilde{I_i}/\widetilde{I_{i - 1}}$ under some conditions.

\begin{lem}\label{lem:Whittaker vectors in a Bruhat cell, algebraic}
Let $\widetilde{\mu}$ be an infinitesimal character of $\sigma$.
Assume that for all $\widetilde{w}\in \widetilde{W}\setminus \widetilde{W_M}$, $(\lambda+\widetilde{\mu}) - \widetilde{w}(\lambda+\widetilde{\mu})\not\in \Z\Delta$.
Then we have $\dim\Wh_\eta(\widetilde{I_i}/\widetilde{I_{i - 1}}) = \dim \Wh_{w_i^{-1}\eta}((\sigma_{\text{\normalfont $M\cap K$-finite}})^*)$.
\end{lem}
\begin{proof}
Put $V = T_{w_i}(U(\mathfrak{g})\otimes_{U(\mathfrak{p})}J^*(\sigma\otimes e^{\lambda+\rho}))$.
By Theorem~\ref{thm:stucture of J^*(I(sigma,lambda))}, we have $\Wh_\eta(\widetilde{I_i}/\widetilde{I_{i - 1}}) = \Wh_\eta(C(V))$.
Let $e_1,\dots, e_l$ be a basis of $\Ad(w_i)\overline{\mathfrak{n}}\cap\mathfrak{n}_0$ such that $\bigoplus_{s\le t - 1}\C e_s$ is an ideal of $\bigoplus_{s\le t}\C e_s$.
Moreover, assume that each $e_i$ is a root vector.
For $\mathbf{k} = (k_1,\dots,k_l)\in \Z^l$, put $e^\mathbf{k} = e_1^{k_1}\dotsm e_l^{k_l}$.
Set $\mathbf{1} = (1,\dots,1)\in\Z^l$.
Then we have
\[
	V = \bigoplus_{k\in\Z_{\ge 0}^l}\C e^{-(\mathbf{k} + \mathbf{1})}\otimes U(\Ad(w_i)\overline{\mathfrak{n}}\cap\overline{\mathfrak{n}_0})\otimes w_iJ^*(\sigma\otimes e^{\lambda+\rho}).
\]
Put 
\[
	V' = \bigoplus_{\mathbf{k}\in \mathcal{A}}e^{-(\mathbf{k} + \mathbf{1})}\otimes U(\Ad(w_i)\overline{\mathfrak{n}}\cap\overline{\mathfrak{n}_0}\cap \mathfrak{m}_\eta)\otimes H^0(\mathfrak{m}\cap\mathfrak{n}_\eta,w_iJ^*(\sigma\otimes e^{\lambda+\rho}))
\]
where $\mathcal{A} = \{(k_1,\dots,k_l)\in \Z^l_{\ge 0}\mid \text{if $e_i\in \mathfrak{n}_\eta$ then $k_i = 0$}\}$.
It is easy to see that $V'$ is an $\mathfrak{m}_\eta\oplus \mathfrak{a}_\eta$-stable and $V'\subset H^0(\mathfrak{n}_\eta,V)$.
We prove that $V' = H^0(\mathfrak{n}_\eta,V)$.

To prove $V' = H^0(\mathfrak{n}_\eta,V)$, it is sufficient to prove that there exists no highest weight vector in $H^0(\mathfrak{n}_\eta,V)/ V'$.
Let $v\in H^0(\mathfrak{n}_\eta,V)$ such that $(\mathfrak{m}_\eta\cap \mathfrak{u})v \in V'$.

First, we prove that $v\in e^{-\mathbf{1}}\otimes U(\Ad(w_i)\overline{\mathfrak{n}}\cap \overline{\mathfrak{n}_0})\otimes J^*(\sigma\otimes e^{\lambda + \rho}) + V'$.
Take $y_\mathbf{k}\in U(\Ad(w_i)\overline{\mathfrak{n}}\cap \overline{\mathfrak{n}_0})\otimes J^*(\sigma\otimes e^{\lambda + \rho})$ such that $v = \sum_{\mathbf{k}} e^{-(\mathbf{k} + \mathbf{1})}\otimes y_\mathbf{k}$.
We prove that if $k_t\ne 0$ and $e_t\in \mathfrak{n}_\eta$ then $y_\mathbf{k} = 0$ by induction on $t$ where $\mathbf{k} = (k_1,\dots,k_l)$.
Put $\mathbf{1}_t = (\delta_{st})_{1\le s\le l}\in \Z^l$ ($\delta_{st}$ is Kronecker's delta).
By inductive hypothesis, for $s < t$ such that $e_s\in\mathfrak{n}_\eta$, if $y_\mathbf{k}\ne 0$ then $k_s = 0$.
By Lemma~\ref{lem:relation in S_w} (2), we have $e_tv = \sum_{\mathbf{k}\in\Z_{\ge 0}^l}e^{-(\mathbf{k} + \mathbf{1}) + \mathbf{1}_t}\otimes y_\mathbf{k}$.
Since $v\in H^0(\mathfrak{n}_\eta,V)$, we have $e_tv = 0$.
Hence if $e^{-(\mathbf{k} + \mathbf{1}) + \mathbf{1}_t}\ne 0$ then $y_\mathbf{k} = 0$.
Since $e^{-(\mathbf{k} + \mathbf{1}) + \mathbf{1}_t} = 0$ if and only if $k_t = 0$, $k_t \ne 0$ implies $y_\mathbf{k} = 0$.
We prove that if $k_t\ne 0$ then $e^{-(\mathbf{k} + \mathbf{1})}\otimes y_\mathbf{k}\in V'$ by induction on $t$.
If $e_t\in \mathfrak{n}_\eta$ then this claim is already proved.
We may assume that $e_t\in \mathfrak{m}_\eta$.
Hence $e_tV'\subset V'$.
By inductive hypothesis, if $k_s\ne 0$ for some $s < t$ then $e^{-(\mathbf{k} + \mathbf{1})}\otimes y_\mathbf{k} \in V'$.
Then we have $e_tv\in \sum_{\mathbf{k}\in\Z_{\ge 0}^l}e^{-(\mathbf{k} + \mathbf{1}) + \mathbf{1}_t}\otimes y_\mathbf{k} + V'$ by Lemma~\ref{lem:relation in S_w} (1).
Since $e_tv\in V'$, we have $\sum_{\mathbf{k}\in\Z_{\ge 0}^l}e^{-(\mathbf{k} + \mathbf{1}) + \mathbf{1}_t}\otimes y_\mathbf{k} \in V'$.
By the definition of $V'$, if $e^{-(\mathbf{k} + \mathbf{1}) + \mathbf{1}_t}\ne 0$ then $e^{-(\mathbf{k} + \mathbf{1})}\otimes y_\mathbf{k}\in V'$.
Hence we get the claim.

We may assume that $v$ is a weight vector with respect to $\mathfrak{h}$.
We can take $\widetilde{w}\in \widetilde{W}$ such that $-\widetilde{w}(\lambda+\widetilde{\mu}) - \widetilde{\rho}$ is a $\mathfrak{h}$-weight of $v$ by Lemma~\ref{lem:infinitesimal character of a space of partial highest weight vectors}.
Put $\widetilde{\rho_M} = \sum_{\alpha\in\Delta_M^+}(1/2)\alpha$.
Since $J^*(\sigma\otimes e^{\lambda + \rho})$ has an infinitesimal character $-(\lambda + \widetilde{\mu} + \rho)$, a $\mathfrak{h}$-weight appearing in $J^*(\sigma\otimes e^{\lambda + \rho})$ is contained in $\{-\widetilde{w}(\lambda + \widetilde{\mu} + \rho) - \widetilde{\rho_M} + \alpha\mid \widetilde{w}\in \widetilde{W_M},\ \alpha\in\Z\Delta_M\}$ by Lemma~\ref{lem:distribution of weight}.
Since $-\rho\in\mathfrak{a}^*$, we have $\widetilde{w}\rho = \rho$ for $\widetilde{w}\in \widetilde{W_M}$.
Hence we have $- \widetilde{w}\rho - \widetilde{\rho_M} = -\rho - \widetilde{\rho_M} = -\widetilde{\rho}$.
Notice that $w_i\widetilde{\rho} - \widetilde{\rho}\in \Z\Delta$.
Therefore a $\mathfrak{h}$-weight appearing in $V$ is contained in
\begin{multline*}
-w_i\widetilde{W_M}(\lambda + \widetilde{\mu}) - w_i\widetilde{\rho} + w_i\Z\Delta_M + \Z_{\ge 0}(w_i\Delta^-\cap \Delta^-) - \Z_{\ge 1}(w_i\Delta^-\cap \Delta^+)\\
\subset -w_i\widetilde{W_M}(\lambda + \widetilde{\mu}) - \widetilde{\rho} + \Z\Delta.
\end{multline*}
This implies that for some $\widetilde{w'}\in\widetilde{W_M}$, we have $\widetilde{w}(\lambda + \widetilde{\mu}) - w_i\widetilde{w'}(\lambda +\widetilde{\mu})\in\Z\Delta$.
By the assumption we have $\widetilde{w}\in w_i\widetilde{W_M}$.
This implies $(\wt v)(\Ad(w_i)H) = -(\lambda(H) + w_i^{-1}\widetilde{\rho}(H))$ for all $H\in \mathfrak{a}$ where $\wt v$ is a $\mathfrak{h}$-weight of $v$.

Take $T_p\in U(\Ad(w_i)\overline{\mathfrak{n}}\cap \overline{\mathfrak{n}_0})$ and $x_p\in w_iJ^*(\sigma\otimes e^{\lambda + \rho})$ such that $v \in \sum_p e^{-\mathbf{1}}\otimes T_p\otimes x_p + V'$.
We may assume that $T_p$ (resp.\ $x_p$) is a $\mathfrak{h}$-weight vector with respect to the adjoint action (resp.\ the action induced from $\sigma\otimes e^{\lambda + \rho}$).
We denote its $\mathfrak{h}$-weight by $\wt T_p$ and $\wt x_p$.
Fix $H\in \mathfrak{a}$.
Then $\alpha(H) = 0$ for all $\alpha\in\Delta_M$.
Since $\wt x_p\in -w_i(\widetilde{W_M}(\lambda + \widetilde{\mu}) + \widetilde{\rho} + \Z\Delta_M)$, $(\wt x_p)(\Ad(w_i)H) = -(\lambda + \widetilde{\rho})(H)$.
Hence
\begin{align*}
&(\wt v)(\Ad(w_i)H)\\
& = (\wt(e^{-\mathbf{1}}) + \wt(T_p) + \wt(x_p))(\Ad(w_i)(H))\\
& = (\wt(e^{-\mathbf{1}})(\Ad(w_i)H) + (\wt T_p)(\Ad(w_i)H) - (\lambda + \widetilde{\rho})(H)\\
& = (\wt(e^{-\mathbf{1}})(\Ad(w_i)H) + (\wt T_p)(\Ad(w_i)H) - (\lambda + \widetilde{\rho})(H).
\end{align*}
We calculate $\wt(e^{-\mathbf{1}})(\Ad(w_i)H)$.
By the definition, $\wt(e^{-\mathbf{1}})(\Ad(w_i)H) = \Tr\ad(\Ad(w_i)H)|_{\Ad(w_i)\overline{\mathfrak{n}}\cap \mathfrak{n}_0}$.
Since we have $\Ad(w_i)\overline{\mathfrak{n}}\cap \mathfrak{n}_0 = \Ad(w_i)\overline{\mathfrak{n}_0}\cap \mathfrak{n}_0$, we have
\begin{multline*}
\Tr\ad(\Ad(w_i)H)|_{\Ad(w_i)\overline{\mathfrak{n}}\cap \mathfrak{n}_0}
= \Tr\ad(\Ad(w_i)H)|_{\Ad(w_i)\overline{\mathfrak{n}_0}\cap \mathfrak{n}_0}\\
~ \Tr\ad(H)|_{\Ad(w_i)^{-1}\mathfrak{n}_0\cap \overline{\mathfrak{n}_0}}
= (-\widetilde{\rho} + w_i^{-1}\widetilde{\rho})(H).
\end{multline*}
Hence we get
\[
	(\wt v)(\Ad(w_i)H) = (\wt T_p)(\Ad(w_i)H) - (\lambda + w_i^{-1}\widetilde{\rho})(H).
\]
We have already proved that $(\wt v)(\Ad(w_i)H) = -(\lambda + w_i^{-1}\widetilde{\rho})(H)$.
Therefore we get $(\wt T_p)(\Ad(w_i)H) = 0$ for all $H\in\mathfrak{a}$.
Since $T_p\in U(\Ad(w_i)\mathfrak{n})$, this implies $T_p\in \C$, i.e., there exist $v'\in e_1^{-1}\dotsm e_l^{-1}\otimes 1\otimes w_iJ^*(\sigma\otimes e^{\lambda+\rho})$ and $v''\in V'$ such that $v = v' + v''$.
Therefore $\mathfrak{n}_\eta(v') = \mathfrak{n}_\eta(v - v'') = 0$.
Hence, $v'\in V'$ by Lemma~\ref{lem:e^{-1}-part}.
Therefore $H^0(\mathfrak{n}_\eta,V) = V'$.

For an $\mathfrak{m}_0\oplus\mathfrak{a}_0$-module $\tau$ and a subalgebra $\mathfrak{c}$ of $\mathfrak{g}$ containing $\mathfrak{m}_0\oplus\mathfrak{a}_0$, put $M_{\mathfrak{c}}(\tau) = U(\mathfrak{c})\otimes_{U(\mathfrak{c}\cap\overline{\mathfrak{p}_0})}(\tau\otimes\rho')$ where $\overline{\mathfrak{n}_0}\cap \mathfrak{c}$ acts on $\tau$ trivially and $\rho'(H) = (\Tr(\ad(H)|_{\mathfrak{c}\cap\overline{\mathfrak{n}_0}}))/2$ for $H\in\mathfrak{a}_0$.

For $\widetilde{\lambda}\in \mathfrak{h}^*$ such that $\widetilde{\lambda}|_{\mathfrak{m}_0}$ is regular dominant integral, let $\sigma_{M_0A_0,\widetilde{\lambda}}$ be the finite-dimensional representation of $M_0A_0$ with an infinitesimal character $\widetilde{\lambda}$.
Let $\ch M$ be the character of $M$.
We can take integers $c_{\widetilde{\lambda}}$ such that 
\begin{multline*}
	\ch D' H^0(\mathfrak{n}_\eta\cap \Ad(w_i)\mathfrak{m},w_iJ^*(\sigma\otimes e^{\lambda+\rho}))\\ = \sum_{\widetilde{\lambda}}c_{\widetilde{\lambda}}\ch M_{(\mathfrak{m}_\eta\cap\Ad(w_i)\mathfrak{m}) + \mathfrak{a}_0}(\sigma_{M_0A_0,\widetilde{\lambda}})
\end{multline*}
Then we have $\ch D'V' = \sum_{\widetilde{\lambda}}c_{\widetilde{\lambda}}\ch M_{\mathfrak{m}_\eta\oplus\mathfrak{a}_\eta}(\sigma_{M_0A_0,\widetilde{\lambda}})$.
The functor $X\mapsto \Wh_{\eta|_{\mathfrak{m}_\eta\cap\mathfrak{n}_0}}(X^*)$ is an exact functor by a result of Lynch~\cite{lynch-whittaker}.
Hence, we have $\dim\Wh_{\eta|_{\mathfrak{m}_\eta\cap\mathfrak{n}_0}}(C(V')) = \sum_{\widetilde{\lambda}}c_{\widetilde{\lambda}}\dim\Wh_{\eta|_{\mathfrak{m}_\eta\cap\mathfrak{n}_0}}(M_{\mathfrak{m}_\eta\oplus\mathfrak{a}_\eta}(\sigma_{M_0A_0,\widetilde{\lambda}})^*)$.
Lynch also proves $\dim\Wh_{\eta|_{\mathfrak{m}_\eta\cap\mathfrak{n}_0}}(M_{\mathfrak{m}_\eta}(\sigma_{M_0A_0,\widetilde{\lambda}})^*) = \dim\sigma_{M_0A_0,\widetilde{\lambda}}$.
Therefore, by Lemma~\ref{lem:two step calc for Whittaker vector}, we have $\dim \Wh_\eta(\widetilde{I_i}/\widetilde{I_{i - 1}}) = \dim\Wh_{\eta|_{\mathfrak{m}_\eta\cap\mathfrak{n}_0}}(C(V')) = \sum_{\widetilde{\lambda}} c_{\widetilde{\lambda}}\dim\sigma_{M_0A_0,\widetilde{\lambda}}$.
By the same argument we have 
\begin{align*}
&\sum_{\widetilde{\lambda}} c_{\widetilde{\lambda}}\dim\sigma_{M_0A_0,\widetilde{\lambda}}\\
& = \sum_{\widetilde{\lambda}} c_{\widetilde{\lambda}}\dim \Wh_{\eta|_{\mathfrak{m}_\eta\cap\Ad(w_i)\mathfrak{m}\cap\mathfrak{n}_0}}(M_{(\mathfrak{m}_\eta\cap\Ad(w_i)\mathfrak{m}) + \mathfrak{a}_0}(\sigma_{M_0A_0,\widetilde{\lambda}})^*) \\
& = \dim\Wh_{\eta|_{\mathfrak{m}_\eta\cap\Ad(w_i)\mathfrak{m}\cap\mathfrak{n}_0}}(CH^0(\mathfrak{n}_\eta\cap \Ad(w_i)\mathfrak{m},w_iJ^*(\sigma\otimes e^{\lambda+\rho})))\\
& = \dim \Wh_{\eta|_{\Ad(w_i)\mathfrak{m}\cap\mathfrak{n}_0}}(C(w_iJ^*(\sigma\otimes e^{\lambda+\rho})))\\
& = \dim\Wh_{w_i^{-1}\eta}(C(J^*(\sigma\otimes e^{\lambda+\rho})))\\
& = \dim \Wh_{w_i^{-1}\eta}((\sigma_{\text{$M\cap K$-finite}})^*).
\end{align*}
This implies the lemma.
\end{proof}

\begin{thm}\label{thm:dimension Whittaker vectors, algebraic}
Let $\widetilde{\mu}$ be an infinitesimal character of $\sigma$.
Assume that for all $\widetilde{w}\in \widetilde{W}\setminus \widetilde{W_M}$, $(\lambda+\widetilde{\mu}) - \widetilde{w}(\lambda+\widetilde{\mu})\not\in \Z\Delta$.
Then we have
\[
	\dim\Wh_\eta((I(\sigma,\lambda)_{\text{\normalfont $K$-finite}})^*) = \sum_{w\in W(M)}\dim \Wh_{w^{-1}\eta}((\sigma_{\text{\normalfont $M\cap K$-finite}})^*).
\]
\end{thm}
\begin{proof}
%Let $I_i$ be a filtration of $J^*(I(\sigma,\lambda))$ defined in Section~\ref{sec:Parabolic induction and Bruhat filtration}.
Since a $\mathfrak{h}$-weight appearing in $T_{w_i}(U(\mathfrak{g})\otimes_{U(\mathfrak{p})}J^*(\sigma\otimes e^{\lambda+\rho}))$ belongs to $\{-w_i\widetilde{w}(\lambda+\widetilde{\mu}) - \widetilde{\rho} +  \alpha\mid \widetilde{w}\in\widetilde{W_M},\ \alpha\in\Delta\}$, the exact sequence $0\to I_{i - 1} \to I_i\to T_{w_i}(U(\mathfrak{g})\otimes_{U(\mathfrak{p})}J^*(\sigma\otimes e^{\lambda+\rho}))\to 0$ splits.
Hence, we have $J^*_\eta(I(\sigma,\lambda)) = \bigoplus_i\Gamma_\eta(C(T_{w_i}(U(\mathfrak{g})\otimes_{U(\mathfrak{p})}J^*(\sigma\otimes e^{\lambda+\rho}))))$.
Therefore the theorem follows from Lemma~\ref{lem:Whittaker vectors in a Bruhat cell, algebraic}.
\end{proof}

Finally we study the case of $\sigma$ is finite-dimensional.
If $\sigma$ is finite-dimensional, then $\mathfrak{m}\cap\mathfrak{n}_0$ acts on $\sigma$ nilpotently.
Hence $\Wh_{w_i^{-1}\eta}(\sigma^*)\ne 0$ if and only if $w_i^{-1}\eta = 0$ on $\mathfrak{m}\cap\mathfrak{n}_0$.

\begin{defn}
Let $\Theta,\Theta_1,\Theta_2$ be subsets of $\Pi$.
\begin{enumerate}
\item Put $W(\Theta) = \{w\in W\mid w(\Theta)\subset \Sigma^+\}$ and $\Sigma_\Theta = \Z\Theta\cap \Sigma$.\newsym{$W(\Theta)$}
\item Put $W(\Theta_1,\Theta_2) = \{w\in W(\Theta_1)\cap W(\Theta_2)^{-1}\mid w(\Sigma_{\Theta_1})\cap \Sigma_{\Theta_2} = \emptyset\}$.\newsym{$W(\Theta_1,\Theta_2)$}
\item Let $W_\Theta$ be the Weyl group of $\Sigma_\Theta$.\newsym{$W_\Theta$}
\end{enumerate}
\end{defn}

\begin{lem}\label{lem:coparison oshima}
Let $\Theta$ be a subset of $\Pi$ corresponding to $P$.
\begin{enumerate}
\item We have $\#W(\supp\eta,\Theta) = \#\{w\in W(M)\mid w(\Sigma^+)\cap \Sigma^+_\eta = \emptyset\}$.
\item We have $\#W(\supp\eta,\Theta)\times\# W_{\supp\eta} = \#\{w\in W(M)\mid \supp\eta\cap w(\Sigma_M^+) = \emptyset\}$.
\end{enumerate}
\end{lem}
\begin{proof}
(1)
Put $\mathcal{W} = \{w\in W(M)\mid w(\Sigma^+)\cap \Sigma^+_\eta = \emptyset\}$.
Let $w_{\eta,0}$ be the longest Weyl element of $W_{M_\eta}$.
We prove that the map $\mathcal{W}\to W(\supp\eta,\Theta)$ defined by $w\mapsto (w_{\eta,0}w)^{-1}$ is well-defined and bijective.

First we prove that the map is well-defined.
Let $w\in\mathcal{W}$.
The equation $w(\Sigma^+)\cap \Sigma_\eta^+ = \emptyset$ implies that $(w_{\eta,0}w)^{-1}(\Sigma^+_\eta)\subset \Sigma^+$.
Hence, $(w_{\eta,0}w)^{-1}\in W(\supp\eta)$.
Moreover, $w(\Sigma_M^+)\subset \Sigma^+$ and $w(\Sigma^+)\cap \Sigma_\eta^+ = \emptyset$ imply that $w(\Sigma_M^+)\subset \Sigma^+\cap (\Sigma\setminus\Sigma^+_\eta) = \Sigma^+\setminus\Sigma^+_\eta$.
Hence, $(w_{\eta,0}w)(\Sigma_M^+)\subset \Sigma^+\setminus\Sigma^+_\eta\subset \Sigma^+$.
We have $(w_{\eta,0}w)^{-1}\in W(\Theta)^{-1}$.
Finally $w(\Sigma_M^+)\subset \Sigma^+\setminus\Sigma_\eta^+$ implies $w(\Sigma)\subset \Sigma\setminus\Sigma_\eta$.
Hence we have $(w_{\eta,0}w)^{-1}\Sigma_\eta\cap\Sigma_M = w^{-1}\Sigma_\eta\cap \Sigma_M = \emptyset$.

Assume that $(w_{\eta,0}w)^{-1}\in W(\supp\eta,\Theta)$.
Then $(w_{\eta,0}w)^{-1}(\Sigma^+_\eta)\subset\Sigma^+$ implies that $w(\Sigma^+)\cap \Sigma^+_\eta=\emptyset$.
Since $(w_{\eta,0}w)^{-1}\Sigma_\eta\cap\Sigma_M = \emptyset$ we have $w(\Sigma_M)\cap \Sigma_\eta = \emptyset$.
By $(w_{\eta,0}w)(\Sigma_M^+)\subset\Sigma^+$  and $w(\Sigma^+)\cap \Sigma_\eta^+ = \emptyset$, we have $w(\Sigma_M^+)\subset ((\Sigma^+\setminus\Sigma_\eta^+)\cup \Sigma_\eta^-)\cap (\Sigma\setminus\Sigma^-_\eta) = (\Sigma^+\setminus\Sigma^+_\eta)$.
Consequently we have $w\in W(M)$.

(2)
Put $\mathcal{W} = \{w\in W(M)\mid \supp\eta\cap w(\Sigma_M^+) = \emptyset\}$.
Define the map $\varphi\colon W(\supp\eta,\Theta)\times W_{\supp\eta} \to \mathcal{W}$ by $(w_1,w_2)\mapsto w_2w_1^{-1}$.
This map is injective since $W(\supp\eta,\Theta)\subset W(\supp\eta)$.
We prove that $\varphi$ is well-defined and surjective.
Since $w_1^{-1}(\Sigma_M^+) = w_1^{-1}(\Sigma_M^+)\cap \Sigma^+\subset \Sigma^+\setminus \Sigma_\eta^+$, we have $w_2w_1^{-1}(\Sigma_M^+)\subset \Sigma^+\setminus\Sigma_\eta^+$.
Hence, $\varphi$ is well-defined.
Next let $w\in\mathcal{W}$.
Let $w_1\in W(\supp\eta)^{-1}$ and $w_2\in W_{\supp\eta}$ such that $w = w_2w_1^{-1}$.
Then $w_1^{-1}(\Sigma_M^+) = w_2^{-1}w(\Sigma_M^+)\subset w_2^{-1}(\Sigma^+\setminus\Sigma_\eta^+) = \Sigma^+\setminus\Sigma^+_\eta$.
This implies $w_1\in W(\supp\eta,\Theta)$.
\end{proof}

\begin{lem}\label{lem:whittaker vectors of finite-dimensional representation}
Assume that $\sigma$ is irreducible and finite-dimensional.
Let $\widetilde{\mu}$ be the highest weight of $\sigma$ and $V$ the irreducible finite-dimensional representation of $M_0A_0$ with highest weight $\lambda+\widetilde{\mu}$.
Then we have $\sigma/(\mathfrak{m}\cap\mathfrak{n}_0)\sigma\simeq V$ as an $M_0A_0$-module.
In particular, $\dim\Wh_0(\sigma') = \dim V$.
\end{lem}
\begin{proof}
We prove that $\Wh_0(\sigma^*) \simeq V^*$.
Let $\widetilde{w}_{M,0}$ be the longest element of $\widetilde{W_M}$.
Then both sides have a highest weight $-\widetilde{w}_{M,0}(\widetilde{\mu} + \lambda)$ and the space of highest weight vectors are $1$-dimensional.
\end{proof}

As an Corollary of Theorem~\ref{thm:dimension Whittaker vectors} and Theorem~\ref{thm:dimension Whittaker vectors, algebraic}, we have the following theorem announced by T. Oshima.
Define $\widetilde{\rho_M}\in\mathfrak{h}^*$ by $\widetilde{\rho_M} = (1/2)\sum_{\alpha\in\Delta_M^+}\alpha$.
\begin{thm}\label{thm:Whittaker vectors, finite-dimensional case}
Assume that $\sigma$ is the irreducible finite-dimensional representation with highest weight $\widetilde{\nu}$.
Let $\dim_M(\lambda+\widetilde{\nu})$ be a dimension of the finite-dimensional irreducible representation of $M_0A_0$ with highest weight $\lambda+\widetilde{\nu}$.
\begin{enumerate}
\item Assume that for all $w\in W$ such that $\eta|_{wN_0w^{-1}\cap N_0} = 1$ the following two conditions hold:
\begin{enumerate}
\item For all $\alpha\in \Sigma^+\setminus w^{-1}(\Sigma^+_M\cup\Sigma_\eta^+)$ we have $2\langle\alpha,\lambda+w_0\widetilde{\nu}\rangle/\lvert\alpha\rvert^2\not\in\Z_{\le 0}$.
\item For all $\widetilde{w}\in\widetilde{W}$ we have $\lambda - \widetilde{w}(\lambda + \widetilde{\nu} + \widetilde{\rho_M})|_\mathfrak{a}\notin \Z_{\le 0}((\Sigma^+\setminus \Sigma_M^+)\cap w^{-1}\Sigma^+)|_\mathfrak{a}\setminus\{0\}$.
\end{enumerate}
Then we have
\[
	\dim \Wh_\eta(I(\sigma,\lambda)') = \# W(\supp\eta,\Theta)\times(\dim_M(\lambda+\widetilde{\nu}))
\]
\item Assume that for all $\widetilde{w}\in\widetilde{W}\setminus\widetilde{W_M}$, $(\lambda+\widetilde{\nu}) - \widetilde{w}(\lambda+\widetilde{\nu}) \not\in\Delta$. 
Then we have
\begin{multline*}
	\dim\Wh_\eta((I(\sigma,\lambda)_{\text{\normalfont $K$-finite}})^*)\\ = \# W(\supp\eta,\Theta)\times \#W_{\supp\eta}\times(\dim_M(\lambda+\widetilde{\nu}))
\end{multline*}
\end{enumerate}
\end{thm}

\appendix
\section{$C^\infty$-function with values in Fr\'echet space}\label{sec:C^infty-function with values in Frechet space}
\subsection{$\mathcal{L}$-distributions and tempered $\mathcal{L}$-distributions}
Let $M$ be a $C^\infty$-manifold, $V$ a Fr\'echet space and $\mathcal{L}$ a vector bundle on $M$ with fibers $V$.
We define the sheaf of $\mathcal{L}$-distributions as follows.

First we assume that $\mathcal{L}$ is trivial on $M$.
Then the definition of $\mathcal{L}$-distributions is found in Kolk-Varadarajan~\cite{MR1621372} and it is easy to see that $\mathcal{L}$-distributions makes a sheaf on $M$.

In general, let $M = \bigcup_{\lambda\in\Lambda}U_\lambda$ be an open covering of $M$ such that on each $U_\lambda$ the vector bundle $\mathcal{L}$ is trivial.
For an arbitrary open subset $U$ of $M$, put 
\[
\mathcal{D}'(U,\mathcal{L}) = \left\{(x_\lambda)\in \prod_{\lambda\in\Lambda}\mathcal{D}'(U\cap U_\lambda,\mathcal{L})\Bigm| \text{$x_\lambda = x_{\lambda'}$ on $U_\lambda\cap U_{\lambda'}$}\right\}.
\newsym{$\mathcal{D}'(U,\mathcal{L})$}
\]
It is independent of the choice of an open covering $\{U_\lambda\}$ and defines the sheaf of $\mathcal{L}$-distributions on $M$.

Now assume that $M$ has a compactification $X$, i.e., $M$ is an open dense subset of a compact manifold $X$.
In this case, we define a subspace $\mathcal{T}(M,\mathcal{L})$ of $\mathcal{D}'(M,\mathcal{L})$ by 
\[
\mathcal{T}(M,\mathcal{L}) = \{x\in \mathcal{D}'(M,\mathcal{L})\mid \text{$x = z|_{M}$ for some $z\in \mathcal{D}'(X,\mathcal{L})$}\}.\newsym{$\mathcal{T}(M,\mathcal{L})$}
\]
An element of $\mathcal{T}(M.\mathcal{L})$ is called a tempered $\mathcal{L}$-distribution.

For a subset $M_0\subset M$, put $\mathcal{D}'_{M_0}(U,\mathcal{L}) = \{x\in \mathcal{D}'(U,\mathcal{L})\mid \supp x \subset M_0\}$ and $\mathcal{T}_{M_0}(M,\mathcal{L}) = \{x\in \mathcal{T}(M,\mathcal{L})\mid \supp x\subset M_0\}$.
Assume that $M_0$ is a closed submanifold of $M$.
Then dualizing the restriction map $C_c^\infty(M,\mathcal{L})\to C_c^\infty(M_0,\mathcal{L})$, we have an injective map $\mathcal{D}'(M_0,\mathcal{L})\to \mathcal{D}'_{M_0}(M,\mathcal{L})$.
This map also implies $\mathcal{T}(M_0,\mathcal{L})\to \mathcal{T}_{M_0}(M,\mathcal{L})$.
Using these maps, we regard $\mathcal{D}'(M_0,\mathcal{L})$ and $\mathcal{T}(M_0,\mathcal{L})$ as a subspace of $\mathcal{D}'_{M_0}(M,\mathcal{L})$ and $\mathcal{T}_{M_0}(M,\mathcal{L})$, respectively.

\subsection{$\mathcal{L}$-distributions with support in a subspace}
Let $M$ be a Euclidean space $\R^n = \{(x_1,\dots,x_n)\in\R^n\}$ and $M_0$ a subspace $\R^{n - m}$ of $M$ defined by the equation $x_1 = \dots = x_m = 0$.
Assume that $M$ has a compactification $X$.
Let $E_1,\dots,E_m$ be vector fields on $M$ such that:
\begin{enumerate}
\item for all $\varphi\in C^\infty(M)$ we have $(E_i\varphi)|_{M_0} = (\frac{\partial}{\partial x_i}\varphi)|_{M_0}$.
\item A space $\sum_{i = 1}^m \C E_i$ is a Lie algebra.
\end{enumerate}

Set $D_i = \frac{\partial}{\partial x_i}$.
The condition (1) implies that $D_iT = E_iT$ for all $T\in \mathcal{D}'(M_0,\mathcal{L})$.
Put $U_n(E_1,\dots,E_m) = \sum_{k_1 + \dots + k_m\le n}\C E_1^{k_1}\dotsm E_l^{k_l}$ and $U(E_1,\dots,E_m) = \sum_n U_n(E_1,\dots,E_m)$.
Then the algebra $U(E_1,\dots,E_m)$ is isomorphic to the universal enveloping algebra of $\sum_{i = 1}^m \C E_i$.
For $\alpha = (\alpha_1,\dots,\alpha_m)$, put $E^\alpha = E_1^{\alpha_1}\dotsm E_m^{\alpha_m}$ where $E_i^0 = 1$.

\begin{lem}\label{lem:diff op nanka do-demoii}
Let $E_1',\dots,E_m'$ be vector fields on $M$ which satisfy the same conditions of $E_1,\dots,E_m$.
We have 
\[
E^\alpha T \in (E')^\alpha T + U_{\lvert\alpha\rvert -1}(E_1',\dots,E_m')\mathcal{T}(M_0,\mathcal{L})
\]
for $T\in\mathcal{T}(M_0,\mathcal{L})$ and $\alpha\in\Z_{\ge 0}^m$.
\end{lem}
\begin{proof}
First we remark that if the order of differential operator $P$ is less than or equal to $k$, then we have $P(\mathcal{T}(M_0,\mathcal{L}))\subset U_k(D_1,\dots,D_m)\mathcal{T}(M_0,\mathcal{L})$.
Take $P\in U_{k - 1}(E_1,\dots,E_m)$.
Then we have 
\begin{multline*}
E_iPT = [E_i,P]T + PE_iT = [E_i,P]T + PD_iT = [E_i - D_i,P]T + D_iPT\\ \in D_iPT + U_{k - 1}(D_1,\dots,D_m)\mathcal{T}(M_0,\mathcal{L})
\end{multline*}
since the order of $[E_i - D_i,P]$ is less than or equal to $k$.
Hence, using induction on $\lvert\alpha\rvert$, we have $E^\alpha T \in D^\alpha T + U_{\lvert\alpha\rvert - 1}(D_1,\dots,D_m)\mathcal{T}(M_0,\mathcal{L})$.

Hence we have $U_k(E_1,\dots,E_m)\mathcal{T}(M_0,\mathcal{L}) \subset U_k(D_1,\dots,D_m)\mathcal{T}(M_0,\mathcal{L})$.
We prove $U_k(E_1,\dots,E_m)\mathcal{T}(M_0,\mathcal{L}) = U_k(D_1,\dots,D_m)\mathcal{T}(M_0,\mathcal{L})$ by induction on $k$.
If $k = 0$ then the claim is obvious.
Assume that $k > 0$ then the above equation and inductive hypothesis imply that if $\lvert\alpha\rvert = k$, then we have $D^\alpha T \in E^\alpha T + U_{k - 1}(E_1,\dots,E_m)\mathcal{T}(M_0,\mathcal{L})$.
Hence we have $U_k(D_1,\dots,D_m)\mathcal{T}(M_0,\mathcal{L})\subset U_k(E_1,\dots,E_m)\mathcal{T}(M_0,\mathcal{L})$.

The same formulas hold for $E_1',\dots,E_m'$.
Hence, we have 
\begin{multline*}
E^\alpha T\in
D^\alpha T + U_{\lvert\alpha\rvert - 1}(D_1,\dots,D_m)\mathcal{T}(M_0,\mathcal{L})\\
= (E')^\alpha T + U_{\lvert\alpha\rvert - 1}(D_1,\dots,D_m)\mathcal{T}(M_0,\mathcal{L})\\
= (E')^\alpha T + U_{\lvert\alpha\rvert - 1}(E_1',\dots,E_m')\mathcal{T}(M_0,\mathcal{L}).
\end{multline*}
\end{proof}

\begin{prop}\label{prop:structure theorem of tempered distributions whose support is contained in submanifold}
A map $\Phi\colon U(E_1,\dots,E_m)\otimes \mathcal{T}(M_0,\mathcal{L})\to\mathcal{T}_{M_0}(M,\mathcal{L})$ defined by $P\otimes T \mapsto PT$ is isomorphic.
\end{prop}
\begin{proof}
First we prove that $\Phi$ is injective.
Let $\sum_{\alpha\in\Z_{\ge 0}^m}E^\alpha\otimes T_\alpha$ (finite sum) be an element of $U(E_1,\dots,E_m)\otimes \mathcal{T}(M_0,\mathcal{L}|_{M_0})$.
Set $T = \sum_{\alpha\in\Z_{\ge 0}^m}E^\alpha T_\alpha$ and assume that $T = 0$.
Put $k = \max\{\lvert\alpha\rvert\mid T_\alpha \ne 0\}$.
We prove that $k = -\infty$.
Assume that $k\ge 0$.
By Lemma~\ref{lem:diff op nanka do-demoii}, if $\lvert\alpha\rvert = k$ then $E^\alpha T_\alpha \in D^\alpha T_\alpha + U_{k - 1}(D_1,\dots,D_m)\mathcal{T}(M_0,\mathcal{L})$.
There exist $T'_\alpha$ such that $\sum_{\alpha\in\Z_{\ge 0}^m}E^\alpha T_\alpha = \sum_{\alpha\in\Z_{\ge 0}^m,\ \lvert\alpha\rvert < k}D^\alpha T'_\alpha + \sum_{\alpha\in\Z_{\ge 0}^m,\ \lvert\alpha\rvert = k}D^\alpha T_\alpha$.
Fix $\beta \in \Z_{\ge 0}^m$ such that $\lvert\beta\rvert = k$ and $f\in C^\infty(M_0)$ with values in $\mathcal{L}$.
Define a function $\varphi$ on $M$ by $\varphi(x_1,\dots,x_n) = x_1^{\beta_1}\dotsm x_m^{\beta_m}f(0,\dots,0,x_{m + 1},\dots,x_n)$.
Then we have $0 = \langle T,\varphi\rangle = \beta_1!\dotsm \beta_m! \langle T_\beta,f\rangle$.
Since $f$ is arbitrary, we have $T_\beta = 0$ for all $\beta$ such that $\lvert\beta\rvert = k$.
This is a contradiction.

We prove that $\Phi$ is surjective.
Take an open covering $X = \bigcup_{\lambda\in \Lambda}U_\lambda$ such that $U_\lambda$ is isomorphic to a Euclidean space and $\overline{M_0}\cap U_\lambda$ is a subspace of $U_\lambda$ where $\overline{M_0}$ is a closure of $M_0$ in $X$.
Since $X$ is compact, we may assume this is a finite covering.
It is sufficient to prove that the map $\Phi$ is surjective on $U_\lambda\cap M$.
However the surjectivity of $\Phi$ on $U_\lambda\cap M$ follows from \cite[(2.8)]{MR1621372}.
\end{proof}

\subsection{Distributions on a nilpotent Lie group}\label{subsec:Distributions on a nilpotent Lie group}
Let $N$ be a connected, simply connected nilpotent Lie group.
Put $\mathfrak{n} = \Lie(N)_\C$.
Then the exponential map $\exp\colon \Lie(N)\to N$ is a diffeomorphism.
A structure of a vector space on $N$ is defined by the exponential map.
Let $\mathcal{P}(N)$ be a ring of polynomials with respect to this vector space structure (cf.\ Corwin and Greenleaf~\cite[\S1.2]{MR1070979}).

Let $\mathcal{L}$ be a vector bundle on $N$ whose fiber is $V$ and assume that $\mathcal{L}$ is trivial on $N$, i.e., $\mathcal{L} = N\times V$.
Fix a Haar measure $dn$ on $N$.
For $F\in C^\infty(N,V')$, we define a distribution $F\delta$ by $\langle F\delta,\varphi\rangle = \int_N F(n)(\varphi(n))dn$.
Then we regard $C^\infty(N,V')$ as a subspace of $\mathcal{D}'(N,\mathcal{L})$.
Let $\mathcal{P}_k(N)$ be the space of polynomials whose degree is less than or equal to $k$.
Put $\mathcal{P}(N) = \bigoplus_k\mathcal{P}_k(N)$.

Take a character $\eta$ of $\mathfrak{n}$.
Then $\eta$ can be extended to the $\C$-algebra homomorphism $U(\mathfrak{n})\to \C$ where $U(\mathfrak{n})$ is the universal enveloping algebra of $\mathfrak{n}$.
We denote this $\C$-algebra homomorphism by the same letter $\eta$.
Let $\Ker\eta$ be the kernel of the $\C$-algebra homomorphism $\eta$.
For $X\in \mathfrak{n}$ and $C^\infty$-function $\psi$, put $(X\psi)(n) = \frac{d}{dt}\psi(\exp(-tX)n)|_{t = 0}$.

The algebraic tensor product $C_c^\infty(N)\otimes V$ is canonically identified with a linear subspace of $C^\infty_c(N,\mathcal{L})$ via $\varphi\otimes v\mapsto (x\mapsto \varphi(x)v)$.
As in \cite[(2.1)]{MR1621372}, this is a dense subspace of $C^\infty_c(N,\mathcal{L})$.

\begin{prop}\label{prop:killed by some power of n is polynomial}
For all $k\in \Z_{>0}$, there exists a positive integer $l$ such that, if $T\in \mathcal{D}'(N,\mathcal{L})$ satisfies $(\Ker\eta)^kT = 0$ then $T\in (\mathcal{P}_l(N)\otimes V')\delta$.

Conversely, for all $l\in\Z_{>0}$ there exists a positive integer $k$ such that $(\Ker\eta)^k(\mathcal{P}_l(N)\otimes V')\delta = 0$.
\end{prop}
\begin{proof}
Fix a basis $\{e_1,\dots,e_n\}$ of $\Lie(N)$.
The map $\R^n\to N$ defined by $(x_1,\dots,x_n)\mapsto \exp(x_1e_1 + \dots + x_me_m)$ is an isomorphism.
Using this map, we introduce a coordinate $(x_1,\dots,x_n)$ of $N$.
If $V = \C$ then this proposition is well-known.
Fix $v\in V$ and consider an ordinal distribution $T_v\colon \varphi\mapsto \langle T,\varphi\otimes v\rangle$ for $\varphi\in C^\infty_c(N)$.
If $T$ satisfies $(\Ker\eta)^kT = 0$, then $T_v$ satisfies $(\Ker\eta)^kT_v = 0$.
Hence for some $l$, $T_v = \sum_{\alpha_1 + \dots + \alpha_n\le l}(x_1^{\alpha_1}\dotsm x_n^{\alpha_n}\otimes c_{v,\alpha_1,\dots,\alpha_n})\delta$, where $c_{v,\alpha_1,\dots,\alpha_n}\in \C$.
The map $v\mapsto c_{v,\alpha_1,\dots,\alpha_n}$ is continuous linear.
Hence it defines an element of $V'$.
We denote this element by $v'_{\alpha_1,\dots,\alpha_n}$.
Then for $\varphi\in C_c^\infty(N)$ and $v\in V$ we have $\langle T,\varphi\otimes v\rangle = \langle (\sum_{\alpha_1 + \dots + \alpha_n\le l}x_1^{\alpha_1}\dotsm x_n^{\alpha_n}\otimes v'_{\alpha_1,\dots,\alpha_n})\delta,\varphi\otimes v\rangle$.
Since $C^\infty_c(N)\otimes V$ is dense in $C^\infty_c(N,\mathcal{L})$, we have $T = (\sum_{\alpha_1 + \dots + \alpha_n\le l}x_1^{\alpha_1}\dotsm x_n^{\alpha_n}\otimes v'_{\alpha_1,\dots,\alpha_n})\delta$.

We prove the second part of the proposition.
For $X\in \mathfrak{n}$, $f\in\mathcal{P}_l(N)$ and $v'\in V'$, we have $X((f\otimes v')\delta) = ((Xf)\otimes v')\delta$.
Hence we may assume that $V = \C$.
In this case, the claim is well-known.
\end{proof}

\begin{cor}\label{cor:polynomial by some power of n}
Let $T\in \mathcal{D}'(N,\mathcal{L})$.
Assume that there exists a positive integer $k$ such that $(\Ker\eta)^kT \in (\mathcal{P}(N)\otimes V')\delta$.
Then we have $T \in (\mathcal{P}(N)\otimes V')\delta$.
\end{cor}
\begin{proof}
By the second part of Proposition~\ref{prop:killed by some power of n is polynomial}, there exists a positive integer $k'$ such that $(\Ker\eta)^kT = 0$.
Hence we have $T \in (\mathcal{P}(N)\otimes V')\delta$ by the first part of Proposition~\ref{prop:killed by some power of n is polynomial}.
\end{proof}

%\bibliographystyle{../my_amsalpha}
%\bibliography{../../bunken,../../my_preprint,../../book,../../non_ams}
\def\cprime{$'$} \def\dbar{\leavevmode\hbox to 0pt{\hskip.2ex
  \accent"16\hss}d}\newcommand{\noop}[1]{}

\end{document}